%% file: ver_sottomessa.tex
\documentclass{article}

\usepackage[utf8]{inputenc}
\usepackage{lipsum}
\usepackage[dvipsnames]{xcolor}
\usepackage{url}
\usepackage{bm}
\usepackage{mathpazo}
\usepackage[T1]{fontenc}
\usepackage[utf8]{inputenc}
\usepackage[english]{babel}
\usepackage[active]{srcltx}
\usepackage{colorspace}
\usepackage{tikz}
\usepackage{sansmath}
\usetikzlibrary{shadings,intersections,patterns}
\usepackage{pgfplots}

\pgfplotsset{compat=1.10}
\usepgfplotslibrary{fillbetween}
\usetikzlibrary{patterns}
\usetikzlibrary{intersections, calc, angles}
\usepackage{tkz-euclide}
\pgfdeclarelayer{pre main}
\usepackage{environ}
\makeatletter
\newsavebox{\measure@tikzpicture}
\NewEnviron{scaletikzpicturetowidth}[1]{%
  \def\tikz@width{#1}%
  \begin{lrbox}{\measure@tikzpicture}%
  \BODY
  \end{lrbox}%
  \pgfmathparse{#1/\wd\measure@tikzpicture}%
  \BODY
}
\makeatother
\usetikzlibrary{arrows,automata}
\usetikzlibrary{positioning}
\usetikzlibrary{calc,through}
\usepackage{etoolbox}%
\apptocmd{\thebibliography}{\fontsize{11}{15}\selectfont}{}{}%

\tikzset{
    state/.style={
           rectangle,
           rounded corners,
           draw=black, very thick,
           minimum height=2em,
           inner sep=2pt,
           text centered,
           },
}

\usepackage{mathrsfs}

\usepackage[left=2cm, right=2cm, top=3cm, bottom=3cm]{geometry}
\usepackage{amsmath}
\usepackage{amssymb}
\usepackage{amsthm}
\usepackage{graphicx}
\usepackage{subcaption}
\usepackage{caption}
\usepackage{xcolor}
\usepackage{longtable}
\usepackage{verbatimbox}
\usepackage{xfrac}

\usepackage[shortlabels]{enumitem}

\usepackage{hyperref} 
\usepackage[capitalise, nameinlink, noabbrev]{cleveref} 
\hypersetup{colorlinks=true, 
	linkcolor=blue,
	citecolor=red,
}

\usepackage[title]{appendix}
\usepackage{bm}
\usepackage{empheq}
\usepackage{placeins}
\theoremstyle{plain}
\newtheorem{theorem}{Theorem}[section]
\newtheorem{lemma}[theorem]{Lemma}
\newtheorem{proposition}[theorem]{Proposition}

\newtheorem{remark}[theorem]{Remark}
\newtheorem{definition}[theorem]{Definition}
\theoremstyle{definition}
\theoremstyle{remark}
\numberwithin{equation}{section}

\newcommand{\B}{\mathcal B}
\newcommand{\E}{\mathcal{E}}
\newcommand{\F}{\mathcal F}

\renewcommand{\H}{\mathcal{H}}

\newcommand{\Sc}{\mathcal S}

\newcommand{\Z}{\mathbb{Z}}

\newcommand{\de}{\delta}
\newcommand{\e}{\varepsilon}

\renewcommand{\l}{\lambda}

\newcommand{\Div}{{\rm div}\,}

\newcommand{\weakstar}{\stackrel{*}{\rightharpoonup}}

\newcommand{\pa}{\partial}

\newcommand{\cl}{\mathrm{cl}\,}

\newcommand{\C}{\mathcal{C}}

\newcommand{\KK}{\mathcal{K}}

\newcommand{\bp}{\boldsymbol{\pi}_0}
\theoremstyle{plain}
\newtheorem*{theorem*}{Theorem}
\newtheorem*{corollary*}{Corollary}

\theoremstyle{definition}

\newtheorem*{notation*}{Notation}

\numberwithin{equation}{section}
\numberwithin{figure}{section}

\definecolor{myred}{rgb}{0.9,0,0}

\definecolor{vargreen}{rgb}{0.0, 0.5, 0.0}

\newcommand{\Om}{\Omega}

\makeatletter
\newcommand{\mylabel}[2]{#2\def\@currentlabel{#2}\label{#1}}
\makeatother

\newcommand{\FF}{\mathcal F_{\Lambda,\eps}}
\newcommand{\EL}{E_{\Lambda,\eps}}

\renewcommand{\rho}{\varrho}
\renewcommand{\theta}{\vartheta}

\newcommand{\eps}{\varepsilon}
\newcommand{\boundellipse}[3]
{(#1) ellipse (#2 and #3)
}

\usepackage[small, sc]{titlesec}

\usepackage[
backend=bibtex,
style=numeric,
sorting=nty,
maxbibnames=50,
giveninits=true
]{biblatex}
\usepackage{csquotes}
\renewbibmacro{in:}{%
	\ifentrytype{article}{}{}}
\DeclareFieldFormat[article]{volume}{\textbf{#1}}
\DeclareFieldFormat[article]{title}{\textit{#1}}
\DeclareFieldFormat[article]{journaltitle}{#1}
\DeclareFieldFormat[article]{number}{\textbf{#1}}
\renewbibmacro*{volume+number+eid}{%
	\printfield{volume}%
	\setunit*{\textbf{\addcolon}}
  \printfield{number}\setunit{\addcomma\space}%
  \printfield{eid}}

\usepackage{xcolor}
\hypersetup{
colorlinks,
linkcolor={red!50!black},
citecolor={blue!50!black},
urlcolor={blue!80!black}
}
\input{defs}

\begin{document}
\title{\textsc{Classical solutions to the soap film \\
capillarity problem for plane boundaries}}

\author{\textsc{Giulia Bevilacqua}$^1$\thanks{\href{mailto:giulia.bevilacqua@dm.unipi.it}{\texttt{giulia.bevilacqua@dm.unipi.it}}}\,\,\,$-$\,\, \textsc{Salvatore Stuvard}$^2$\thanks{\href{mailto:salvatore.stuvard@unimi.it}{
\texttt{salvatore.stuvard@unimi.it}}} \,\,\,$-$\,\, \textsc{Bozhidar Velichkov}$^1$\thanks{\href{mailto:bozhidar.velichkov@unipi.it}{
\texttt{bozhidar.velichkov@unipi.it}}}\bigskip\\
\normalsize$^1$ Dipartimento di Matematica, Università di Pisa, Largo Bruno Pontecorvo 5, I–56127 Pisa, Italy\\
\normalsize$^2$ Dipartimento di Matematica, Università degli Studi di Milano, Via Saldini 50, I-20133 Milano, Italy\\
}

\date{}

\maketitle

\bigskip

\begin{abstract}
\noindent We study the {\it soap film capillarity problem}, in which soap films are modeled as sets of least perimeter among those having prescribed (small) volume and satisfying a topological spanning condition. When the given boundary is the closed tubular neighborhood in $\mathbb{R}^3$ of a smooth Jordan curve (or, more generally, the closed tubular neighborhood in $\mathbb{R}^d$ of a smooth embedding of $\mathbb{S}^{d-2}$ in a hyperplane), we prove existence and uniqueness of \emph{classical} minimizers, for which the collapsing phenomenon does not occur. We show that the boundary of the unique minimizer is the union of two symmetric smooth normal graphs over a portion of the plane; the graphs have positive constant mean curvature bounded linearly in terms of the volume parameter, and meet the boundary of the tubular neighbourhood orthogonally. Moreover, we prove uniform bounds on the sectional curvatures in order to show that the boundaries of solutions corresponding to varying volumes are ordered monotonically and produce a foliation of space by constant mean curvature hypersurfaces.
\end{abstract}

\bigskip
\bigskip

\textbf{Mathematics Subject Classification (2020)}: 49Q05, 49S05, 49J45

\textbf{Keywords}: Soap film capillarity problem, classical solutions, curvature estimates, foliations

\bigskip
\bigskip

\maketitle

\tableofcontents

\section{Introduction}

Finding the shape of a soap film spanning an assigned wire is a centuries-old mathematical problem which dates back to Plateau \cite{plateau1873experimental}. 
Numerous techniques and theories have been developed in the literature with the aim of addressing and studying it. We refer for instance to the contributions of Douglas-Rad\'o \cite{rado1930problem, douglas1931solution}, Reifenberg \cite{reifenberg1960solution}, Federer-Fleming \cite{federer1960normal}, and Almgren \cite{almgren1968existence}; see also the recent survey by David \cite{david2014should}.
 The starting points in the present paper are the Harrison-Pugh's formulation \cite{harrison2016existence} of the so-called \emph{homotopic Plateau's problem}, then revised by De Lellis, Ghiraldin, and Maggi in \cite{de2017direct}, and the {\it soap film capillarity problem} proposed by Maggi, Scardicchio, and the second-named author in \cite{maggi2019soap}. First, we briefly introduce the state of the art in \cref{sub:intro-plateau} and \cref{sub:intro-beautiful}, then in \cref{sub:intro-main-theorem} we state our main results (\cref{th:sogno_finale} and \cref{th:calotte_circonferenza}).

\subsection{Homotopic Plateau problem}\label{sub:intro-plateau}

In the setting of \cite{de2017direct}, the wire $W$ is a closed subset of $\R^d$, while the admissible region to host the solution is defined as the open set $\Omega:= \R^d\setminus W$. 
A family $\mathscr{C}_{W}$ of curves is introduced as follows
$$
\mathscr{C}_{W}:=\left\{ \gamma : {\mathbb{S}}^1 \to \Om : \gamma \hbox{ is a smooth embedding} \right\},
$$
and a {\it spanning class} is any subset $\C \subset \mathscr{C}_{W}$ which is closed by homotopy. Namely, $\mathcal C$ is a spanning class (for $W$) if it has the following property: if $\gamma_0 \in \mathcal{C}$ and $\gamma_1 \in \mathscr{C}_W$, and there exists a (continuous) homotopy connecting $\gamma_0$ to $\gamma_1$ in $\Omega$, then $\gamma_1 \in \mathcal C$. A relatively closed set $S \subset \Om$ is said to be \emph{$\mathcal C$-spanning $W$} if  
\begin{equation*}
S \cap \gamma(\mathbb{S}^1) \neq \emptyset \quad \text{for every}\quad \gamma \in \C\,.
\end{equation*}
The \emph{homotopic Plateau problem} associated to the wire $W$ and to the spanning class $\mathcal C$ is then the variational problem 
\begin{equation}\label{e:hp-intro}
\inf\Big\{\H^{d-1}(S)\ :\ S\in\Sc(W,\C)\Big\}\,,
\end{equation}
where 
\begin{equation*}
\Sc(W,\C):=\left\{ S\subset \Omega\, \colon \, S \hbox{ is relatively closed in } \Omega \hbox{ and } S \hbox{ is } \C-\hbox{spanning } W \right\}\,.
\end{equation*}
In \cite[Theorem 4]{de2017direct}, it is proved that if the infimum in \eqref{e:hp-intro} is finite then there exists a minimizer, and that any minimizer is an $(\mathbf M, 0, \infty)$-minimal set in $\Omega$ in the sense of Almgren; see \cite{Alm_Memoirs} and \cite{Taylor76}. 
We remark that this framework is very flexible from the point of view of mathematical modeling: 
see, for instance, \cite{giusteri2017solution, bevilacqua2019soap, bevilacqua2020dimensional}, where the wire is a thick ``elastic rod'', and where the energy appearing in the minimization problem takes into account the competition between the elasticity of the rod and the area of the soap film.

\subsection{Soap film capillarity problem}\label{sub:intro-beautiful}

The focus of this paper is on a minimization problem proposed by Maggi, Scardicchio, and the second-named author in \cite{maggi2019soap}, where soap films are modeled as open sets in $\R^d$ enclosing a prescribed, small volume, and whose boundary satisfies the homotopic spanning condition defined above. The introduction of the volume constraint is justified by the fact that the soap films observed in nature, in the physical dimension $d=3$, are three-dimensional structures, with small but non-zero thickness, containing a given amount of fluid. This point of view introduces a length scale in Plateau's problem, which is essential in order to explain and characterize the bubble and film burst phenomenon; see the discussion in  \cite[Section 2]{maggi2019soap}.
Let us explain the model more in detail, following the rigorous definition provided in \cite{king2022plateau}. Let the ``boundary wire'' be a given, suitably regular compact set $W \subset \R^{d}$, let $\mathcal C \subset \mathscr{C}_W$ be a spanning class for $W$, and let $\varepsilon > 0$. A competitor for the soap film capillarity problem at volume $\varepsilon$ is then an open subset $E \subset \Omega$ of the admissible region $\Omega := \R^d \setminus W$ so that ${\rm Vol}(E) = \varepsilon$ and $\Om \cap \pa E$ is $\mathcal C$-spanning $W$. The surface tension energy of such a competitor is 
$
\sigma \, \mathcal H^{d-1} (\Om \cap \pa E)\,,
$
where $\sigma > 0$ is a constant, the \emph{surface tension parameter}.
In what follows, we shall always set $\sigma=1$, for simplicity.  
We can then define the minimization problem
\begin{equation}\label{e:soap-film-cap-pb-intro}
\inf\Big\{\H^{d-1}(\Omega\cap\pa E)\ :\ E\in\mathcal E(W,\mathcal C),\, {\rm Vol}(E)=\e\Big\}\,,
\end{equation}
where 
\begin{equation*}
\E(W,\C):=\left\lbrace E\subset \Omega\,\colon\,\mbox{$E$ is open, $\Omega\cap\pa E\in\Sc(W,\C)$, and $\pa E$ is smooth in $\Omega$}\right\rbrace\,.
\end{equation*}
Notice that different choices of $\mathcal C$ may result in different minimizers of \eqref{e:soap-film-cap-pb-intro}.

\medskip

The existence theory for \eqref{e:soap-film-cap-pb-intro} was first studied by King, Maggi, and the second-named author in \cite{king2022plateau}. There, it is observed that, in general, a minimizing sequence $E_j$ for \eqref{e:soap-film-cap-pb-intro} may exhibit the following behavior: in some portion of space the sets $E_j$ may become thinner and thinner, with boundaries $\partial E_j$ consisting of two disjoint smooth surfaces that converge simultaneously to a portion of a single minimal surface (with multiplicity two). This phenomenon is referred to as {\it collapsing} in \cite{king2022plateau}. As a consequence of this possibility, \cite{king2022plateau} establishes for \eqref{e:soap-film-cap-pb-intro} the existence of \emph{generalized minimizers}: these are pairs $(K,E)$, obtained as a suitable, measure-theoretic limit of (a subsequence of) the $E_j$, where $E$ (the wetted region) is an open set with ${\rm Vol}(E) = \varepsilon$, $K$ contains $\Om \cap \pa E$ and is $\mathcal C$-spanning $W$, and the infimum in \eqref{e:soap-film-cap-pb-intro} is realized as the relaxed energy
\[
\mathcal H^{d-1}(\pa E \cap \Om) + 2\, \mathcal H^{d-1}(K \setminus \pa E)\,.
\]
For general choices of $W$, collapsing (explicitly, the condition that $K \setminus \pa E \neq \emptyset$) is unavoidable; see the discussions in  \cite{king2021collapsing,king2022plateau,MaggiNovackRestrepo}.
On the other hand, in the physical situation in which $W$ is a tubular neighborhood of a {\em nice curve} (for instance a circle in $\R^3$),  it is natural to expect that \eqref{e:soap-film-cap-pb-intro} admits a classical solution.
It is indeed conjectured in \cite[Example 1.1]{king2022plateau} that the non-existence of singular solutions to the homotopic Plateau problem \eqref{e:hp-intro} should be sufficient to exclude collapsing; 
however, in the literature thus far there were no examples of classes of wires $W$ for which non-collapsing (and thus existence of classical solutions to \eqref{e:soap-film-cap-pb-intro}) could be guaranteed. 

Motivated by the above considerations, in this paper we study the soap film capillarity problem under suitable geometric assumptions on the wire $W$ and the spanning class $\C$. In \cref{th:sogno_finale}, we prove an existence and regularity result for \eqref{e:soap-film-cap-pb-intro} for ``planar'' wires, which gives a positive answer to the above non-collapsing conjecture in this geometry.

\subsection{Main result}\label{sub:intro-main-theorem}
Throughout this paper we always assume that $W$ and $\mathcal C$ have the following properties:
 \begin{enumerate} 
 \item[(W)] $W$ is the $\delta$-tubular neighborhood 
\begin{equation}
    \label{Wcond}
    W=I_\delta(\Gamma):=\Big\{x\in \R^{d}\ :\ \text{\rm dist}(x,\Gamma)\le\delta \Big\}\,,
\end{equation}
 where $\Gamma$ is a smooth embedding of the $(d-2)$-dimensional sphere $\mathbb S^{d-2}$ in the hyperplane 
 $$\bp:=\{x_{d}=0\}=\R^{d-1}\times\{0\},$$ 
 and $\delta>0$ is so small that $\partial I_\delta(\Gamma)$ is a smoothly embedded submanifold of $\R^{d}$;
\item[(C)] the spanning class $\mathcal C$ is the set of all smooth loops $\gamma\in \mathscr{C}_{W}$, which are non-trivial elements of the fundamental group $\pi_1 (\Omega)$, where $\Omega = \R^d \setminus W$.
\end{enumerate}
\begin{figure}[htbp]
    \centering
   \begin{tikzpicture}[rotate=0, scale= 0.75]
\draw[thick, NavyBlue, very thick, name path=a] (0,3) arc [start angle=90, end angle = 141,x radius = 45mm, y radius = 45mm];
\draw[thick, NavyBlue, very thick, name path=aa] (0,3) arc [start angle=90, end angle = 39,x radius = 45mm, y radius = 45mm];
\draw[thick, NavyBlue, very thick, name path=b] (0,-3) arc [start angle=-90, end angle = -39,x radius = 45mm, y radius = 45mm];
\draw[thick, NavyBlue, very thick, name path=bb] (0,-3) arc [start angle=-90, end angle = -141,x radius = 45mm, y radius = 45mm];
\draw[thick, BlueViolet, very thick, name path=c] (0,2.2) arc [start angle=90, end angle = 141,x radius = 45mm, y radius = 45mm];
\draw[thick, BlueViolet, very thick, name path=cc] (0,2.2) arc [start angle=90, end angle = 39,x radius = 45mm, y radius = 45mm];
\draw[thick, BlueViolet, very thick, name path=d] (0,-2.2) arc [start angle=-90, end angle = -141,x radius = 45mm, y radius = 45mm];
\draw[thick, BlueViolet, very thick, name path=dd] (0,-2.2) arc [start angle=-90, end angle = -39,x radius = 45mm, y radius = 45mm];
\draw[thick, MidnightBlue, very thick, name path=e] (0,1.4) arc [start angle=90, end angle = 141,x radius = 45mm, y radius = 45mm];
\draw[thick, MidnightBlue, very thick, name path=ee] (0,1.4) arc [start angle=90, end angle = 39,x radius = 45mm, y radius = 45mm];
\draw[thick, MidnightBlue, very thick, name path=f] (0,-1.4) arc [start angle=-90, end angle = -141,x radius = 45mm, y radius = 45mm];
\draw[thick, MidnightBlue, very thick, name path=ff] (0,-1.4) arc [start angle=-90, end angle = -39,x radius = 45mm, y radius = 45mm];
\begin{scope}[transparency group]
\tikzfillbetween[of=a and c] {color=Cyan!20};
\end{scope}
\begin{scope}[transparency group]
\tikzfillbetween[of=aa and cc] {color=Cyan!20};
\end{scope}
\begin{scope}[transparency group]
\tikzfillbetween[of=bb and d] {color=Cyan!20};
\end{scope}
\begin{scope}[transparency group]
\tikzfillbetween[of=b and dd] {color=Cyan!20};
\end{scope}
\begin{scope}[transparency group]
\tikzfillbetween[of=c and e] {color=SkyBlue!20};
\end{scope}
\begin{scope}[transparency group]
\tikzfillbetween[of=cc and ee] {color=SkyBlue!20};
\end{scope}
\begin{scope}[transparency group]
\tikzfillbetween[of=d and f] {color=SkyBlue!20};
\end{scope}
\begin{scope}[transparency group]
\tikzfillbetween[of=dd and ff] {color=SkyBlue!20};
\end{scope}
\begin{scope}[transparency group]
\tikzfillbetween[of=e and f] {color=Emerald!20};
\end{scope}
\begin{scope}[transparency group]
\tikzfillbetween[of=ee and ff] {color=Emerald!20};
\end{scope}
\draw[thick, NavyBlue, very thick] (0,3) arc [start angle=90, end angle = 141,x radius = 45mm, y radius = 45mm];
\draw[thick, NavyBlue, very thick] (0,3) arc [start angle=90, end angle = 39,x radius = 45mm, y radius = 45mm];
\draw[thick, NavyBlue, very thick] (0,-3) arc [start angle=-90, end angle = -39,x radius = 45mm, y radius = 45mm];
\draw[thick, NavyBlue, very thick] (0,-3) arc [start angle=-90, end angle = -141,x radius = 45mm, y radius = 45mm];
\draw[thick, BlueViolet, very thick] (0,2.2) arc [start angle=90, end angle = 141,x radius = 45mm, y radius = 45mm];
\draw[thick, BlueViolet, very thick] (0,2.2) arc [start angle=90, end angle = 39,x radius = 45mm, y radius = 45mm];
\draw[thick, BlueViolet, very thick] (0,-2.2) arc [start angle=-90, end angle = -141,x radius = 45mm, y radius = 45mm];
\draw[thick, BlueViolet, very thick] (0,-2.2) arc [start angle=-90, end angle = -39,x radius = 45mm, y radius = 45mm];
\draw[thick, MidnightBlue, very thick] (0,1.4) arc [start angle=90, end angle = 141,x radius = 45mm, y radius = 45mm];
\draw[thick, MidnightBlue, very thick] (0,1.4) arc [start angle=90, end angle = 39,x radius = 45mm, y radius = 45mm];
\draw[thick, MidnightBlue, very thick] (0,-1.4) arc [start angle=-90, end angle = -141,x radius = 45mm, y radius = 45mm];
\draw[thick, MidnightBlue, very thick] (0,-1.4) arc [start angle=-90, end angle = -39,x radius = 45mm, y radius = 45mm];
\draw[draw=none, fill=white] (5,0) circle [radius = 2cm];
\draw[draw=none, fill=white] (-5,0) circle [radius = 2cm];
\draw[very thick, color=black, name path=a] (-5,0) circle [radius=2];
\draw[very thick, color=black, name path=b] (5,0) circle [radius=2];
\tikzfillbetween[of=a and b] {color=gray!30};
\draw [color = black] node at (6,-1) {$W$};
\draw [color = black] node at (-6,-1) {$W$};
\draw[very thick, densely dashed, color=black,->] (-8,0) -- (8,0);
\draw[very thick, densely dashed, color=black,->] (0,-3.6) -- (0,3.6);
\draw [color = NavyBlue, very thick] node at (0.5,2.6) {\large{\bm{$ \varepsilon_3$}}};
\draw [color = NavyBlue] node at (-0.5,-2.6) {\large{\bm{$ \varepsilon_3$}}};
\draw [color = BlueViolet] node at (0.5,1.8) {\large{\bm{$ \varepsilon_2$}}};
\draw [color = BlueViolet] node at (-0.5,.-1.8) {\large{\bm{$ \varepsilon_2$}}};
\draw [color = MidnightBlue] node at (0.5,0.8) {\large{\bm{$ \varepsilon_1$}}};
\draw [color = MidnightBlue] node at (-0.5,-0.8) {\large{\bm{$ \varepsilon_1$}}};
\end{tikzpicture}
    \caption{Graphical sectional representation of a specific case in the physically relevant case $d = 3$, see \cref{th:calotte_circonferenza}. The solutions are normal graphs over a portion of the plane $\{x_3 = 0\}$ foliating $\overline{\Omega}$, thus there exists a unique solution for a fixed amount of volume $\eps$.}
    \label{fig:foliazione}
\end{figure}
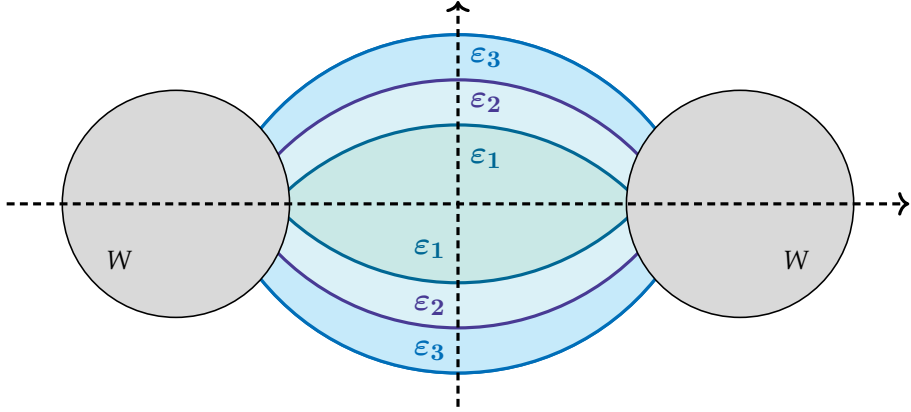 
For this choice of $W$ and $\mathcal C$, we prove that:
\begin{enumerate}[\rm (i)]
\item for every sufficiently small $\eps>0$, there is a unique solution $E^\eps$ to the soap film capillarity problem \eqref{e:soap-film-cap-pb-intro};
\item the boundary $\overline\Omega\cap\partial E^\eps$ is the union of two disjoint $C^\infty$ submanifolds with boundary having positive constant mean curvature, which are also smooth graphs over (a portion of) the hyperplane $\bp$;
\item the surfaces $\overline\Omega\cap\partial E^\eps$ are varying continuously (and smoothly) with respect to the parameter $\eps$ and, as $\eps\to0$, both components of $\overline\Omega\cap\partial E^\eps$ are approaching the plane $\bp$ in a $C^\infty$ way;
\item in a neighborhood of the hyperplane $\bp$, the family of surfaces $\big\{\overline\Omega\cap\partial E^\eps\big\}_{\eps\ge 0}$ are foliating $\overline\Omega$.
\end{enumerate}
These results are rigorously stated in our main theorem, which follows.

\begin{theorem}
\label{th:sogno_finale}
There exists $\delta_0 > 0$ with the following property. For every $\delta \in \left( 0, \delta_0\right]$, there exists $\e_0=\e_0(\de)>0$ so that, for $W=I_\de(\Gamma)$, $\Omega=\R^{d}\setminus W$ and $\mathcal C$ as above, and for every $\varepsilon \in \left( 0, \e_0\right]$, the soap film capillarity problem \eqref{e:soap-film-cap-pb-intro} admits a unique minimizer $E^\eps$, and the latter can be characterized as follows. There exist a closed domain $D^\e \subset \R^{d-1}$ with smooth boundary and a  function $u^\eps\colon \overline D^\eps\to \R$, smooth up to $\partial D^\eps$ and with $u^\eps(x') > 0$ for all $x' \in \overline D^\eps$ such that 
\begin{equation*}
    E^\eps = \Om \cap \left\{x' + t\, e_{d}\,:\, x'\in D^\eps\,,\, -u^\eps(x')< t< u^\eps(x') \right\}.
\end{equation*}
Furthermore, there exists $\lambda_\e\in(0, \Pi\eps)$, where $\Pi>0$ is a dimensional constant, such that the function $u^\eps$ satisfies
$$
  \left\{
  \begin{aligned}
  &-{\rm div}\, \left(\frac{\nabla u^\eps}{\sqrt{1 + \abs{\nabla u^\eps}^2}}\right) = \lambda_\e &&\hbox{ in }\ D^\eps\,,\\
  &\left( -\nabla u^\eps , 1 \right) \cdot \nu_W = 0 &&\hbox{ on }\ \partial D^\eps\,,
  \end{aligned}
  \right.
  $$
where $\nu_W$ is the outer unit normal vector field to $\pa W$. Moreover, as functions of $\e$, the sets $D^\eps$ and $E^\eps$ are (strictly) increasing (with respect to inclusion) as well as the functions $u^\eps$.
\end{theorem}

\begin{remark}\label{rema:D0}
In \cref{th:sogno_finale}, by hypothesis, $W$ is a planar wire, that is,  the smooth embedding $\Gamma$ of $\mathbb{S}^{d-2}$ is contained in the hyperplane $\bp$. Therefore, the set $\bp\cap \Omega$ has exactly two connected components, of which one is bounded and the other one is unbounded. Hence, the homotopic Plateau problem \eqref{e:hp-intro} has a unique solution, given by the bounded component $D_0$ of $\bp\cap \Omega$ (see \cref{app:Plateau_unica_sol}).
\end{remark}

\begin{remark}
    Notice that, as a consequence of the monotonicity properties stated in \cref{th:sogno_finale}, the set $D_0$ and the family of graphs $\{{\rm graph} (u^\eps)\}_{\e \in \left( 0, \e_0 \right]}$ form (in a neighborhood of $D_0$) a foliation of $\overline\Omega$ by constant mean curvature hypersurfaces.
\end{remark}

The prototypical case one may think of is when $d=3$ and $\Gamma$ is a circle in $\bp$. In this case, one gets a much more precise description of the unique minimizer. Indeed, we can show that the graph of the function $u^\eps$ is a spherical cap of radius $\sfrac{2}{\lambda_\eps}$ meeting $\pa W$ orthogonally and such that the corresponding set $E^\varepsilon$ contains volume $\eps$; a sectional representation is depicted in \cref{fig:foliazione}. We give the precise statement in the next theorem, and we remark that it extends to any dimension $d\ge 3$ with the same proof. 

\begin{theorem}
\label{th:calotte_circonferenza}
 Let $d=3$ and $\Gamma = \pa B_1 \cap \{x_3 = 0\}$. There exists $\delta_0 > 0$ with the following property. For every $\delta \in \left( 0, \delta_0\right]$, there exists $\e_0=\e_0(\de)>0$ so that, for $W=I_\de(\Gamma)$, $\Omega=\R^{3}\setminus W$ and $\mathcal C$ as above, and for every $\varepsilon \in \left( 0, \e_0\right]$, the soap film capillarity problem \eqref{e:soap-film-cap-pb-intro} admits a unique minimizer $E^\eps$ whose boundary in $\Omega$ is the union of two spherical caps $SC_{\lambda_\eps}$. More precisely,
 $$
 E^\varepsilon=\mathcal B_{\lambda_\eps}=\Omega \cap B_{\sfrac{2}\lambda_\eps}(0,0,z_C) \cap B_{\sfrac{2}{\lambda_\eps}}(0,0,-z_C),
 $$
 where $\lambda_\eps >0$ is the curvature and $z_C$ is the position of the center of the sphere on the axis $x_3$ so that $\mathcal B_{\lambda_\eps}$ contains volume $\eps$. Moreover, as a function of $\eps \in (0, \varepsilon_0]$, the sets $\mathcal B_{\lambda_\eps}$ are strictly increasing with respect to inclusion, and their boundaries $\Omega \cap SC_{\lambda_\eps}$ are disjoint and, together with $D_0$, foliate a portion of $\overline\Omega$.
\end{theorem}

\paragraph{Main difficulties}

The key steps and the major difficulties in the proof of \cref{th:sogno_finale} are the following:
\begin{itemize}
\item {\it non-collapsing:} proving that the boundary of $E$ in $\Omega$ is composed of two disjoint $C^{1,\alpha}$ graphs; 
\item {\it foliation:} proving that the sets $E$ are increasing with respect to the volume $\eps$ and that their constant mean curvature boundaries form a foliation of the space.
\end{itemize}

The {\it non-collapsing} property relies on a uniform (with respect to the volume $\eps$) upper bound on the generalized mean curvatures $\lambda$, which allows to prove Hausdorff and varifold convergence of minimizers to the disc as $\eps\to0$; we explain our strategy in detail in \cref{sub:strategy_intro} below.

The {\it foliation} property on the other hand is a consequence of uniform (in terms of the volume $\eps$ and the thickness $\delta$ of the tubular neighborhood) bounds of sectional curvatures of the surface $\pa E \cap \Omega$, which we obtain through an hodograph-type transform. Foliations of the space with non-trivial global area-minimizing surfaces are known to exist only in few occasions; for instance such foliations have been previously constructed in the seminal work of Hardt and Simon \cite{simon1985area} (see also \cite{de2022inhomogeneous} for the free boundary counterpart) around singular cones. 
The construction from \cite{simon1985area} is based on a linearization argument, which cannot be applied to our case due to the volume constraint and the boundary conditions on the curved wire $W$. Thus, we build the foliation adopting a different technique based on the above curvature bounds and on viscosity arguments, which are specific for ``planar'' wires, but do not require further geometric or regularity assumptions on the wire.

Finally, we stress that the proof of \cref{th:calotte_circonferenza} is quite different with respect to the general case. Indeed, once we have the uniform upper bound on the mean curvature, in $d=3$ and when $\Gamma \subset \bp$ is a circle, we can construct an explicit foliation of the space with spherical caps of mean curvature $\lambda_\varepsilon$ containing volume $\eps$. Precisely, the thesis is a consequence of a sliding viscosity argument, where competitors are chosen to be suitable spherical caps, and the application of the maximum principle (see \cref{th:circonferenza3D}). We discuss the case $d=3$ for simplicity, but the same proof applies in any dimensions $d \geq 3$.

\subsection{Strategy of the proof}\label{sub:strategy_intro}

In order to prove that minimizers of \eqref{e:soap-film-cap-pb-intro} exist, it is natural to consider a minimizing sequence for the variational problem \eqref{e:soap-film-cap-pb-intro} of (smooth) sets $E_j\in \mathcal{E}(W,\C)$ with ${\rm Vol}(E_j) = \varepsilon$. Since $E_j$ is a sequence of sets of uniformly bounded perimeter, by the classical theory of the sets of finite perimeter, we have that, up to a subsequence, $E_j$ converges to a Lebesgue measurable set $E\subset \Omega$ such that 
$$
{\rm Vol}(E)=\lim_{j\to\infty}{\rm Vol}(E_j)=\eps\qquad\text{and}\qquad \text{\rm Per}(E;\Omega)\le \liminf_{j\to\infty}\text{\rm Per}(E_j;\Omega),$$
where $\text{\rm Per}(E;\Omega):=\H^{d-1}(\partial^\ast E\cap\Omega)$ is the relative perimeter of $E$ (in $\Omega$) and $\partial^\ast E$ is the reduced boundary of $E$. Now, the fundamental difficulty with respect to the classical isoperimetric problem is that it is not immediate to guarantee that the limit set $E$ satisfies the spanning condition:
if in some ball $B_r(x_0)\subset \Omega$ the sets $E_j$ are getting thinner and thinner, the limit $E$ could be empty in $B_r(x_0)$ so some of the loops in $\mathcal C$ might not be crossing $\partial E$.
In \cite{king2022plateau}, this problem was overcome by taking the limits (of suitable modifications) of both the sets $E_j$ and their boundaries $\Omega \cap \partial E_j$ in a measure-theoretic sense.
Thus, one obtains a pair $(K,E)$ of sets such that:  
$E$ is an open set of finite perimeter which has measure ${\rm Vol}(E) = \eps$ and satisfies $\Omega \cap {\rm cl}(\pa^*E) = \Omega \cap \pa E$, while $K$ is a relatively compact, $\H^{d-1}$-rectifiable, $\mathcal C$-spanning set in $\Omega$ and $\Omega \cap \partial E\subset K$. 

By defining the relaxed energy of the pair $(K,E)$ as
\begin{equation}
\label{e:intro_funzionaleF}
    \F(K,E):= \H^{d-1}(\Omega\cap\pa^* E)+ 2\,\H^{d-1}(K\setminus\pa^* E)\,,
\end{equation}
by \cite{king2022plateau}, it holds 
$$
\F(K,E)=\lim_{j\to\infty}\text{\rm Per}(E_j;\Omega)=\inf\Big\{\H^{d-1}(\Omega\cap\pa E)\ :\ E\in\mathcal E(W,\mathcal C),\, {\rm Vol}(E)=\e\Big\}\,.
$$
Notice that the relaxed energy consists of the sum of the perimeter of the \emph{wetted region} $E$ in $\Om$ plus the area of the portion of the $\mathcal C$-spanning surface $K$ that lies outside of the reduced boundary of the wetted region \emph{counted with multiplicity $2$}, following the fact that the latter emerges from the collapsing of boundaries along a minimizing sequence.
We stress that the generalized minimizers $(K,E)$ provided by \cite{king2022plateau} are not known to be minimizers of the relaxed energy $\F$,  
but are only minimal with respect to smooth deformations preserving the volume.

To obtain a pair $(K,E)$ that minimizes $\mathcal F$ or a suitable modification thereof, one can apply the recent result of Maggi, Novack, and Restrepo \cite{MaggiNovackRestrepo}, which provides the existence of minimizers for the relaxed energy functional
\[
\widehat{\mathcal F}(K,E):=\mathcal H^{d-1}(\Omega \cap \pa^*E) + 2\, \mathcal H^{d-1}(K \cap E^{(0)}),
\]
 among all pairs $(K,E)$ of sets so that $K$ and $E$ are Borel subsets of $\Omega$, $E$ has locally finite perimeter in $\Omega$, $\mathcal{H}^{d-1}((\Omega \cap \pa^*E) \setminus K)=0$, ${\rm Vol}(E)=\varepsilon$, and $K\cup E^{(1)}$ is $\mathcal C$-spanning $W$. Here, $E^{(\theta)}$ is the set of points in $\R^d$ of Lebesgue density $\theta$ for $E$. Notice that for this process to work one needs to make sense of the $\mathcal C$-spanning condition for a merely \emph{Borel set}: this was achieved in \cite{MaggiNovackRestrepo} through the clever introduction of a measure-theoretic overhaul of the homotopic spanning condition, and we refer to their paper for more on this point.

 Thanks to \cite[Theorem 1.6]{MaggiNovackRestrepo}, minimizers of $\widehat{\mathcal F}$ can be taken so that $E$ is an open set with finite perimeter, $K\subset \Omega$ is relatively closed,  $(d-1)$-rectifiable, $\Om \cap {\rm cl}(\pa^* E) = \Om \cap \pa E \subset K$, and $K \cap E^{(1)}=\emptyset$.   In particular, this immediately implies that a minimizer $(K,E)$ of $\widehat{\mathcal F}$ is such that $K \cup E$ satisfies the $\mathcal C$-spanning condition in the classical sense, and $\widehat{\mathcal F}(K,E) = \mathcal F (K,E)$. Thus, \cite{MaggiNovackRestrepo} provides a minimizer $(K,E)$ of the variational problem
\begin{align}
\label{e:inf_F_intro}
\inf\Big\{\mathcal F(K,E)\ :\ K\cup E \in\mathcal S(W,\mathcal C),\, {\rm Vol}(E)=\e\Big\},
\end{align}
and we have that
\begin{align*}
\mathcal F(K,E)\le \inf\Big\{\H^{d-1}(\Omega\cap\pa E)\ :\ E\in\mathcal E(W,\mathcal C),\, {\rm Vol}(E)=\e\Big\}.
\end{align*}
Moreover, given a minimizer $(K,E)$, the integral varifold $V$ associated to the measure  
$$\mu:=\H^{d-1}\text{\Large$\llcorner$}(\partial^\ast E)+2\,\H^{d-1}\text{\Large$\llcorner$}(K\setminus\partial^\ast E),$$
has $L^\infty$-bounded generalized mean curvature vector (equal to $\lambda\, \nu_E$, with $\lambda \in \R$ a Lagrange multiplier associated with the volume constraint, on $\Omega \cap\partial^*E$ and $0$ elsewhere), and as $\eps\to0$, $\mu$ converges weakly* to the measure $\mu_0:=2\,\H^{d-1}\text{\Large$\llcorner$}(D_0)$, where $D_0$ is the unique solution to the homotopic Plateau problem. We emphasize that, in general, there are no uniform curvature bounds for $K$ and $E$ and the curvature $\lambda$ might be unbounded as $\eps\to0$.

\medskip

In \cref{sec:esistenza_min_generalizzati}, we introduce the auxiliary energy functional
$$
\widehat{\mathcal F}_{\Lambda,\eps}(K,E) = \H^{{d-1}}(\Omega \cap \partial^\ast E) + 2\H^{{d-1}}(K \cap E^{(0)}) + \Lambda \abs{{\rm Vol}(E) - \eps},
$$
where $\Lambda >0$ is a constant. Applying \cite{MaggiNovackRestrepo} and ensuring that the sequence of volumes ${\rm Vol}(E_j)$ of a minimizing sequence $(K_j,E_j)$ is not vanishing, in \cref{prop:esistenza_FF} we show that there is a minimizer $(K,E)$ of the problem
\begin{equation}
\label{e:inf_intro_FLambdaEps}
\inf\Big\{\widehat{\mathcal F}_{\Lambda,\eps}(K,E)\ :\ K\cup E^{(1)}\in\mathcal S(W,\mathcal C)\Big\}\,.
\end{equation}
A minimizer $(K,E)$ of \eqref{e:inf_intro_FLambdaEps} has the same properties as minimizers of \eqref{e:inf_F_intro}, and in particular one may just work with the functional
\[
\FF(K,E) = \H^{{d-1}}(\Omega \cap \partial^\ast E) + 2\H^{{d-1}}(K \setminus \pa^*E) + \Lambda \abs{{\rm Vol}(E) - \eps}\,,
\]
with the constraint that $K\cup E \in \mathcal S(W,\mathcal C)$, namely $K\cup E$ is $\mathcal{C}$-spanning $W$ in the classical sense of \cref{sub:intro-plateau}.
On the other hand, it has the advantage that the constant mean curvature $\lambda$ of $\pa^*E$ is controlled by the volume variation, that is the positive constant $\Lambda$, see \cref{lemma:stime_curvature}. Next, we aim to show that for any minimizer of $\FF\,$, volume of $E$ is precisely $\eps$. In fact, this will immediately lead to: 
\begin{itemize}
\item for small $\eps$ the two problems \eqref{e:inf_F_intro} and \eqref{e:inf_intro_FLambdaEps} are equivalent;
\item the generalized mean curvature of the solutions of \eqref{e:inf_F_intro} is uniformly bounded as $\eps\to0$;
\item the minimizers of \eqref{e:inf_F_intro} converge in the Hausdorff sense to $D_0$ as $\eps\to0$. 
\end{itemize}

We obtain ${\rm Vol}(E) = \eps$ as a consequence (see \cref{lemma:stime_curvature}) of the following upper and lower bounds on the generalized mean curvature $\lambda$
$$
0 < \lambda \leq \Pi \eps,
$$
where $\Pi>0$ is a positive dimensional constant, see \cref{sec:bounds_sulle_curvature}. The lower bound follows from a viscosity argument combined with the Schatzle's maximum principle \cite[Theorem 6.2]{schatzle2004quadratic}, where we use as competitors (smooth deformations of) planes parallel to $\bp$ coming from infinity and touching $K$ from outside.  
The upper bound is more involved as it requires the construction of competitors touching $\partial E$ from inside of $E$. 

To prove the upper bound, we proceed as follows. First, combining the weak* convergence of $(K,E)$ and the densities estimates, we show that for small $\eps$, $K\cup E$ lies in a (open) thin strip $\mathcal{T}$ containing the disc $D_0$. 
Then, we consider the only bounded connected component $\widetilde\Omega$ of the set $\Omega \cap\mathcal{T}$ containing $D_0$. By a careful analysis of the Lebesgue density of $\H^{d-1}$-almost every point $x_0 \in K$ and the relative isoperimetric inequality, see \cref{prop:lunga!!!}, we show that for small volumes the set $\widetilde\Omega\setminus(K \cup E)$ has exactly two connected components, which we denote by $S_+$ and $S_-$. 
Then, any minimizer $(K,E)$ of $\FF$ can be rewritten, up to a $\H^{d-1}$-negligible set, as
$$
K=
\widetilde\Omega \cap(\partial S_+\cup\partial S_-)\qquad\text{and}\qquad  E=\widetilde\Omega\setminus (\overline S_+\cup\overline S_-),
$$
getting that the energy results into
$$\FF(K,E)=\mathcal G(S_+,S_-),$$
where the functional $\mathcal G$ is defined on the couples of disjoint sets of finite perimeter $S_+,S_-$ as
$$\mathcal G(S_+,S_-):=\H^{d-1}(\widetilde\Omega\cap \partial^\ast S_+)+\H^{d-1}(\widetilde\Omega\cap \partial^\ast S_-)+\Lambda\,\left|{\rm Vol}\left(\widetilde\Omega\setminus (\overline S_+\cup\overline S_-)\right)-\eps\right|.
$$
Now, we apply again a viscosity sliding argument, the maximum principle, and the fact that $S_+$ and $S_-$ are almost minimizers for the perimeter. Precisely, using, as competitors, planes parallel to $\bp$ coming from below and touching $S_+$ from outside, we show that: $\overline S_+$ lies in the upper half-space, $\widetilde\Omega \cap \pa S_+$ has constant mean curvature $\lambda$ and it meets $\pa \Omega$ orthogonally (analogously, $\overline S_-$ lies in the lower half-space and it has the same properties), see \cref{lemma:pernonesseretroppolunghi}. 
Finally, by constructing a suitable family of large balls from outside of $S_+$, we prove the uniform upper bound on the generalized mean curvature $\lambda$ (see \cref{l:upper_bound_curvatura_eps}), which gives that ${\rm Vol}(E)=\eps$. 

\medskip

At this point, as a consequence of the results from \cref{sec:bounds_sulle_curvature} and Allard's epsilon-regularity theorem (see \cref{sec:regolarita}), for small $\eps$, we have obtained that any minimizer of $(K,E)$ of \eqref{e:inf_F_intro} has the following properties: 
\begin{itemize}
\item {\it regularity:} $\Omega \cap \partial E$ is a smooth surface with constant mean curvature, which attaches orthogonally to $\partial\Omega$;
\item {\it non-collapsing:} $K=\Omega \cap \partial E$ has exactly two disjoint connected components, namely 
$$ K_+:=\widetilde\Omega\cap\partial S_+\qquad\text{and}\qquad K_-:=\widetilde\Omega\cap\partial S_-\,,$$ 
lying respectively in $\{x_d>0\}$ and $\{x_d<0\}$.
\end{itemize}
Once the regularity property on $\Omega \cap \pa E$ is obtained, \cref{th:calotte_circonferenza} follows by applying again a viscosity argument. Precisely, by choosing as competitors spherical caps slid from above and from below and using a similar technique as the one adopted in \cref{lemma:pernonesseretroppolunghi}, we show that $K_+$ and $K_-$ coincide with $SC_{\lambda} \cap \Omega\cap\{x_3 >0\}$ and $SC_{\lambda} \cap\Omega\cap \{x_3 <0\}$ respectively (see \cref{sec:spherical_caps}), where $\sfrac{2}{\lambda}$ is the radius of the spherical cap.

\medskip

Finally, to conclude the proof of \cref{th:sogno_finale}, in \cref{sec:graph_ordered}, we construct the foliation of the space for general planar curves $\Gamma \subset \bp$ (\cref{th:ordered_graphs}) and we prove that the optimal sets are symmetric (\cref{t:symmetry}). We consider two different volumes $\eps_1$ and $\eps_2$ and the corresponding boundaries $K_+^{\eps_1}$ and $K_+^{\eps_2}$. In order to show that $K_+^{\eps_1}$ and $K_+^{\eps_2}$ are ordered with respect to the volume, we slide $K_+^{\eps_1}$ until it touches $K_+^{\eps_2}$ from one side at some point $x$. For internal touching point $x\in\Omega$, the maximum principle guarantees that the mean curvatures and the two boundaries are ordered (see \cref{th:ordered_graphs}). Thus, we only need to exclude the existence of contact points at the fixed boundary $\partial\Omega$.
To rule out this possibility, we show that the fixed boundary $\partial\Omega$ ``bends faster'' than the free boundary $K_+$, that is, the sectional curvatures of $K_+$ are bounded by a universal constant independent of the radius $\delta$ of the tubular neighborhood $W$. 
Precisely, we obtain these curvature bounds as a consequence of uniform $C^{2,\alpha}$ estimates on the functions $u^\eps_{\pm}$ up to the boundary. Thus, first we perform a change of variable to reparametrize the tubular neighbourhood $W$ and making its boundary flat, as 
$$
    W= \Big\{\big(\phi(s,t),x_d\big)\ :\ s\in B'',\ t\in (-\delta_0,\delta_0),\,\ x_d\in (-\delta_0,\delta_0), \,\ t^2+x_d^2\le \delta^2\Big\},
$$
where $B''$ is a ball in $\R^{d-2}$, $\phi(s,t) = \gamma(s) + t \nu_\gamma(s)$ is a suitable smooth parametrization of the tubular neighborhood constructed over $\Gamma$ and then we set $v(s,t) := u^\eps_+(\phi(s,t))$.
We introduce the following hodograph-type transform 
$$
v\Big(s,(\delta + \rho) \cos\Theta(s, \rho)\Big) =(\delta + \rho) \sin\Theta(s, \rho),
$$
where the function $\Theta: B''\times [0,\delta_0)\to\R$ exists by the $\C^{1,\alpha}$- regularity provided in \cref{sec:regolarita}. The geometric meaning of this transform is the following: we parametrize any point $P$ on the surface $K_+$ in terms of:
\begin{itemize}
    \item the projection $\gamma(s)$ of $P$ on the $(d-2)$-dimensional manifold $\Gamma$;
    \item the distance $\rho$ from $P$ to the tubular neighborhood $W$ of $\Gamma$;
    \item the angle $\Theta$ between $P-\gamma(s)$ and the plane $\bp$. 
\end{itemize}
This transform allows to parametrize the surfaces $K_\pm$ in terms of solutions to the following nonlinear elliptic PDE problem with Neumann boundary conditions 
\begin{equation*}
    \left\{
\begin{aligned}
    &{\rm div}\left(\frac{N}{\sqrt{1 + M}}{\big(\tens{I}+\tens{T}\tens{B}\tens{T}\big) \nabla \Theta} \right) = \frac{G}{(\delta+\rho)^2} \quad&&\hbox{ in } B''\times (0,\delta_0),\\
    & e_{d-1}\cdot (\tens{I}+\tens{T} \tens{B}\tens{T})\nabla \Theta= 0\quad &&\hbox{ on } B''\times\{0\},
\end{aligned}
\right.
\end{equation*}
where $M= M(s, \rho, \Theta, \nabla \Theta)$, $N=N(s, \rho, \Theta)$ are scalars, $\tens{T} = \tens{T}(\Theta)$ is a diagonal matrix and $\tens{B} = \tens{B}(s, (\delta + \rho)\cos\Theta)$ is a symmetric matrix that depends on the performed change of variables $\phi$. The key observation is that the right-hand side $G= G(s, \rho, \Theta, \nabla \Theta)$ can be bounded in terms of the $C^{1,\alpha}$ norm of $\Theta$, which in turn is controlled by the flatness of $K_+$ at scale $1$. Thus, by taking the volume $\eps$ small enough, we get that $G$ is uniformly bounded by a constant independent on $\delta$, so by the classical Schauder estimates this provides the desired uniform $C^{2,\alpha}$ estimate on $\Theta$ and $u_\pm^\eps$.

\section{A family of auxiliary variational problems}
\label{sec:esistenza_min_generalizzati}

Let the sets $W$, $\Omega$ and the spanning class $\mathcal C$ be as in (W) and (C). We say that a pair $(K,E)$ is {\it admissible} for $W$ and $\mathcal C$, and we write $(K, E) \in \mathcal A (W, \mathcal C)$ if: 
\begin{itemize}
\item[(i)] $K \subset \Om$ is relatively closed and $\mathcal H^{d-1}$-rectifiable in $\Omega$;
\item[(ii)] $E \subset \Om$ is open, with  finite perimeter in $\Om$, and with $\Om \cap {\rm cl}(\pa^* E) = \Om \cap \pa E \subset K$;
\item[(iii)] $K\cup E$ is $\mathcal C$-spanning $W$.  
\end{itemize}

\begin{remark}
We notice that the class $\mathcal A (W, \mathcal C)$ differs from the class $\KK (W, \mathcal C)$ from \cite{king2022plateau} in the requirement that the $\mathcal C$-spanning condition is imposed to the set $K \cup E$, rather than to $K$ alone. Moreover, for couples $(K,E)$ satisfying (i) and (ii) above, we have that $K\cup E$ is $\mathcal C$-spanning if and only if $K\cup E^{(1)}$ is $\mathcal C$-spanning. Indeed, any point $x \in E^{(1)}\Delta E$ must belong to $\pa E$, so $\Omega\cap \big(E^{(1)}\Delta E\big)\subset  \Omega\cap\pa E \subset K$. We observe that the spanning condition from \cite{MaggiNovackRestrepo} is formulated in measure theoretic sense and it is imposed on $K\cup E^{(1)}$; on the other hand, the regularity theorems from \cite{MaggiNovackRestrepo} allow us to work directly with couples of sets $(K,E)$ satisfying (i) and (ii) above, for which the spanning condition (iii) is equivalent and less technical to handle.
\end{remark}

\begin{proposition}\label{prop:esistenza_FF}
Let $W$ and $\mathcal C$ be as above. Then, for every $\Lambda>0$, there is $\eps_0>0$ such that for every $\eps<\eps_0$, there exists an admissible pair $(K,E)\in\mathcal A(W,\mathcal C)$ such that ${\rm Vol}(E)>0$ and the following holds.
\begin{enumerate}[\rm (i)]
\item  Setting 
\begin{align*}
    \widehat{\mathcal F}(K,E):=\mathcal H^{d-1}(\pa^*E \cap \Omega) + 2\, \mathcal H^{d-1}(K \cap E^{(0)})
    \end{align*}
    and
    \begin{align*}
    \widehat{\mathcal F}_{\Lambda,\eps}(K,E):=\widehat{\mathcal F}(K,E)+\Lambda\big|\text{\rm Vol}\,(E)-\eps\big|\,,
\end{align*}
we have that    
\begin{equation}\label{e:minimality-final}
        \widehat{\mathcal F}_{\Lambda,\eps}(K,E) \leq \inf\Big\{\mathcal H^{d-1}(\Omega\cap \partial^\ast E) +\Lambda|{\rm Vol}(E)-\eps|\ :\ E\in \mathcal E(W,\mathcal C)\Big\}\,.
    \end{equation}
    \item For all $\eps\leq \eps_0$ the following estimate holds
    \begin{equation}
        \label{e:stima_disco}
        \widehat{\mathcal F}_{\Lambda,\eps}(K,E)\le 2\H^{d-1}(D_0)+C_d\eps^2,
    \end{equation}
    where $D_0 \subset \bp$ is the bounded connected component of $\bp\setminus W$ and $C_d>0$ is a dimensional constant.
\item For every diffeomorphism $\Phi:\Omega\to\Omega$ with compact support in $\Omega$, it holds:
\begin{equation}\label{e:minimality-diffeo}
        \widehat{\mathcal F}_{\Lambda,\eps}(K,E) \leq \widehat{\mathcal F}_{\Lambda,\eps}(\Phi(K),\Phi(E))\,.
    \end{equation}
\end{enumerate}
\end{proposition}
\begin{proof}
Fix $W$, $\mathcal C$, and $\Lambda$, and let $\varepsilon$ be sufficiently small. Let $\{(K_j,E_j)\}_{j \geq 1}$ be a sequence in $\mathcal A (W,\mathcal C)$ such that 
$$\lim_{j\to\infty}\widehat{\mathcal F}_{\Lambda,\varepsilon}(K_j,E_j) =\inf\Big\{\widehat{\mathcal F}_{\Lambda, \varepsilon}(K,E) \, \colon \, (K,E)\in \mathcal A(W,\mathcal C)\Big\}\,.$$
Set $t_j := {\rm Vol}(E_j)$, and let $\bar t$ be any subsequential limit of $\{t_j\}$. We claim that it must be $\bar t > 0$. Define $S_{\eps}$ to be the open and bounded connected component of the set
\begin{equation*}
\Omega\cap\{-\sigma(\eps)<x_d<\sigma(\eps)\}\,,
\end{equation*}
which contains the disc $D_0$, the constant $\sigma(\eps)>0$ being chosen in such a way that ${\rm Vol}(S_\eps)=\eps$. Since $S_\e \in \mathcal E(W, \mathcal C)$ for $\e$ sufficiently small (depending on $W$), we can use the pair $(\pa S_\eps \cap \Omega , S_\eps)$ as a competitor for $\widehat{\mathcal F}_{\Lambda, \e}$ on $\mathcal A (W,\mathcal C)$.
Moreover, since $(K_j,E_j)$ is a minimizing sequence, up to extracting a subsequence, we have that for $j$ large enough
\begin{equation} \label{e:strip-upper}
 \widehat{\mathcal F}_{\Lambda, \eps}(K_j, E_j) = \widehat{\mathcal F}(K_j,E_j)+\Lambda|{\rm Vol}(E_j)-\eps|\le\H^{d-1}(\Omega\cap\partial^\ast S_\eps) \leq2\H^{d-1}(D_0)+C_d\eps^2\,,
\end{equation}
where $C_d$ is a dimensional constant. Assume by contradiction that $\bar t = 0$. By \cite[Theorem 1.5]{MaggiNovackRestrepo}, for every $t>0$ there exists a pair $(K_t, E_t) \in \mathcal{A}(W,\mathcal C)$ such that 
$$
\widehat{\mathcal F}(K_t, E_t)=\inf\Big\{\widehat{\mathcal F}(K,E)\, \colon \, (K,E)\in \mathcal A(W,\mathcal C)\,, \,  {\rm Vol}(E)=t\Big\} \,.
$$
Corresponding to the choice $t=t_j$, we get then pairs $(K_{t_j}, E_{t_j})$ for which $\widehat{\mathcal F}(K_{t_j}, E_{t_j}) \leq \widehat{\mathcal F}(K_j, E_j)$. Furthermore, we have that 
$$\H^{d-1} \text{\Large$\llcorner$} (\Om \cap \pa^*E_{t_j}) + 2\, \H^{d-1} \text{\Large$\llcorner$} (K_{t_j} \cap E_{t_j}^{(0)})\weakstar\,2\,\H^{d-1}\text{\Large$\llcorner$} D_0 \qquad \hbox{as}\quad j \to \infty\,,$$
and in fact it holds that
\begin{equation} \label{e:strip-lower}
\exists\,\lim_{j\to+\infty}\widehat{\mathcal F}(K_{t_j}, E_{t_j})+\Lambda|{\rm Vol}(E_{t_j})-\eps|=\lim_{j\to +\infty}\widehat{\mathcal F}(K_{t_j}, E_{t_j})+\Lambda|t_j-\eps|=2\,\H^{d-1}(D_0)+\Lambda\eps\,,
\end{equation}
which together with \eqref{e:strip-upper} implies 
\begin{equation} \label{contr_concl}
2\,\H^{d-1}(D_0)+\Lambda\eps\le 2\H^{d-1}(D_0)+C_d\eps^2\,.
\end{equation}
For $\eps$ small enough, \eqref{contr_concl} is violated, and this proves that indeed it is $\bar t > 0$.  

Now, applying again \cite[Theorem 1.5]{MaggiNovackRestrepo}, there is an admissible couple $(K,E)\in\mathcal A(W,\mathcal C)$ such that 
$${\rm Vol}(E)= \bar t$$
and such that 
\begin{equation*}
        \widehat{\mathcal F}_{\Lambda,\eps}(K,E) =\widehat{\mathcal F}(K,E)+\Lambda|{\rm Vol}(E)-\eps|\le \lim_{j\to+\infty}\widehat{\mathcal F}(K_j,E_j)+\Lambda|{\rm Vol}(E_j)-\eps|\,,
\end{equation*}
implying that \eqref{e:minimality-final}, \eqref{e:stima_disco} and \eqref{e:minimality-diffeo} hold.
\end{proof}

\begin{remark}
    We notice that the inequality in \eqref{e:minimality-final} is in fact an equality, as a consequence of \cite[Theorem 2.4]{Novack_2024}. Moreover, by the regularity result in \cite[Theorem 1.6]{MaggiNovackRestrepo}, the set $(K,E)$ of \cref{prop:esistenza_FF} may be chosen so that $K$ is the disjoint union of the sets $\Om \cap \pa^* E$, $\Omega \cap (\pa E \setminus \pa^* E)$, and $K \setminus \pa E$. In particular, this implies:
    \begin{equation} \label{e:relaxation_bd}
    \widehat{\mathcal F}(K,E) = \mathcal F (K,E) := \H^{d-1}(\Om \cap \pa^* E) + 2\, \H^{d-1}(K \setminus \pa^*E)\,.
    \end{equation}
    In what follows, we will then work with the functional $\mathcal F$ and with the corresponding penalized functional
    \begin{equation} \label{e:pen_relaxation_bd}
        \FF (K,E) := \F(K,E) + \Lambda \, |{\rm Vol}(E) - \varepsilon|\,.
    \end{equation}
\end{remark}

\begin{lemma}
\label{lemma:stime_curvature}
Let $\Lambda>0$, $\eps>0$ and $(K,E)$ be a minimizer of $\FF$ in the class $\mathcal{A}(W, \C)$. Then there exists $\lambda\in\R$ such that
  \begin{equation}
    \label{e:stationary-main}
    \l\,\int_{\pa^*E}X\cdot\nu_E\,d\H^{d-1}=\int_{\pa^*E}\Div^K\,X\,d\H^{d-1}+2\,\int_{K\setminus\pa^*E}\Div^K\,X\,d\H^{d-1}\,,
  \end{equation}
  for every $X\in C^1_c(\R^{d};\R^{d})$ with $X\cdot\nu_\Om=0$ on $\pa\Om$, where $\Div^K$ denotes tangential divergence along the rectifiable set $K$; $\nu_E$ is the outer unit normal to $\partial^\ast E$; $\nu_\Omega$ is the outer unit normal to $\partial\Omega$. 
Moreover:
\begin{enumerate}[\rm (1)]
\item if $|E|>\eps$, then $\lambda+\Lambda=0$;
\item if $|E|<\eps$, then $\lambda-\Lambda=0$;
\item if $|E|=\eps$, then $-\Lambda\le \lambda\le\Lambda$.
\end{enumerate}
\end{lemma}
\begin{proof}
The identity \eqref{e:stationary-main} follows from \cite[Theorem 1.5]{MaggiNovackRestrepo} or \cite[Theorem 1.6]{king2022plateau}. The claims (1) and (2) follow since $(K,E)$ minimizes $\mathcal F(K,E)+\Lambda{\rm Vol}(E)$ (in the case (1)) and $\mathcal F(K,E)-\Lambda{\rm Vol}(E)$ (in the case (2)). Lastly, (3) follows by comparing the variation of $\mathcal F$ in the direction of $X$, which is $\l\!\int_{\pa^*E}X\cdot\nu_E\,d\H^{d-1}$, with the volume variation of the term $\Lambda|{\rm Vol}(E)-\eps|$ in the direction $X$, which is at most $\Lambda\big|\!\int_{\pa^*E}X\cdot\nu_E\,d\H^{d-1}\big|$. 
\end{proof}

\noindent We introduce next the $(d-1)$-dimensional integral varifold $V:= {\bf var}(K, \vartheta)$ with multiplicity function $\theta$ given by
  \[
      \theta (x) =
      \begin{cases}
      1 & \mbox{if $x \in \Om \cap \pa^* E$}, \\
      2 & \mbox{if $x \in K \setminus \pa^* E$}\,.
  \end{cases}
  \]
With this definition, the first variation formula \eqref{e:stationary-main} can be written as
\begin{equation}
\label{stationary_varifold}
   \delta V (X) = \int_{\R^{d} \times {\rm G}(d-1,d)} \Div_{S}\,X\,dV(x, S) = \int_{\R^{d}} H \cdot X \,d \norm{V} 
\end{equation}
for all vector fields $X \in C^1_c(\R^{d}; \R^{d})$ with $X \cdot \nu_{\Omega} = 0$, where: 
\begin{itemize}
\item ${\rm G}(d-1,d)$ denotes the Grassmannian of (unoriented) $(d-1)$-dimensional linear subspaces in $\R^d$; 
\item $\|V\|:=\H^{d-1}\text{\Large$\llcorner$}(\partial^\ast E)+2\,\H^{d-1}\text{\Large$\llcorner$}(K\setminus\partial^\ast E)$; 
\item the \textit{generalized mean curvature vector} $H$ of $V$ is defined by
\begin{equation}
\label{eq:H-curvatura}
    H(x) = 
\begin{cases}
    \lambda\, \nu_{E}(x) &\mbox{if $x \in \Om \cap \pa^*E$}\,,\\
    0 & \mbox{if $x \in K \setminus \pa^*E$}\,,
\end{cases}
\end{equation}
where $\nu_E$ is the outer unit normal to $\partial^\ast E$. 
\end{itemize}

Hence, by \eqref{stationary_varifold} and \eqref{eq:H-curvatura}, $V$ is an integral varifold with $L^\infty$-bounded generalized mean curvature in $\Om$. Furthermore, following \cite{kagayatone}, $V$ has contact angle $\sfrac{\pi}{2}$ at the boundary ${\rm cl}(K) \cap \pa \Om$ in the sense of distributions, where ${\rm cl}(\cdot)$ denotes the closure in $\R^d$ of the considered set. The following lemma is an immediate consequence of this observation. 

\begin{lemma}\label{l:density-estimate}
Let $\Lambda>0$ and $\eps_0>0$ be as in \cref{prop:esistenza_FF}. Let $\eps\in(0,\eps_0)$ and $(K,E)$ be a minimizer of $\FF$ in the class $\mathcal{A}(W, \C)$. Then, there are $r_0>0$ and $C>0$, depending on $d$ the dimension, $W$ and $\Lambda$ such that
$$
\H^{d-1}(\pa^\ast E \cap B_r(x_0)) + 2 \H^{d-1}\left((K \setminus \pa^\ast E) \cap B_r(x_0)\right)\ge C r^{d-1},
$$
for every $r<r_0$ and every $x_0\in {\rm cl}(K)$. In particular, the set ${\rm cl} \,(K)\cup {\rm cl}\,(E)$ is bounded. 
\end{lemma}
\begin{proof}
By a standard argument, the first variation identity from \cref{lemma:stime_curvature} implies a monotonicity formula for $\H^{d-1}\text{\Large$\llcorner$}(\partial^\ast E)+2\H^{d-1}\text{\Large$\llcorner$}(K\setminus\partial^\ast E)$; see for instance \cite[Theorem 4.1]{king_arma}. Precisely, there exists $r_0 = r_0(W, d, \Lambda)$ such that $\forall\, r \leq r_0$ and for all $x_0 \in {\rm cl}(K)$, it holds 
\begin{equation}
    \label{eq:denisty_estimate_lemma}
    \frac{\H^{d-1}(\pa^\ast E \cap B_r(x_0)) + 2 \H^{d-1}\left((K \setminus \pa^\ast E) \cap B_r(x_0)\right)}{\omega_{d-1} r^{d-1}} \geq e^{-|\lambda| r}\geq e^{-\Lambda r} \geq C,
\end{equation}
where $C = C(\Lambda,r_0)$. As a consequence, since $\mathcal F(K,E)<+\infty$, there cannot be a sequence of points $x_n\to+\infty$ lying in $K$, so $K$ is bounded; in particular, since $\partial E\subset K$ and ${\rm Vol}(E)<+\infty$, also $E$ is bounded. 
\end{proof}

\section{Mean curvature estimates}
\label{sec:bounds_sulle_curvature}

In this section, we provide lower and upper bounds for the generalized mean curvature $\lambda$ of any minimizer $(K,E)$ of the energy functional $\FF$ defined in \eqref{e:pen_relaxation_bd}.

\subsection{Lower bound on the mean curvature \texorpdfstring{$\lambda$}{lambda}}
\label{sec:lambda_positive}

In this section, we prove that the mean curvature $\lambda$ of a minimizer $(K,E)$ to $\mathcal F_{\Lambda,\eps}$ has to be strictly positive (see \cref{thm:lambda>0}) and, as a consequence, that $|E|\leq\eps$. We will prove this by a comparison viscosity argument that relies on the fact that the fixed boundary $\pa W$ is positively curved along the normal direction to the solution $D_0 \subset \bp$ to the homotopic Plateau problem.

\medskip

We first recall the following definitions.

\begin{definition}
We define the {\em upper and lower height functions} $\phi_{+}$ and $\phi_{-}$ for $K$ to be the maps
$$
\begin{aligned}
 z \in \bp\, \mapsto \, &\phi^{+}(z) := \sup\left\{t: (z,t) \in K\right\},\\
 z \in \bp\,\mapsto\, &\phi^{-}(z) := \inf\left\{t: (z,t) \in K\right\},
\end{aligned}
$$
and we set $\phi_{+}(z) = -\infty$ and $\phi_{-}(z) = + \infty$ if $K \cap (\{z\} \times \R) = \emptyset$. 
\end{definition}

\begin{remark}
Since $K \cup E$ is $\mathcal C$-spanning $W$, every point $z \in D_0$ is such that ${\bf p}^{-1}(\{z\}) \cap (K \cup E) \neq \emptyset$ (here, ${\bf p}$ denotes orthogonal projection onto $\bp$). In fact, since $K \supset \Om \cap \pa E$, one immediately sees that ${\bf p}^{-1}(\{z\}) \cap K \neq \emptyset$ for every $z \in D_0$. In particular, ${\bf p}(K)$ contains $D_0$; notice that, by definition, for every $z\in {\bf p}(K)$ one has $-\infty < \phi^-(z) \leq \phi^+ (z) < \infty$.
\end{remark}

\begin{definition}\label{def:touching}
We say that the graph of a function $\zeta: \bp \to \R$ touches $K$ from above (resp. from below) at a point $x_0\in{\rm cl}(K)$ if there exists a neighborhood $U'$ in $\bp$ of the projection $z_0:={\bf p}(x_0)$ such that $\zeta(z_0) = \phi^+(z_0)$ (resp. $\zeta(z_0) = \phi^-(z_0)$) and $\zeta(z) \ge \phi^+(z)$ (resp. $\zeta(z) \le \phi^-(z)$) for every $z \in U'\cap {\bf p}(K)$. 
\end{definition}

Fundamental ingredients in the proofs of the following lemmas will be two comparison principles: one for touching points on the boundary ${\rm cl}(K) \cap \pa \Omega$ and another one for touching points in the interior $\cl (K) \cap \Omega=K$.

\begin{lemma}[Hopf lemma]
\label{lemma:con_bordo}
Let $(K,E)\in \mathcal{A}(W, \C)$ be a minimizer of $\FF$. 
Let $\zeta:\bp\to\R$ be a smooth function whose graph touches $K$ from above (in the sense of \cref{def:touching}) at a boundary point  
$\overline{x} \in \cl (K) \cap \partial \Omega$. Then, $\nu_\Omega(\overline{x})\cdot\nu_\zeta(\overline{x})\ge 0$, where $\nu_\Omega$ is the exterior normal to $\Omega$ at $\pa\Omega$ and $\nu_\zeta$ is the normal to the graph $\zeta$ pointing upwards.
\end{lemma}
\begin{proof}

Suppose by contradiction that $\nu_\Omega(\overline{x})\cdot\nu_\zeta(\overline{x})< 0$ at the boundary point  
$\overline{x} \in \cl (K) \cap \partial \Omega$. We perform a blow-up analysis at the touching point $\overline x$. Upon introducing the function 
\begin{equation*}
    \eta_{\overline x, r} (x):= \frac{x-\overline{x}}{r} \qquad \mbox{for every $r > 0$}\,,
\end{equation*}
we consider the translated and rescaled sets  $\Omega_{\overline{x}, r} := \eta_{\overline x, r}(\Omega)$, $K_{\overline{x}, r} := \eta_{\overline x, r} (K)$, $E_{\overline{x}, r} := \eta_{\overline x, r} (E)$, as well as the varifold $V_{\overline{x}, r} := (\eta_{\overline x, r})_\sharp V$. Then, for every vector field $Y\in C^1_c(\R^d;\R^d)$ with $Y \cdot \nu_{\Omega_{\overline x,r}}=0$, we have from \eqref{stationary_varifold} that 
\begin{equation} \label{e:first-variation-bu}
\int_{\R^d \times {\rm G}(d-1,d)} \Div_{S}\,Y(y)\,dV_{\overline{x}, r}(y, S) = r \int_{\R^d} Y(y) \cdot H(\overline{x} + r y)\,d\norm{V_{\overline{x}, r}},
\end{equation}
where $H$ is defined as in \eqref{eq:H-curvatura}. Notice also that, when $r \to 0$, the rescaled surfaces $\eta_{\overline x,r}(\pa\Omega)$ converge smoothly to the tangent plane $T_{\overline x} (\pa\Om) =: T_{\pa\Om}$, and the open sets $\Om_{\overline{x},r}$ converge to a half-space $\mathbb H$ with $\pa \mathbb H = T_{\pa\Om}$. Let now $\{r_k\}_{k=1}^\infty$ be a sequence with $r_k \to 0$ as $k \to \infty$. By the boundary monotonicity formula, there exists a (not relabeled) subsequence as well as a $(d-1)$-dimensional integral varifold cone $\overline V$ such that $V_k := V_{\overline{x},r_k}$ converge to $\overline V$ in the sense of varifolds. Furthermore, since the first variation is continuous with respect to varifolds convergence, and since $H$ is bounded in $L^\infty$, we deduce from \eqref{e:first-variation-bu} that 
\begin{equation}
    \label{eq:first_var_blowup}
    \int_{\R^d\times{\rm G}(d-1,d)} \Div_{S}\,Y(x)\,d\overline V(x, S)=0\,,
\end{equation}
for every vector field $Y \in C^1_c(\R^d;\R^d)$ such that $Y \cdot \nu_{T_{\pa\Om}}=0$, where we have denoted $\nu_{T_{\pa\Om}}$ the exterior unit normal to $\mathbb H$ at $T_{\pa\Om}$. In other words, $\overline V$ is stationary in $\mathbb H$ and has contact angle of $\sfrac{\pi}{2}$ at $T_{\pa\Om}$. Notice that $\nu_{T_{\pa\Om}}$ is the vector field constantly equal to $\nu_\Omega (\overline x)$. Furthermore, since the graph of $\zeta$ touches $K$ from above at $\overline{x}$, the support of $\overline V$ must be contained in the wedge $\mathbb W \subset \cl (\mathbb H)$ formed by $T_{\pa\Om}$ and the hyperplane $T_\zeta$ tangent to the graph of the function $\zeta$ at $\overline{x}$ (which has constant normal $\nu_\zeta(\overline{x})$). 

We claim that there is no such a $\overline V$ if $\nu_\Om(\overline{x}) \cdot \nu_\zeta (\overline{x}) < 0$. To see this, first observe that, by a dimension reduction argument, we may assume that the support of $\overline V$ meets $T_{\partial\Omega}\cap T_\zeta$ only at the origin. Up to a rotation, we may assume that $e_1\in T_{\partial\Omega}$ and $e_3, \dots\, e_d \in  \big(T_{\partial\Omega}\cap T_\zeta\big)$. Define 
$$
Y(x):= (\phi(x_1),0,0,\dots,0),
$$
where $\phi$ is a decreasing compactly supported function on $[0,\infty)$, i.e. $\frac{\pa \phi}{\pa x_1}<0$. We notice that $Y$ satisfies $Y \cdot \nu_{T_{\pa \Om}}=0$, and that, since the angle between $T_{\partial\Omega}$ and $T_\zeta$ is smaller than $\sfrac{\pi}{2}$, $Y$ can be modified into a compactly supported vector field in $\R^d$ without affecting the evaluation of the left-hand side of \eqref{eq:first_var_blowup}. Thus, by computing the tangential divergence in \eqref{eq:first_var_blowup} and upon denoting $\nu_{\overline V}(x)$ a unit vector field spanning the orthogonal complement to $S$ at $V$-a.e. $(x,S)$ (notice that $S$ is uniquely identified by $x$ because $\overline V$ is rectifiable), we get
$$
\begin{aligned}
    \diver_S Y(x) &= \diver_S \left((\phi(x_1), 0, \dots, 0)\right)\\
    &= \diver\left((\phi(x_1), 0, \dots, 0)\right) - \nu_{\overline V} \cdot \nabla \left((\phi(x_1), 0, \dots, 0)\right) \nu_{\overline{V}}\\
    &= \frac{\pa \phi}{\pa x_1}\left(1- (e_1 \cdot \nu_{\overline{V}})^2\right)\\
    &\leq 0\,.
\end{aligned}
$$
It follows from \eqref{eq:first_var_blowup} that $\nu_{\overline V}(x)=e_1$ for $\|\overline V\|$-almost every $x$ which in turn gives that $\overline V$ is supported on the orthogonal complement to $e_1$. In particular, $\overline V$ cannot lie in the wedge formed by $T_{\partial\Omega}$ and $T_\zeta$, and this is a contradiction.
\end{proof}

\begin{theorem}
    \label{thm:lambda>0}
Let $\Lambda>0$ and $\eps\in(0,\eps_0)$ with $\eps_0$ as in \cref{prop:esistenza_FF}. Let $(K,E)\in \mathcal{A}(W, \C)$ be a minimizer of $\FF$, and let $\lambda$ be as in Lemma \ref{lemma:stime_curvature}.  
Then $\lambda > 0$ and, by \cref{lemma:stime_curvature}, $|E|\le \eps$.
\end{theorem}

\begin{proof}
Suppose by contradiction that $\lambda\le 0$. 
For every $\tau \in \R$, we define $\zeta_\tau \colon \bp \to \R$ to be the function $\zeta_\tau (z) \equiv \tau$. Let $\tau^+$ to be the smallest $\tau$ for which $\zeta_\tau\ge \phi^+$ on ${\bf p}(K)$ and $\tau^-$ to be the largest $\tau$ for which $\zeta_\tau\le \phi^-$ on ${\bf p}(K)$. Notice that, since $K$ is bounded (see \cref{l:density-estimate}), 
$$-\infty < \tau^- \leq \tau^+ < +\infty\,.$$ 
We will prove that the assumption $\lambda\le 0$ entails that $\tau^+$ cannot be strictly positive (respectively that $\tau^-$ cannot be strictly negative); thus, the only remaining possibility is $\tau^+=\tau^-=0$, which implies that $K$ is contained in $\bp$, which is impossible as it violates the fact that ${\rm Vol}(E)>0$. We prove next that $\tau^+ \leq 0$ if $\lambda \leq 0$ (the result for $\tau^-$ being analogous). 

Arguing by contradiction, we suppose that $\tau^+>0$ and for simplicity we set  $\zeta_+:=\zeta_{\tau^+}$. Let $\overline x$ be a contact point of the graph of $\zeta_+$ with $\cl (K)$. We notice that by \cref{lemma:con_bordo}
$\overline x\notin \partial\Omega$; indeed, the normal to the graph of $\zeta_+$ pointing upwards is $e_d$, and $\nu_\Omega\cdot e_d < 0$ in the upper half-space $\{x_d>0\}$. So, the only remaining possibility is that $\overline{x} \in K$. For every $n\in\mathbb{N}_{\ge 1}$ and every $\tau\in\R$, let $\zeta_{\tau,n} \colon \bp\to \R$ be the function defined by
\[
\zeta_{\tau,n}(z) := \tau-\frac1{2n}|z-\overline x|^2\,.
\]
For fixed $n$, let $\tau_n^+$ be the smallest $\tau$ such that $\zeta_{\tau,n}\ge \phi^+$. By construction $\tau_n^+\ge\tau^+$. Moreover, due to \cref{lemma:con_bordo}, for $n$ large enough, there cannot be contact points between the graph of $\zeta_{\tau_n^+,n}$ and $\cl (K)$ on the boundary $\partial\Omega$. Thus, we can find $n$ such that the graph of $\zeta_{\tau_n^+,n}$ touches $K$ at a point $x_n\in\Omega$. Now, on one hand 
$$
-\diver\left(\frac{\nabla \zeta_{\tau_n^+,n}}{\sqrt{1 + \abs{\nabla \zeta_{\tau_n^+,n}}^2}}\right)(z)=\frac{dn^2+(d-1)|z-\overline x|^2}{\big(n^2+|z-\overline x|^2\big)^{3/2}}\ge\overline c,
$$
for some constant $\overline c>0$ depending only on $d$ and $n$. On the other hand, there exists a set $Z \subset \mathbf{p}(K) \subset \bp$ with $\mathcal{L}^{d-1}(Z)=0$ such that for every $z \in \mathbf{p} (K) \setminus Z$ one has
$$
H(z,\phi^+(z))=\left\{
\begin{aligned}
&\lambda\, \nu_E(z,\phi^+(z))&&\text{when}\quad (z,\phi^+(z))\in \partial^\ast E\\
&0&&\text{when}\quad (z,\phi^+(z))\in K\setminus \partial^\ast E\,.
\end{aligned}
\right.
$$
Since
\[
\nu_E(z, \phi^+ (z)) = \frac{(-\nabla\phi^+(z),1)}{\sqrt{1 + |\nabla\phi^+(z)|^2}} \qquad \mbox{whenever $(z,\phi^+ (z)) \in \pa^*E$}\,,
\]
we have that, for $z \in \mathbf{p} (K) \setminus Z$,
\[
H (z, \phi^+(z)) \cdot \left(\frac{\left(-\nabla \phi^{+}(z), 1\right)}{\sqrt{1 + \abs{\nabla \phi^{+}(z)}^2}}\right) = 
\begin{cases}
    \lambda & \mbox{when $(z,\phi^+(z)) \in \pa^*E$} \\
    0 & \mbox{when $(z, \phi^+ (z)) \in K \setminus \pa^*E$}\,.
\end{cases}
\]
In particular, if $\lambda \leq 0$, then
 \begin{equation} \label{ordering of ops}
    -\diver\left(\frac{\nabla \zeta_{\tau_n^+,n}}{\sqrt{1 + \abs{\nabla \zeta_{\tau_n^+,n}}^2}}\right) (z) > H (z, \phi^+(z)) \cdot \left(\frac{\left(-\nabla \phi^{+}(z), 1\right)}{\sqrt{1 + \abs{\nabla \phi^{+}(z)}^2}}\right),
    \end{equation}
for $\mathcal L^{d-1}$-a.e. $z \in \mathbf{p}(K)$. This is a contradiction to Sch\"atzle's strong maximum principle, see \cite[Theorem 6.2]{schatzle2004quadratic}, and concludes the proof.
\end{proof}

\subsection{Uniform upper bound on the mean curvature \texorpdfstring{$\lambda$}{lambda}}

First of all we show that if the volume is small enough, then any minimizer $(K,E)$ is contained in a strip. For any $\sigma > 0$ we set the notation
\[
\mathcal T_\sigma:=\{- \sigma <  x_d <  \sigma\}\,.
\]
\begin{definition}\label{definition:tilde}
We define $\widetilde\Omega$ to be the only bounded connected component of $\Om \cap \mathcal T_\delta$. The open set $\widetilde \Omega$ can also be characterized as the connected component of $\Omega$ containing $D_0$.    
\end{definition} 

\begin{lemma}
\label{lemma:Estanellastriscia}
For every $\Lambda>0$, $\de > 0$ and $\sigma>0$, there exists $\varepsilon_0 > 0$ with the following property. For $\varepsilon 
\leq \varepsilon_0$ if $(K,E)$ is a minimizer of $\FF$ then
\begin{equation}
    \label{eq:tuttostriscia}
    K \cup E\subset \Om \cap \mathcal T_{\delta \sigma}\,.
\end{equation}
\end{lemma}

\begin{proof}
For every $\eps$ small enough, let $(K_{\Lambda,\eps},E_{\Lambda,\eps})\in\mathcal A(W,\mathcal C)$ be a minimizer of $\mathcal F_{\Lambda,\eps}$ given by \eqref{e:pen_relaxation_bd}.
First of all, we notice that $\H^{d-1}\text{\Large$\llcorner$}(\partial^\ast E_{\Lambda,\eps}\cap \Omega) + 2 \H^{d-1}\text{\Large$\llcorner$}(K_{\Lambda,\eps} \setminus\pa^\ast E_{\Lambda,\eps})$ converges weakly* to $2\H^{d-1}\text{\Large$\llcorner$} D_0$ as $\eps \to 0$. Indeed, since $\Lambda$ is fixed and $0< |E_{\Lambda,\eps}|\le \eps$ (by \cref{thm:lambda>0}), the claim follows by \cite[Theorem 1.5]{MaggiNovackRestrepo} or \cite[Theorem 1.9]{king2022plateau}.
As a consequence, given $\delta >0$ and $\zeta>0$, there exists $\eps_0=\eps_0(\delta,\zeta)$ such that 
$$
\H^{d-1}(\partial^\ast\EL \cap\Omega\cap \mathcal T_{2 \delta \sigma}) + 2 \H^{d-1}(K_{\Lambda,\eps} \setminus\partial^\ast\EL\cap \mathcal T_{2 \delta \sigma}) \ge 2\H^{d-1}(D_0)-\zeta \qquad \forall\, \eps<\eps_0.
$$
On the other hand, we recall that, for $\eps$ small enough (with respect to $\delta$), it holds:
$$
\H^{d-1}(\Omega\cap \partial^\ast\EL ) + 2 \H^{d-1}(K_{\Lambda,\eps} \setminus\partial^\ast\EL)\le 2\H^{d-1}(D_0)+C\eps^2,
$$
which implies that for every $\eps\le \eps_0$
\begin{equation}
\label{eq:quantostain_S110}
    \H^{d-1}(\Omega\cap \partial^\ast\EL \setminus \mathcal T_{2\delta \sigma}) + 2 \H^{d-1}\left((K_{\Lambda,\eps} \setminus\partial^\ast\EL)\setminus \mathcal T_{2 \delta \sigma}\right)\le \zeta+C\eps_0^2.
\end{equation}
Now, assuming by contradiction that 
$$
\H^{d-1}(\Omega \cap \partial^\ast\EL\setminus \mathcal T_{\delta \sigma}) + 2 \H^{d-1}\left((K_{\Lambda,\eps} \setminus\partial^\ast\EL)\setminus \mathcal T_{\delta \sigma}\right) >0,
$$
and combining the density estimate \cref{l:density-estimate} with
\eqref{eq:quantostain_S110}, we get
$$
\zeta+C\eps_0^2\geq \H^{d-1}(\pa^\ast E_{\Lambda, \eps} \cap \Omega \setminus \mathcal T_{2\delta \sigma}) + 2 \H^{d-1}\left((K_{\Lambda,\eps} \setminus\partial^\ast\EL)\setminus \mathcal T_{2\delta \sigma}\right) \geq C\left(\min\left\{r_0,2 {\delta \sigma}\right\}\right)^{d-1},
$$
where $r_0$ and $C$ are the constants from \eqref{eq:denisty_estimate_lemma}.
Choosing
$$
\zeta = \frac{1}{2}C \left(\min\left\{r_0,2{\delta \sigma}\right\}\right)^{d-1} \qquad \hbox{ and }\qquad C\eps_0^2 \leq \zeta,
$$
we get a contradiction. Hence, \eqref{eq:tuttostriscia} must be true. 
\end{proof}

In order to get the main result of the section, we provide an equivalent formulation.
\begin{proposition}
\label{prop:lunga!!!}
From \cref{lemma:Estanellastriscia}, let $\sigma = \sfrac{1}{5}$. For every $(K,E)\in\mathcal A(W,\mathcal C)$ minimizer of $\mathcal F_{\Lambda,\eps}$ contained in the strip $\mathcal T_{\sfrac{\delta}{5}}$, there are two disjoint connected open sets $S_+$ and $S_-$ of finite perimeter with the following properties. Let $\widetilde\Omega$ be the only bounded connected component of $\Om \cap \mathcal T_\delta$. Then:
\begin{enumerate}[\rm(i)]
\item $\widetilde\Omega\cap\{\sfrac{\delta}{5}<x_d<\delta\}\subset S_+$;
\item $\widetilde\Omega\cap\{-\delta<x_d<-\sfrac{\delta}{5}\}\subset S_-$; 
\item up to a $\H^{d-1}$-negligible set, we have
\begin{equation}
\label{e:tildeK_tildeE}
 K=\widetilde\Omega\cap(\partial S_+\cup\partial S_-)\qquad\text{and}\qquad  E=\widetilde\Omega\setminus (\overline S_+\cup\overline S_-),
\end{equation}
and the energy of $(K,E)$ is given by
$$\FF(K,E)=\H^{d-1}(\widetilde\Omega\cap \partial^\ast S_+)+\H^{d-1}(\widetilde\Omega\cap \partial^\ast S_-)+\Lambda\left|{\rm Vol}\left(\widetilde\Omega\setminus (\overline S_+\cup\overline S_-)\right)-\eps\right|.
$$
\end{enumerate}
\end{proposition}
\begin{proof}
We divide the proof in few steps.
\\
\\
\noindent {\it Step 1. Construction of $S_+$ and $S_-$.} Since $K \cup E \subset \Omega \cap \mathcal T_{\sfrac{\delta}{5}}$, by \cref{lemma:Estanellastriscia}, we have that 
\begin{align}
\label{e:omega+}
    \Omega_+ \cap (K\cup E)&:= \Big(\widetilde\Omega\cap\{\sfrac{\delta}{5}<x_d<\delta\}\Big)\cap (K\cup E)=\emptyset,\\ 
\label{e:omega-}
    \Omega_-\cap (K\cup E)&:=\Big(\widetilde\Omega\cap\{-\delta<x_d<-\sfrac{\delta}{5}\}\Big)\cap (K\cup E)=\emptyset.
\end{align}
Let $S_+$ and $S_-$ be the connected components of the open set $$\widetilde\Omega\setminus (K\cup E),$$
containing $\Omega_+:=\widetilde\Omega\cap\{\sfrac{\delta}{5}<x_d<\delta\}$ and $\Omega_-:=\widetilde\Omega\cap\{-\delta<x_d<-\sfrac{\delta}{5}\}$ respectively. Since $K \cup E$ is $\mathcal C$-spanning, we obtain that $S_+\cap S_-=\emptyset$.
Moreover,
we have that
$$\widetilde\Omega\cap\partial S_\pm\subset \mathcal T_{\sfrac{\delta}{5}}.$$
Since 
$$\widetilde\Omega\cap\partial S_\pm\subset K$$
and $\mathcal H^{d-1}(K)<+\infty$, we get
$$\H^{d-1}(\widetilde\Omega\cap\partial S_\pm)\leq  \mathcal H^{d-1}(K)<+\infty$$ 
implying that by \cite[Section 5.11]{evans2015} the sets $S_\pm$ have relatively finite perimeter in $\widetilde\Omega$ and 
$$\H^{d-1}(\widetilde\Omega\cap \partial^\ast S_\pm)\le \H^{d-1}(\widetilde\Omega\cap \partial S_\pm).$$
Moreover, we also know that $\widetilde\Omega\cap \partial S_\pm$ are both $\mathcal C$-spanning since $\widetilde \Omega$ is homotopically equivalent to a cylinder.
\\
\\
\noindent{\it Step 2.} 
Let $\{S_j\}_{j\ge 1}$ be the family of connected components of $\widetilde\Omega\setminus(K \cup E)$ different from $S_+$ and $S_-$. We claim that:
\begin{equation}\label{e:competitor-in-strip-volume}
{\rm Vol}\left(\widetilde\Omega\setminus(\overline S_+\cup \overline S_-)\right)={\rm Vol}(E)+\sum_{j=1}^{+\infty}{\rm Vol}(S_j),
\end{equation}
and
\begin{equation}\label{e:competitor-in-strip-perimeter}
\H^{d-1}(\widetilde\Omega\cap\partial^\ast S_+)+\H^{d-1}(\widetilde\Omega\cap\partial^\ast S_-)+\sum_{j=1}^{+\infty}\H^{d-1}(\widetilde\Omega\cap\partial^\ast S_j)\leq \mathcal F(K,E).
\end{equation}
The equality \eqref{e:competitor-in-strip-volume} follows immediately. Indeed, by construction, 
$$E\cup\left(\bigcup_{j=1}^{+\infty}S_j\right)\subset\, \widetilde\Omega\setminus(\overline S_+\cup \overline S_-)\,\subset E\cup\left(\bigcup_{j=1}^{+\infty}S_j\right)\cup K$$
and $\H^{d-1}(K)<+\infty$. In order to prove \eqref{e:competitor-in-strip-perimeter}, we notice that: 
\begin{itemize}
\item every connected component $S$ of $\widetilde \Omega\setminus(K\cup E)$ has finite perimeter in $\widetilde\Omega$ and, by Federer's theorem \cite[Theorem 16.2]{maggi_book}, up to removing a $\H^{d-1}$-negligible set, we have: 
$$\partial S=(S^{(0)}\cap\partial S)\cup (S^{(1)}\cap\partial S)\cup S^{(\sfrac12)}\qquad\text{and}\qquad \partial^\ast S=S^{(\sfrac12)},$$
where $S^{(\alpha)}$ denotes the set of points in $\R^d$ at which the Lebesgue density of $S$ is equal to $\alpha$;
\item similarly, up to removing another $\H^{d-1}$-negligible set, we have: 
$$\partial E=(E^{(0)}\cap\partial E)\cup (E^{(1)}\cap\partial E)\cup E^{(\sfrac12)}\qquad\text{and}\qquad \partial^\ast E=E^{(\sfrac12)};$$
\item by \cite{king2022plateau, MaggiNovackRestrepo}, we have that for $\mathcal H^{d-1}$-almost every point $x_0\in K\subset\Omega$, the following holds: 
$$\theta(x_0):=\lim_{r\to0}\frac{\H^{d-1}(B_r(x_0)\cap \partial^\ast E)+2\H^{d-1}(B_r(x_0)\cap K\setminus\partial^\ast E)}{\omega_{d-1}r^{d-1}}=\begin{cases}1\quad\text{if}\quad x_0\in\partial^\ast E,\smallskip\\
2\quad\text{if}\quad x_0\in K\setminus \partial^\ast E.
\end{cases}$$
\end{itemize}
Let $S'$ and $S''$ be two different connected components of $\widetilde \Omega\setminus(K\cup E)$. Then, for $\H^{d-1}$-almost every point $x_0$ it holds:
 \begin{itemize}
\item if $x_0\in \partial^{\ast}S'\cap\partial^{\ast}S''$, then $x_0\notin\partial^\ast E$ and so $\theta(x_0)=2$; 
moreover, $x_0\notin  \partial^\ast S'''$, for any connected component $S'''$ different from $S'$ and $S''$;
\item if $x_0\in \partial^{\ast}S'$ and $x_0\notin   \partial^\ast S''$, for any connected component $S''\neq S'$, then $\theta(x_0)=1$ and $x_0\in\partial^\ast E$.
\end{itemize}
This concludes the proof of \eqref{e:competitor-in-strip-perimeter}; notice that the inequality in \eqref{e:competitor-in-strip-perimeter} a priori might be strict; for instance, there might be points in $K\cap E^{(1)}$ at which $\theta=2$ or points in $K\cap E^{(0)}$  at which one of the connected components of $\widetilde\Omega\setminus(K\cup E)$ have density one. We also notice that by the same argument 
$$\H^{d-1}(\widetilde\Omega\cap\partial^\ast S_+)+\H^{d-1}(\widetilde\Omega\cap\partial^\ast S_-)\le \mathcal F(K',E'),$$
where the couple $(K',E')$ is defined as follows: 
\begin{equation}
    E'=\widetilde\Omega\setminus(\overline S_+\cup \overline S_-)\qquad\text{and}\qquad K'=\widetilde\Omega\cap\big(\partial S_+\cup\partial S_-\big)\,.
\end{equation}
and that again the above inequality might be strict.
\\
\\
\noindent{\it Step 3.} Directly from \eqref{e:competitor-in-strip-perimeter}, we get 
\begin{align*}
\mathcal F(K,E)+\Lambda|{\rm Vol}(E)-\eps|&\ge \H^{d-1}(\widetilde\Omega\cap\partial^\ast S_+)+\H^{d-1}(\widetilde\Omega\cap\partial^\ast S_-)+\sum_{j=1}^{+\infty}\H^{d-1}(\widetilde\Omega\cap\partial^\ast S_j)\\
&\qquad+\Lambda\Big|{\rm Vol}\left(\widetilde\Omega\setminus(\overline S_+\cup \overline S_-)\right)-\sum_{j=1}^{+\infty}{\rm Vol}(S_j)-\eps\Big|\\
&\ge \mathcal{G}(S_+, S_-)+\sum_{j=1}^{+\infty}\H^{d-1}(\widetilde\Omega\cap\partial^\ast S_j)-\Lambda \sum_{j=1}^{+\infty}{\rm Vol}(S_j),
\end{align*}
where
\begin{equation}
    \label{e:def_G}
    \mathcal{G}(S_+, S_-):=\H^{d-1}(\widetilde\Omega\cap\partial^\ast S_+)+\H^{d-1}(\widetilde\Omega\cap\partial^\ast S_-)+\Lambda|{\rm Vol}(\widetilde\Omega\setminus(S_+\cup S_-))-\eps|.
\end{equation}
By the relative isoperimetric inequality in $\Omega$, we have that for all (non-empty) $S_j$
\begin{align*}
\H^{d-1}(\widetilde\Omega\cap\partial^\ast S_j)-\Lambda {\rm Vol}(S_j)=\H^{d-1}(\Omega\cap\partial^\ast S_j)-\Lambda {\rm Vol}(S_j) > 0,
\end{align*}
hence we obtain
\begin{align*}
\mathcal F(K,E)+\Lambda|{\rm Vol}(E)-\eps| \geq \mathcal{G}(S_+, S_-),
\end{align*}
where $\mathcal G$ is defined in \eqref{e:def_G}.
\\
\\
\noindent{\it Step 4. We show that $\mathcal G$ admits a minimizer.} Consider the following variational problem of minimizing
$$\mathcal{G}(U_+, U_-):=\H^{d-1}(\widetilde\Omega\cap\partial^\ast U_+)+\H^{d-1}(\widetilde\Omega\cap\partial^\ast U_-)+\Lambda|{\rm Vol}(\widetilde\Omega\setminus(U_+\cup U_-))-\eps|$$
among all disjoint sets $U_+$, $U_-$ of finite perimeter such that $\Omega_+\subset U_+$ and $\Omega_-\subset U_-$, where $\Omega_+$ and $\Omega_-$ are defined in \eqref{e:omega+} and \eqref{e:omega-} respectively. The existence of a minimizer $(A_+,A_-)$ follows by classical arguments for set of finite perimeters. Moreover, the sets $A_\pm$ are open and there is a constant $c\in(0,1)$ such that the following density estimate holds for $A_+$ and $A_-$ separately:
\begin{equation}
    \label{e:density_est_A+A-}
    c\,{\rm Vol}(B_r)\le {\rm Vol}(A_\pm\cap B_r(x_0))\le (1-c){\rm Vol}(B_r),
\end{equation}
for every $x_0\in\Omega\cap\partial A_\pm$ and every $r\in(0,1)$. In particular, up to a $\H^{d-1}$ negligible set, we have:
\begin{equation}
\label{e:bordi_ridotti}
    \partial A_\pm=\partial^\ast A_\pm=A_\pm^{(\sfrac12)}\qquad\text{and}\qquad (A_\pm)^{(0)}\cap\partial A_\pm=(A_\pm)^{(1)}\cap\partial A_\pm=\emptyset.
\end{equation}
We next define the couple $(K'',E'')$ as follows: 
\begin{equation*}
    E''=\widetilde\Omega\setminus(\overline A_+\cup \overline A_-)\qquad\text{and}\qquad K''=\widetilde\Omega\cap(\partial A_+\cup\partial A_-)\,,
\end{equation*}
and we notice that $K''$ is $(d-1)$-rectifiable, $\partial E''\subset K''$ and $K''$ is $\mathcal C$-spanning (in particular, also $K''\cup E''$ is $\mathcal C$-spanning).
We claim that 
\begin{equation}
    \label{e:uguaglianza_funzionale}
    \H^{d-1}(\widetilde\Omega\cap\partial^\ast A_+)+\H^{d-1}(\widetilde\Omega\cap\partial^\ast A_-)=\mathcal F(K'',E'').
\end{equation}
In order to show the above identity, we need to analyze different cases
\begin{enumerate}
    \item By \eqref{e:bordi_ridotti}, we have that $$(E'')^{(\sfrac12)}\cap A_\pm^{(1)}=\emptyset\ \qquad\text{and}\qquad (E'')^{(\sfrac12)}\cap A_\pm^{(0)}=\emptyset;$$
    since $\widetilde\Omega=A_+\cup A_-\cup E''$ up to a set of Lebesgue measure zero, we have that
    \begin{align*}
    \H^{d-1}(\widetilde\Omega \cap \pa^\ast E'') &=\H^{d-1}(\widetilde\Omega\cap\partial^\ast A_+\cap \pa^\ast E'')+\H^{d-1}(\widetilde\Omega\cap \pa^\ast E''\setminus\partial^\ast A_+)\\
    &=\H^{d-1}(\widetilde\Omega\cap\partial^\ast A_+\cap \pa^\ast E'')+\H^{d-1}(\widetilde\Omega\cap \pa^\ast E''\cap\partial^\ast A_-)\\
    &=\H^{d-1}(\widetilde\Omega\cap\partial^\ast A_+\setminus \pa^\ast A_-)+\H^{d-1}(\widetilde\Omega\cap \pa^\ast A_-\setminus\partial^\ast A_+).
    \end{align*}
    \item We next analyze the set $K''\setminus\partial^\ast E''$. We have that 
    $$K''\setminus\partial^\ast E''\subset (E'')^{(0)},$$
    up to a $\H^{d-1}$ negligible set; moreover, if $x_0 \in (E'')^{(0)}\cap K''$, then by \eqref{e:bordi_ridotti}, $x_0 \in (A_+)^{(\sfrac{1}{2})}\cap (A_-)^{(\sfrac{1}{2})}$. Thus
    $$\H^{d-1}(K''\setminus \partial^\ast E'')=\H^{d-1}((E'')^{(0)}\cap K'')=\H^{d-1}(\widetilde\Omega \cap \left(\pa^\ast A_+\cap \pa^\ast A_-\right)).$$
\end{enumerate}
In conclusion, putting together the points 1 and 2 above, we get
    $$
    \begin{aligned}
        \H^{d-1}(\widetilde\Omega \cap \pa^\ast E'') + 2 \H^{d-1}(K''\setminus \pa^\ast E'')= \H^{d-1}(\widetilde\Omega\cap\partial^\ast A_+)+\H^{d-1}(\widetilde\Omega\cap\partial^\ast A_-),
    \end{aligned}
    $$
    which is precisely \eqref{e:uguaglianza_funzionale}.
    \\
    \\
    {\it Step 5.}
By Step 3 and Step 4, we have:
\begin{align*}
\mathcal F(K,E)+\Lambda|{\rm Vol}(E)-\eps| &\geq \mathcal{G}(S_+, S_-)\\
&\geq \mathcal{G}(A_+, A_-)\\
&=\mathcal F(K'',E'')+\Lambda|{\rm Vol}(E'')-\eps|.
\end{align*}
By minimality of $(K,E)$ given by \cref{prop:esistenza_FF}, we get 
$$
\mathcal F_{\Lambda,\eps}(K,E) = \mathcal F_{\Lambda,\eps}(K'', E''),
$$
and thus that 
$$\mathcal F(K,E)+\Lambda|{\rm Vol}(E)-\eps| =\mathcal{G}(S_+, S_-)\qquad\text{and}\qquad \mathcal{G}(S_+, S_-)= \mathcal{G}(A_+, A_-).$$
By Step 3, the first equality above implies that $\widetilde\Omega\setminus (K\cup E)=S_+\cup S_-$, while the second one gives that $(S_+,S_-)$ minimizes $\mathcal G$ among all disjoint sets $U_+$, $U_-$ such that $\Omega_+\subset U_+$ and $\Omega_-\subset U_-$; in particular, $S_+$ and $S_-$ satisfy the density estimates \eqref{e:density_est_A+A-} and it holds: 
$$\mathcal F(K,E)=\H^{d-1}(\widetilde\Omega\cap \partial^\ast S_+)+\H^{d-1}(\widetilde\Omega\cap \partial^\ast S_-),$$
we have the identities $K=\widetilde\Omega\cap\left(\partial S_+\cup \partial S_-\right)$ and $E=\widetilde\Omega\setminus(\overline S_+\cup\overline S_-)$ and this concludes the proof.
\end{proof}

Once we have introduced $S_+$ and $S_-$, in the following Lemma, we provide their additional regularity properties.

\begin{lemma}
\label{lemma:pernonesseretroppolunghi}
Let $d\ge 3$. Then, there is $\eps_0$ such that for every $\eps\le \eps_0$ the following holds.
Let $S_+$ and $S_-$ be the open connected disjoint sets from \cref{prop:lunga!!!}. Then
\begin{enumerate}[\rm(i)]
\item $\overline S_+\subset\{x_d>0\}$ and $\overline S_-\subset\{x_d<0\}$; 
\item $K_+:=\widetilde\Omega \cap \partial S_+$ and $K_-:=\widetilde\Omega \cap \partial S_-$ (oriented oppositely with respect to $\partial E$) have constant mean curvature $\lambda$;
\item $K_+$ and $K_-$ are orthogonal to $\partial \Omega$ in viscosity sense.
\end{enumerate}
\end{lemma}
\begin{proof} We divide the proof in some steps.
\\
\\
\noindent{\it Step 1. We prove the regularity of $S_+$ and $S_-$.} Let $B_r(x_0)$ be a ball contained in $\Omega$ and let $V_+$ and $V_-$ be minimizers of the variational problems
$$V_\pm=\argmin\Big\{\H^{d-1}(\pa^\ast U_\pm \cap B_r(x_0))\ :\ {\rm Per }(U_\pm)<+\infty\,;\quad U_\pm=S_\pm\ \text{in}\ \widetilde\Omega\setminus B_r(x_0)\Big\}.$$
Then, $V_+$ and $V_-$ are disjoint sets. On the other hand, the minimality of $(S_+,S_-)$ yields:
$$\H^{d-1}(\pa^\ast S_+ \cap B_r(x_0))+\H^{d-1}(\pa^\ast S_- \cap B_r(x_0))\le \H^{d-1}(\pa^\ast V_+ \cap B_r(x_0))+\H^{d-1}(\pa^\ast V_- \cap B_r(x_0))+\Lambda|B_r|,$$
which implies that separately $S_+$ and $S_-$ are $(\Lambda,r_0)$-minimizers (for details we refer to \cite[Chapter 21]{maggi_book}) in the sense that 
$$\H^{d-1}(\pa^\ast S_\pm \cap B_r(x_0))\le \H^{d-1}(\pa^\ast U_\pm \cap B_r(x_0))+\Lambda|B_r|,$$
for every $U_\pm$ of finite perimeter such that $U_\pm=S_\pm$ outside $B_r(x_0)$. In particular, there are sets $\Sigma_\pm\subset \widetilde\Omega\cap\partial S_\pm=: K_\pm$, relatively closed in $\Omega$ and of Hausdorff dimension at most $d-8$, such that $K_\pm\setminus\Sigma_\pm$ is $C^{1,\alpha}$ manifold. 
\\
\\
\noindent{\it Step 2. $ S_+\subset\{x_d>0\}$ and $S_-\subset\{x_d<0\}$.}
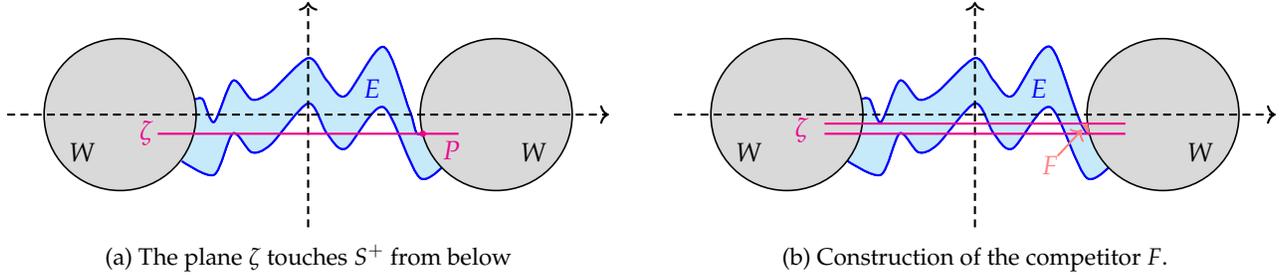
\begin{figure}[htbp]
	\begin{subfigure}{.5\linewidth}
		\centering
		\begin{tikzpicture}[rotate=0, scale= 0.5]
\draw [thick, color=blue, name path=s] plot [smooth] coordinates {(-3,0.4) (-2.8,0.4) (-2.5,-0.2) (-2,0.9) (-1.5, 0.4)(-1,0.6) (0,1.5) (0.5, 0.9) (1,0.5) (2,1.8) (2.7, 0.1) (3,-0.5) (4,1)};
\draw [thick, color=blue, name path=t] plot [smooth] coordinates {(-3.355,-1.237) (-2.5,-1.6) (-2,-0.5) (-1.5, -1)(-1,-0.8) (0,0.3) (0.5, -0.5) (1,-0.9) (2,0.2) (3,-1.7) (4,-0.9)};
\begin{scope}[transparency group]
\tikzfillbetween[of=s and t] {color=cyan!20};
\end{scope}
\draw [thick, color=blue] plot [smooth] coordinates {(-3,0.4) (-2.8,0.4) (-2.5,-0.2) (-2,0.9) (-1.5, 0.4)(-1,0.6) (0,1.5) (0.5, 0.9) (1,0.5) (2,1.8) (2.7, 0.1) (3,-0.5) (4,1)};
\draw [thick, color=blue] plot [smooth] coordinates {(-3.355,-1.237) (-2.5,-1.6) (-2,-0.5) (-1.5, -1)(-1,-0.8) (0,0.3) (0.5, -0.5) (1,-0.9) (2,0.2) (3,-1.7) (4,-0.9)};
\draw[draw=none, fill=white] (5,0) circle [radius = 2cm];
\draw[draw=none, fill=white] (-5,0) circle [radius = 2cm];
\draw[very thick, color=black, name path=a] (-5,0) circle [radius=2];
\draw[very thick, color=black, name path=b] (5,0) circle [radius=2];
\tikzfillbetween[of=a and b] {color=gray!30};
\draw [color = black] node at (6,-1) {$W$};
\draw [color = black] node at (-6,-1) {$W$};
\draw [color = blue] node at (1.7,0.7) {$E$};
\draw[thick, densely dashed, color=black,->] (-8,0) -- (8,0);
\draw[thick, densely dashed, color=black,->] (0,-3) -- (0,3);
\draw[thick, color=magenta] (-4,-0.51) -- (4,-0.51);
\filldraw [magenta] (3.05,-0.51) circle (2pt);
\draw [color = magenta] node at (-4.3,-0.5) {$\zeta$};
\draw [color = magenta] node at (3.8,-0.96) {$P$};
\end{tikzpicture}
		\caption{The plane $\zeta$ touches $S^+$ from below}
		\label{fig:piano_tocca_S+}
	\end{subfigure}
 \begin{subfigure}{.5\linewidth}
		\centering
		\begin{tikzpicture}[rotate=0, scale= 0.5]
\draw [thick, color=blue, name path=s] plot [smooth] coordinates {(-3.05,0.4) (-2.8,0.4) (-2.5,-0.2) (-2,0.9) (-1.5, 0.4)(-1,0.6) (0,1.5) (0.5, 0.9) (1,0.5) (2,1.8) (3,-0.5) (4,1) (5,0.5)};
\draw [thick, color=blue, name path=t] plot [smooth] coordinates {(-3.39,-1.23) (-2.5,-1.6) (-2,-0.5) (-1.5, -1)(-1,-0.8) (0,0.3) (0.5, -0.5) (1,-0.9) (2,0.2) (3,-1.7) (4,-0.9)};
\begin{scope}[transparency group]
\tikzfillbetween[of=s and t] {color=cyan!20};
\end{scope}
\draw [thick, color=blue] plot [smooth] coordinates {(-3.05,0.4) (-2.8,0.4) (-2.5,-0.2) (-2,0.9) (-1.5, 0.4)(-1,0.6) (0,1.5) (0.5, 0.9) (1,0.5) (2,1.8) (3,-0.5) (4,1) (5,0.5)};
\draw [thick, color=blue] plot [smooth] coordinates {(-3.39,-1.23) (-2.5,-1.6) (-2,-0.5) (-1.5, -1)(-1,-0.8) (0,0.3) (0.5, -0.5) (1,-0.9) (2,0.2) (3,-1.7) (4,-0.9)};
\draw[draw=none, fill=white] (5,0) circle [radius = 2cm];
\draw[draw=none, fill=white] (-5,0) circle [radius = 2cm];
\draw[very thick, color=black, name path=a] (-5,0) circle [radius=2];
\draw[very thick, color=black, name path=b] (5,0) circle [radius=2];
\tikzfillbetween[of=a and b] {color=gray!30};
\draw [color = black] node at (6,-1) {$W$};
\draw [color = black] node at (-6,-1) {$W$};
\draw [color = blue] node at (1.7,0.7) {$E$};
\draw[thick, densely dashed, color=black,->] (-8,0) -- (8,0);
\draw[thick, densely dashed, color=black,->] (0,-3) -- (0,3);
\draw[thick, color=magenta] (-4,-0.24) -- (4,-0.24);
\draw[thick, color=magenta] (-4,-0.51) -- (4,-0.51);
\draw [color = magenta] node at (-4.6,-0.41) {$\zeta$};
\draw[draw=none, fill=red!50]  (2.83,-0.24) -- (3,-0.24) -- (3.05,-0.5) -- (2.96,-0.46) -- cycle;
\draw[thick, color =red!50, ->] (2.2,-1.1) -- (2.85,-0.4);
\draw [color = red!50] node at (2,-1.34) {$F$};
\end{tikzpicture}
		\caption{Construction of the competitor $F$.}
		\label{fig:costruzione_F}
	\end{subfigure}
	\caption{(a) Graphical sectional representation of the touching point $P \in \pa \Omega\cap\partial S_+$ from below by a horizontal plane $\zeta$. (b) Graphical sectional representation of the constraction of the competitor $F$ defined in \eqref{eq:def_F_insieme}.}
	\label{fig:non_toccco_il_bordo_S+}
\end{figure}
We will prove that $S_+\subset\{x_d>0\}$.
Let us consider a plane $\zeta:=\{x_d = c\}$ moved from $-\infty$ along the $x_d$-axis and we stop moving it until it touches a point $P \in \pa S_+$. We claim that $P\in\{x_d\ge 0\}$. Suppose by contradiction that $P\in\{x_d< 0\}$. We analyze different cases. 
\begin{enumerate}[{\rm (A)}]
\item Suppose that $P \in \pa \Omega\cap\partial S_+$. Then also $P \in\pa \Omega\cap {\rm cl}(K_+)$ (see for a graphical representation \cref{fig:piano_tocca_S+}). 
Hence, we {\em move a little bit} the plane inside $S_+$.  
For every $\xi>0$, we define (see \cref{fig:costruzione_F}) the open set 
\begin{equation}
\label{eq:def_F_insieme}
    F_\xi := S_+ \cap \left\{x_d > c + \xi\right\}.
\end{equation}
We will show that (after a modification in some small ball) the couple $(F_\xi,S_-)$ has strictly smaller energy than $(S_+,S_-)$.
Indeed, on one hand, by choosing $\xi_0$ small enough such that 
$${\rm Vol}(\Omega\cap\{c\le x_d\le c+\xi_0\})\le \frac{1}{2}{\rm Vol}(\widetilde\Omega\setminus(S_+\cup S_-)),$$
we know that there is a point $Q\in \partial^\ast E$ lying in $\Omega\cap \{x_d<c\}$ or in $\Omega\cap \{x_d>c+\xi_0\}$. In a neighborhood $B_r(Q)$, either of the following holds:
\begin{itemize}
\item $B_r(Q)\cap S_+=\emptyset$ and $K_-\cap B_r(Q)$  is a smooth manifold of negative constant mean curvature (directed outwards $S_-$);
\item $B_r(Q)\cap S_-=\emptyset$ and $K_+\cap B_r(Q)$  is a smooth manifold of positive constant mean curvature (directed inwards $S_+$).
\end{itemize}
By performing a smooth variation of $S_+$ or $S_-$ in $B_r(Q)$ (and leaving the other set invariant), we obtain a couple $(\widetilde S_+,\widetilde S_-)$ of disjoint open sets that coincide with $(S_+,S_-)$ outside $B_r(Q)$ and such that:
$$\begin{aligned}
    \H^{d-1}(B_r(Q)\cap\partial^\ast \widetilde S_+)+\H^{d-1}(B_r(Q)\cap\partial^\ast \widetilde S_-)&< \H^{d-1}(B_r(Q)\cap\partial^\ast S_+)+\H^{d-1}(B_r(Q)\cap\partial^\ast S_-),\\
    {\rm Vol}(B_r(Q)\setminus(\widetilde S_+\cup \widetilde S_-))&<{\rm Vol}(B_r(Q)\setminus(S_+\cup S_-)).
\end{aligned}$$
Consider the constant vector field $X=e_d$, we have that $X \cdot \nu_{\Omega} \geq 0$ in a neighborhood of $\overline  S_+\setminus F_\xi$, where $\nu_{\Omega}$ is the outer normal to the boundary of $\Omega$. Then, by the divergence theorem in $S_+\setminus F_\xi$, we obtain
$$
\begin{aligned}
0 = \int_{S_+\setminus F_\xi} \diver\, X \, dx &\ge  \int_{\widetilde\Omega\cap\partial S^+ \cap \{x_d<c+\xi\}} X \cdot \nu_{S^+}\, d\H^{d-1} + \int_{\{x_d=c+\xi\} \cap S^+} X \cdot e_d \, d\H^{d-1}\\
&\ge  -\H^{d-1}(\widetilde\Omega\cap\{x_d<c+\xi\}\cap\partial S^+) + \H^{d-1}(\widetilde\Omega\cap\{x_d=c+\xi\}\cap S_+)\\
&=-\H^{d-1}(\widetilde\Omega\cap\partial^\ast S^+) + \H^{d-1}(\widetilde\Omega\cap\partial^\ast F_\xi),
\end{aligned}
$$
where we denoted with $\nu_{S_+}$ the outer normal to the boundary of $S_+$. We  choose $\xi$ and the couple $(\widetilde S_+,\widetilde S_-)$ in such a way that 
$$ {\rm Vol}(B_r(Q)\setminus(S_+\cup S_-))-{\rm Vol}(B_r(Q)\setminus(\widetilde S_+\cup \widetilde S_-))={\rm Vol}(S_+\setminus F_\xi).$$
Now taking the competitor 
$$Z_-=\widetilde S_-\ ,\qquad Z_+=\widetilde S_+\cap \{x_d> c+\xi\},$$
and noticing that $Z_+=F_\xi$ outside the small ball $B_r(Q)$, we can conclude that 
$$\begin{aligned}
    \H^{d-1}(\widetilde\Omega\cap\partial^\ast  Z_+)+\H^{d-1}(\widetilde\Omega\cap\partial^\ast Z_-)&< \H^{d-1}(\widetilde\Omega\cap\partial^\ast S_+)+\H^{d-1}(\widetilde\Omega\cap\partial^\ast S_-),\\
    {\rm Vol}(\widetilde\Omega\setminus(Z_+\cup Z_-))&={\rm Vol}(\widetilde\Omega\setminus(S_+\cup S_-)),
\end{aligned}$$
which is a contradiction with the minimality of $(S_+,S_-)$.

\item Suppose that $P \in \widetilde\Omega \cap \pa S_+ \setminus \pa S_-=K_+\setminus K_-$. Since $S_+$ is an almost minimizer of the perimeter, we have that $K_+$ is a $C^{1,\alpha}$ graph in a neighbourhood of $P$. Now, from \cref{lemma:stime_curvature}, we deduce that $K_+$ is $C^\infty$ and has constant mean curvature $-\lambda<0$ in a neighborhood of $P$, which give a contradiction by the maximum principle. 
\item Suppose that $P \in \widetilde\Omega \cap \pa S^+\cap \pa S^-= K_+\cap K_-$. Also in this case $K_+$ and $K_-$ are $C^{1,\alpha}$ graphs with tangent plane $\{x_d=c\}$ at $P$. Moreover, the generalized mean curvature of $K_+$ in a neighbourhood of $P$ is given by 
\begin{equation}\label{e:gen-mean-curve-S+}
\lambda_+(x):=\begin{cases}
-\lambda\ &\text{ if }\  x\in\widetilde\Omega\cap\partial S_+\setminus\partial S_-\ ,\\
0\ &\text{ if }\  x\in\widetilde\Omega\cap\partial S_+\cap\partial S_-\ .
\end{cases}
\end{equation}
Again, by Sch\"atzle's maximum principle \cite[Theorem 6.2]{schatzle2004quadratic}, $\{x_d=c\} \subset K_+$, so there is second contact point $P_2\in \{x_d=c\}\cap\partial\Omega\cap \pa S_+$. As we showed in the case (A) above, this is not possible.
\end{enumerate} 
Similar conclusions can be provided for $S_-\subset\{x_d<0\}$.
\\
\\
{\it Step 3. We claim that $\partial S_+ \cap (\widetilde\Omega\cap\{x_d = 0\}) = \emptyset$ and $\partial S_- \cap (\widetilde\Omega\cap\{x_d = 0\}) = \emptyset$.} We analyze different cases.
\begin{enumerate}[{\rm (1)}]
   \item Suppose by contradiction that there is a point $x_0\in (\partial S_+ \setminus \partial  S_-)\cap (\{x_d = 0\}\cap\widetilde\Omega)$. 
As in case (B) of Step 2, we have that, in a neighborhood of $x_0$, $\partial S_+$ is a $C^\infty$ graph with constant mean curvature $-\lambda$, but this is a contradiction by the maximum principle. Thus, $(\partial S_+ \setminus \partial  S_-)\cap (\{x_d = 0\}\cap\widetilde\Omega)=\emptyset$.
\item Suppose by contradiction that there is a point $x_0\in (\partial S_+ \cap \partial  S_-)\cap (\{x_d = 0\}\cap\widetilde\Omega)$. Then, the plane $\{x_d=0\}$ touches $\partial S_+$ from below and $\partial S_-$ from above at $x_0$. As in case (C) of Step 2, this implies  $\partial S_+$ and $\partial S_-$ are both $C^{1,\alpha}$ graphs; moreover, in a neighbourhood of $x_0$, $\partial S_+$ has non-negative generalized mean curvature (see \eqref{e:gen-mean-curve-S+}), while $\partial S_-$ has non-positive generalized mean curvature. Applying again the maximum principle (see for instance \cite[Theorem 6.2]{schatzle2004quadratic}), we obtain that $(\{x_d=0\}\cap\widetilde\Omega)\subset \partial S_+$ and $(\{x_d=0\}\cap\widetilde\Omega)\subset \partial S_-$. Then, up to a $\H^{d-1}$-negligible set, we have that 
$$D_0=\{x_d=0\}\cap\widetilde\Omega\,\subset\,\partial^\ast S_+\cap \partial^\ast S_-=K\setminus\partial^\ast E.$$
In particular, this implies that 
$$\mathcal F(K,E)\ge 2\H^{d-1}(D_0)+\H^{d-1}(\widetilde\Omega\cap\partial^\ast E)\ge 2\H^{d-1}(D_0)+C_1\big({\rm Vol}(E)\big)^{\frac{d}{d-1}}=2\H^{d-1}(D_0)+C_1\eps^{\frac{d}{d-1}},$$
by the relative isoperimetric inequality in $\Omega$. On the other hand, we recall that (see \cref{prop:esistenza_FF}) the optimality of $(K,E)$ yields
$$\mathcal F(K,E)\le 2\H^{d-1}(D_0)+C_2\eps^2,$$
where $C_1$ and $C_2$ are constants depending on $\widetilde\Omega$. By choosing $\eps$ small enough, this leads to a contradiction.
\end{enumerate}
Thus $(\partial S_+\cap \partial S_-) \cap (\{x_d = 0\}\cap\widetilde\Omega)=\emptyset$, which concludes the proof of Step 3. In particular, this implies that $\widetilde\Omega\cap \partial^\ast S_+$ is a $C^\infty$ smooth manifold of constant mean curvature $\lambda$ and  $\partial^\ast S_+\cap \widetilde\Omega$ is a $C^\infty$ smooth manifold of constant mean curvature $-\lambda$, while the sets  $\widetilde\Omega \cap (\partial S_\pm\setminus \partial^\ast S_\pm)$ are relatively closed in $\widetilde\Omega$ and have Hausdorff dimension $d-8$; moreover, $S_+$ and $S_-$ are local $\Lambda$-minimizers of the perimeter respectively in $\{x_d>0\}\cap\widetilde\Omega$ and $\{x_d<0\}\cap\widetilde\Omega$. This concludes the proof of claim {\rm (ii)} of the present lemma.
\\
\\
\noindent{\it Step 4. We claim that $\partial S_+ \cap (\partial\Omega\cap\{x_d = 0\}) = \emptyset$ and $\partial S_- \cap (\partial\Omega\cap\{x_d = 0\}) = \emptyset$.} We will show that $\partial S_+ \cap (\partial\Omega\cap\{x_d = 0\}) = \emptyset$. 
Suppose that $x_0\in\partial S_+ \cap (\partial\Omega\cap\{x_d = 0\})$. Let $B_{r}'(y_0)$ be a $(d-1)$-dimensional ball in $\Omega\cap\{x_d=0\}$ touching $\partial\Omega\cap\{x_d=0\}$ at $x_0$. We set $R:=\sfrac{(d-1)}{\lambda}$, where $\lambda$ is the curvature of $\partial S_+$, and we consider the family of $d$-dimensional balls $B_{R}(y_0-te_d)$ for $t\in[\sqrt{R^2-r^2},R]$.
We decrease $t$ from $R$ towards $\sqrt{R^2-r^2}$ until $B_{R}(y_0-te_d)$ touches $\partial S_+$ from below at some point $z_t$. If $t >\sqrt{R^2-r^2}$, then the contact point $z_t$ has to lie in $\Omega\cap\{x_d>0\}$, but this leads to a contradiction by the maximum principle. 
Thus, the only possibility is that $t=\sqrt{R^2-r^2}$, which means that the ball $B_{R}(y_0-\sqrt{R^2-r^2}e_d)$ touches $\partial S_+$ from below at $x_0$. Let the planes $T_\Omega$ and $T_B$, and the set $T_+$, be the blow-ups in $x_0$ of $\partial\Omega$, $\partial B_{R}(y_0-\sqrt{R^2-r^2}e_d)$ and $S_+$. Since $S_-\subset \{x_d<0\}$ we know that $S_+$ is stationary with respect to vector fields pointing upwards; precisely, if $X\in C^1_c(\R^d;\R^d)$ is a vector field parallel to $\partial\Omega$ and such that $X\cdot e_d>0$, the optimality of $(S_+,S_-)$ gives that:
$$\mathcal G\big((Id+tX)(S_+),S_-\big)\ge \mathcal G\big(S_+,S_-)\quad\text{for every}\quad t>0.$$
Thus, if $\nu_{S_+}$ and $H_{S_+}$ are the normal to $\partial S_+$ pointing upwards and the generalized mean curvature of $S_+$, we have 
$$\int_{\widetilde\Omega\cap\partial S_+}H_{S_+}\cdot X\,d\H^{d-1}+\Lambda \int_{\widetilde\Omega \cap\partial S_+}\nu_{S_+}\cdot X\,d\H^{d-1}\ge 0.$$
In particular, by a change of coordinates to ensure that the vector field $X$ is always tangent along the blow-up sequence we obtain that 
$$\int_{\widetilde\Omega\cap\partial T_+}H_{T_+}\cdot X\,d\H^{d-1}\ge 0,$$
for every vector field $X\in C^1_c(\R^d;\R^d)$ parallel to $T_\Omega$ (that is, $X(T_\Omega)\subset T_\Omega$ in the blow-up limit) and pointing upwards (that is, $X\cdot e_d>0$). Since $T_+$ is contained in the wedge formed by $T_\Omega$ and $T_B$, this is a contradiction by deforming $T_+$ with a suitably chosen vector field as in \cref{lemma:con_bordo}.
A similar conclusion holds for $S_-$, thus $\partial S_- \cap (\partial\Omega\cap\{x_d = 0\}) = \emptyset$. In particular, this concludes the proof of claim {\rm (i)} of the lemma.
\\
\\
\noindent{\it Step 5. $K_+=\partial S_+\cap\widetilde\Omega$ and $K_-=\partial S_-\cap\widetilde \Omega$ are orthogonal to $\partial \Omega$ in viscosity sense.} The proof of claim {\rm (iii)} is immediate since both $S_+$ and $S_-$ are locally almost-minimizers for the relative perimeter in $\Omega\cap\{x_d>0\}$ and $\Omega\cap\{x_d<0\}$, respectively. Thus, for every vector fields $X\in C^{1}_c(\R^d;\R^d)$ such that $X\cdot \nu_\Omega$, the Euler-Lagrange equation ensures that both $K_+$ and $K_-$ meet $\pa \Omega$ with an angle of $\sfrac{\pi}{2}$.
\end{proof}

Finally, we can provide the uniform upper bound on the mean curvature $\lambda$ showing the following lemma.

\begin{lemma}
\label{l:upper_bound_curvatura_eps}
There are  $\eps_0 >0$ and $\Pi>0$ such that for all $\eps \leq \eps_0$ 
the following holds. 
Let $(K,E)$ be a minimizer of $\mathcal F_{\Lambda,\eps}$ given by \cref{prop:lunga!!!} with generalized mean curvature $\lambda$. Then
  \begin{equation}
    \label{eq:bound_curvature-upper}
   0<\lambda \leq \Pi\, \varepsilon.
\end{equation}
In particular ${\rm Vol}(E) = \eps$.
\end{lemma}

\begin{proof}
By \cref{thm:lambda>0}, we already know that $\lambda>0$ so we only need to prove the second inequality in \eqref{eq:bound_curvature-upper}. Let $B_r'(x_0)$ be a $(d-1)$-dimensional ball contained in $\widetilde\Omega\cap \{x_d=0\}$; for simplicity we set $x_0=0$. For every $\eps>0$ we choose $h>0$ such that 
$$\eps=\omega_{d-1}\left(\frac{r}{2}\right)^{d-1}\frac{h}{2},$$
and we define 
$$R:=\frac{1}{2h}(r^2+h^2),$$
so that the ball $B_R(-(R-h)e_d)$ contains the $(d-2)$-sphere $\partial B_r'$. Consider the family of $d$-dimensional balls $B_R(-te_d)$ for $t\in [R-h;R]$. When $t=R$, the ball $B_R(-R e_d)$ is contained in $\{x_d<0\}$. We start decreasing $t$ from $R$ to $R-h$. We notice that by construction 
$$B_{\sfrac{r}{2}}'\times\big(0,\sfrac{h}{2}\big)\subset B_R(-(R-h)e_d)\cap\{x_d>0\},$$
so that (by the choice of $h$)
$${\rm Vol}\Big(B_R(-(R-h)e_d)\cap\{x_d>0\}\Big)>\eps.$$
On the other hand 
$$\big(\widetilde\Omega\cap\{x_d>0\}\big)\setminus S_+\subset E\qquad\text{and so}\qquad {\rm Vol}\Big(\big( \widetilde\Omega\cap\{x_d>0\}\big)\setminus S_+\Big)<\eps.
$$
Thus, there is $t\in(R-h,R)$ for which $\partial B_R(-te_d)$ touches from below $\partial S_+$ at some point $x_0\in \partial B_R(-te_d)\cap \partial S_+$. Then $\partial S_+$ is a smooth graph in a neighborhood of $x_0$ and by the maximum principle, we have that 
$$\lambda\le \frac{d-1}{R}=(d-1)\frac{2h}{r^2+h^2}<\frac{(d-1)2h}{r^2}=\frac{d-1}{\omega_{d-1}}\left(\frac{2}{r}\right)^{d+1}\eps=:\Pi \,\eps,$$ 
which concludes the proof of \eqref{eq:bound_curvature-upper}. In order to prove the last claim, we take $\eps<\Lambda/\Pi$. If $E$ is the minimizer of $\mathcal F_{\Lambda,\eps}$ such that ${\rm Vol}(E)<\eps$, by  \cref{lemma:stime_curvature}, we have that its generalized mean curvature $\lambda$ is equal to $\Lambda$, but this contradicts \eqref{eq:bound_curvature-upper}. 
\end{proof}

\begin{remark}
\label{rem:nocollapsing}
    An immediate consequence of \cref{prop:lunga!!!}, \cref{lemma:pernonesseretroppolunghi} and \cref{l:upper_bound_curvatura_eps} is that (when $\eps>0$ is small enough) for any minimizer $(K,E)$ of $\FF$, we have
    $$
    K = \Omega \cap \pa E, \qquad  \dim_{H}(\Omega\cap (\partial E\setminus\partial^\ast E))\le d-8, \qquad\text{and}\qquad {\rm Vol}(E) = \eps,
    $$
    ensuring that both exterior and interior collapsing do not occur and we are able to provide an upper estimate for the generalized mean curvature $\lambda$ which is uniform in the volume $\eps$, which is exactly \eqref{eq:bound_curvature-upper}.
\end{remark}

\medskip

In the following section, we consider a specific case: a circle contained in the plane $\bp$, showing that the foliation is made of spherical caps of curvature $\lambda$ touching the boundary of the tubular neighbourhood $\pa W$ forming an angle of $\sfrac{\pi}{2}$ and varying with respect to the contained volume $\eps$. 
 
\section{A particular case: the foliation on a  circle contained in \texorpdfstring{$\bp$}{bp}}
\label{sec:spherical_caps}

In this section, we show that if the planar curve is a circle, then minimizers are spherical caps. We present the result in $\R^3$ but it can be easily generalized in higher dimensions. First, we show that fixing a certain amount of volume $\eps$ and imposing that it touches the boundary forming an angle of $\sfrac{\pi}{2}$, then there exists a unique set $\mathcal B_k$ 
 whose boundary $\partial \mathcal B_{\kappa}\cap\widetilde\Omega$ is a the union of two disjoint spherical caps $SC_{\kappa}^+$ and $SC_{\kappa}^-$.
Second, in order to prove that the minimizer $E$ constructed above is exactly $\mathcal B_{\kappa}$, we adapt the maximum principle argument from the previous section. 

\medskip

We consider the curve $\Gamma$ as 
$$\Gamma=\partial B_1\cap\{x_3=0\}=\{(\cos \varphi,\sin\varphi, 0)\ :\ \varphi\in[0,2\pi]\}.$$
Given $\delta>0$, $\eps>0$ and $\widetilde\Omega$ as in \cref{definition:tilde}, we define 
$$
\mathcal{B}_{\kappa}:=\widetilde\Omega \cap B_{\sfrac2\kappa}(0,0,z_C) \cap B_{\sfrac2\kappa}(0,0,-z_C),
$$
whose (free) boundary is the union of two disjoint spherical caps
$$
\widetilde\Omega\cap\partial \mathcal{B}_{\kappa}= SC_\kappa^+ \cup SC_\kappa^-\,,
$$
where $(z_C,\kappa)$ is the unique couple of parameters in $(-\infty,0)\times (0,+\infty)$ for which:
%
%
\begin{itemize}
\item $SC_\kappa^+ \cup SC_\kappa^-$ has mean curvature $\kappa$;
\item $\displaystyle{\rm Vol}(\mathcal B_\kappa)=\eps$;
\item $SC_\kappa^+ \cup SC_\kappa^-$ meets $\partial\Omega$ orthogonally.
\end{itemize}
Precisely, we describe $SC_\kappa^+$ (and similarly $SC_\kappa^-$) in terms of a parameter $\theta\in(0,\sfrac{\pi}{2})$ as follows:
$${\rm cl}(SC_\kappa^{+})\cap\partial\Omega:=\Big\{\Big((1-\delta\cos\theta)\cos \varphi,(1-\delta\cos\theta)\sin\varphi, \delta\sin\theta\Big)\ :\ \varphi\in[0,2\pi]\Big\}.$$
The orthogonality condition is satisfied when $\theta$, $\kappa$ and $z_C$ are such that 
$$z_C=\frac{-\delta+\cos\theta}{\sin\theta} \qquad \hbox{ and } \qquad \frac1{\kappa}=\frac{1-\delta\cos\theta}{\sin\theta},$$
while the volume constraint holds if $\kappa$ and $\theta$ are chosen such that 
\begin{equation*}
\frac{\varepsilon}{2} 
= \vartheta \left(\kappa^{-2} - \delta^2\right) +\frac{\delta^2 \sin^2\vartheta - \left(\cos\vartheta  - \delta\right)^2}{\sin\vartheta\,\cos\vartheta}.
\end{equation*}
This determines uniquely $z_C$, $\theta$ and $\kappa$ in terms of $\eps$ and $\delta$. 


\begin{theorem}
\label{th:circonferenza3D}
  Given $\Gamma=\mathbb{S}^1$, $\eps$ and $\delta$, let $(K,E)$ be a minimizer of $\F$ with ${\rm Vol}(E) = \eps$ and mean curvature $\lambda$. Then $E=\mathcal B_\lambda$, where $\mathcal B_\lambda$ is as above. In particular, $\Omega\cap\partial E=SC_\lambda^+ \cup SC_\lambda^-$.
\end{theorem}
\begin{proof}
By the results from the previous section \cref{rem:nocollapsing}, we have that 
$$\Omega\cap \partial E=K_+\cup K_-,$$
where $K_+\subset\{x_d>0\}$ and $K_-\subset\{x_d<0\}$ are sets of constant mean curvature $\lambda>0$. Without loss of generality, we can assume that 
\begin{equation}\label{e:volume-E-+}
{\rm Vol}(\{x_3>0\}\cap E)\ge \frac{\eps}{2}\,.
\end{equation}
Let now $\kappa$ be chosen as above in such a way that $\mathcal B_\kappa$ contains volume $\eps$. 
 We divide the proof in some steps.
    \\
  \\
   
    \noindent {\it Step 1. $\kappa < \lambda$ or we have that $\{x_3>0\}\cap E=\{x_3>0\}\cap \mathcal B_\kappa$ (and $\kappa=\lambda$), see for graphical representation \cref{fig:SCK_toccaE_sopra}.} Suppose that $\kappa\ge \lambda$. We consider translations $\partial B_{\sfrac2\kappa}(0,0,\overline t-z_C)$ of the sphere $\partial B_{\sfrac2\kappa}(0,0,-z_C)$. Let $\overline t$ be the smallest $t$ such that $B_{\sfrac2\kappa}(0,0,\overline t-z_C)$ contains $E$ and let $\overline x\in\overline\Omega$ be the contact point between $\partial B_{\sfrac2\kappa}(0,0,-z_C)$ and $\partial E$; clearly $\overline x\in \cl(K_+)$. 
   We notice that, by \eqref{e:volume-E-+}, $\overline t\ge 0$, where an equality holds if and only if $\{x_3>0\}\cap E=\{x_3>0\}\cap\mathcal B_\kappa$. 
   Then two cases can happen.
    \begin{enumerate}
        \item $\overline x$ lies on the boundary of the tubular neighbourhood: $\overline{x} \in \pa \Omega \cap \pa E \cap \partial B_{\sfrac2\kappa}(0,0,\overline t-z_C)$ . Then, by \cref{lemma:pernonesseretroppolunghi} (iii), the geometric construction of $SC_\kappa^+$ (the angle between the sphere $\partial B_{\sfrac2\kappa}(0,0, t-z_C)$ and $\partial\Omega$ is smaller than $\sfrac\pi2$ when $t>0$ and is greater than $\sfrac\pi2$ when $t<0$) and the assumption \eqref{e:volume-E-+}, we have that $\overline t=0$ and 
        $$
        \{x_3>0\}\cap \Omega \cap \pa E  = \Omega \cap SC_\kappa^{+},
        $$
        and necessarily the two curvatures must coincide $\kappa = \lambda$.
        \item $\overline{x}$ is an interior point: $ \overline{x} \in \Omega \cap \pa E \cap \partial B_{\sfrac2\kappa}(0,0,\overline t-z_C)$. In this case the conclusion depends on the relation between the curvature $\kappa$ of the spherical cap $ SC_\kappa^+$ and the mean curvature $\lambda$ of $\pa E$.
        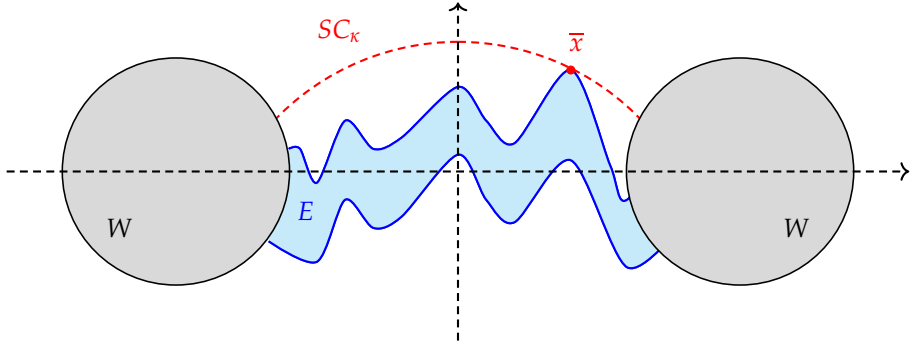
\begin{figure}[htbp]
		\centering
		\begin{tikzpicture}[rotate=0, scale= 0.75]
\draw[thick, red, densely dashed, name path=n] (0,2.3) arc [start angle=90, end angle = 141,x radius = 45mm, y radius = 45mm];
\draw[thick, red, densely dashed, name path=r] (0,2.3) arc [start angle=90, end angle = 39,x radius = 45mm, y radius = 45mm];
\draw [thick, color=blue, name path=s] plot [smooth] coordinates {(-3,0.4)  (-2.8,0.4) (-2.5,-0.2) (-2,0.9) (-1.5, 0.4)(-1,0.6) (0,1.5) (0.5, 0.9) (1,0.5) (2,1.8) (2.7, 0.1) (3,-0.5) (4,1)};
\draw [thick, color=blue, name path=t] plot [smooth] coordinates {(-3.355,-1.237) (-2.5,-1.6) (-2,-0.5) (-1.5, -1)(-1,-0.8) (0,0.3) (0.5, -0.5) (1,-0.9) (2,0.2) (3,-1.7) (4,-0.9)};
\begin{scope}[transparency group]
\tikzfillbetween[of=s and t] {color=cyan!20};
\end{scope}
\draw [thick, color=blue] plot [smooth] coordinates {(-3,0.4)  (-2.8,0.4) (-2.5,-0.2) (-2,0.9) (-1.5, 0.4)(-1,0.6) (0,1.5) (0.5, 0.9) (1,0.5) (2,1.8) (2.7, 0.1) (3,-0.5) (4,1)};
\draw [thick, color=blue] plot [smooth] coordinates {(-3.355,-1.237) (-2.5,-1.6) (-2,-0.5) (-1.5, -1)(-1,-0.8) (0,0.3) (0.5, -0.5) (1,-0.9) (2,0.2) (3,-1.7) (4,-0.9)};
\filldraw [red] (2,1.8) circle (2pt);
\draw[draw=none, fill=white] (5,0) circle [radius = 2cm];
\draw[draw=none, fill=white] (-5,0) circle [radius = 2cm];
\draw[very thick, color=black, name path=a] (-5,0) circle [radius=2];
\draw[very thick, color=black, name path=b] (5,0) circle [radius=2];
\tikzfillbetween[of=a and b] {color=gray!30};
\draw [color = red] node at (2.1,2.3) {$\overline{x}$};
\draw [color = red] node at (-2.1,2.5) {$SC_\kappa$};
\draw [color = black] node at (6,-1) {$W$};
\draw [color = black] node at (-6,-1) {$W$};
\draw [color = blue] node at (-2.7,-0.7) {$E$};
\draw[thick, densely dashed, color=black,->] (-8,0) -- (8,0);
\draw[thick, densely dashed, color=black,->] (0,-3) -- (0,3);
\end{tikzpicture}
		\caption{$SC^+_\kappa$ touches $\Omega \cap \pa E \cap \{x_3 > 0\}$ in $\overline{x}$ and the curvature are ordered $\kappa < \lambda$.}
		\label{fig:SCK_toccaE_sopra}
\end{figure}
      Since $\kappa \geq \lambda$, by the maximum principle, we get
        $$
        \Omega \cap \partial B_{\sfrac2\kappa}(0,0,\overline t-z_C) \subset \Omega \cap \pa E.
        $$
        Thus, $\lambda = \kappa$. Hence, since both $\Omega \cap \pa E$ and $\Omega \cap SC_\kappa$ are $C^{1,\alpha}$ graphs with the same curvature $\kappa = \lambda$ and, using {\rm (iii)} of \cref{lemma:pernonesseretroppolunghi}, they meet $\pa \Omega$ with an angle of $\sfrac{\pi}{2}$, we have that ${\rm Vol}(E\cap\{x_3>0\}) = {\rm Vol}(\mathcal B_\kappa\cap\{x_3>0\}) = \sfrac{\eps}{2}$. This implies that $\overline t=0$ and 
        $$
        \Omega \cap  SC^+_\kappa = \Omega \cap \pa E \cap \{x_3 > 0\},
        $$
        and, in particular, we also deduce that $\Omega \cap \pa E \cap \{x_3 > 0\}$ is a normal graph over the plane $\{x_3 =0\}$.
    \end{enumerate}
    \noindent {\it Step 2. $\kappa > \lambda$ or we have that $\{x_3<0\}\cap E=\{x_3<0\}\cap \mathcal B_\kappa$ (and $\kappa=\lambda$).}
   We suppose that $\kappa\le \lambda$ and we move the spherical cap $SC_\kappa^-$ from $+\infty$ along the $x_3$-axis. Let $\underline t$ be the largest $t$ such that $(te_3+SC_\kappa^-)$ lies above $\cl{(K_-)}$. The stopping criterion assures the existence of a touching point from above $P \in \overline{\Omega} \cap \cl{(K_-)}\cap (\underline te_3+SC_\kappa^-)$. Precisely, as a consequence of \cref{prop:lunga!!!} and \cref{lemma:pernonesseretroppolunghi}, the point $P$ lies in the half-space $\{x_3 <0\}$. Two cases can happen.
\begin{itemize}
    \item $P$ is an interior point: $P \in \Omega \cap K_-\cap (\underline te_3+SC_\kappa^-)$. In this case, $\underline t = 0$ and by the maximum principle 
    $$
        \Omega \cap SC^-_\kappa = \Omega \cap \pa E \cap \{x_3 < 0\} \qquad \hbox{ and } \qquad \kappa = \lambda.
        $$
    \item $P$ is a boundary point: $P \in \pa \Omega \cap\cl(K_-) \cap (\underline te_3+SC_\kappa^-)$. As in Case 1 of Step 1, by combining {\rm (iii)} of \cref{lemma:pernonesseretroppolunghi} and the geometrical construction of $SC_\kappa^-$, we get that 
    $$
        \Omega \cap \pa E\cap \{x_3 <0\} = \Omega \cap SC_\kappa^{-},
        $$
        and necessarily the two curvatures must coincide $\kappa = \lambda$.
\end{itemize}
We next notice that by Step 1 and Step 2, we have that necessarily $\Omega \cap \pa E \cap \{x_3 >0\}$ coincides with $\Omega \cap SC_\kappa^+$ or $\Omega \cap \pa E \cap \{x_3 <0\}$ coincides with $\Omega \cap SC_\kappa^-$. In both cases we get
$${\rm Vol}(\{x_3>0\}\cap E)= {\rm Vol}(\{x_3<0\}\cap E)= \frac{\eps}{2}\,.$$
In particular, applying Step 1 and Step 2 to $E$ and its reflection (with respect to the plane $\{x_3=0\}$), we get the claim. 
%
\end{proof}

Thus, in the physical dimension $d = 3$ and for a circle $\Gamma = \pa B_1(0) \cap \{x_3 = 0\}$, we are the able to geometrical characterize the foliation.

\begin{remark}
\label{cor:true_minimizer}
In the physical dimension $d = 3$, the unique minimizer $E^\eps$ of $\F$ is given by
 $$
 E^\eps = \mathcal{B}_{\lambda_\eps}:=\Omega \cap B_{\sfrac{2}\lambda_\eps}(0,0,z_C) \cap B_{\sfrac{2}{\lambda_\eps}}(0,0,-z_C).
 $$
Moreover, as a function of $\eps$, the sets $ \mathcal{B}_{\lambda_\eps}$ are increasing with respect to inclusion, their boundaries 
$$\Omega\cap \pa \mathcal{B}_{\lambda_\eps} = SC_{\lambda_\eps}^+\cup SC_{\lambda_\eps}^-$$ 
are disjoint spherical caps and, together with the disk $D_0=\{x_3=0\}\cap \Omega$, form a foliation of $\overline{\Omega}$ in a neighborhood of $D_0$.
\end{remark}

\section{Epsilon - regularity and existence of classical solutions}
\label{sec:regolarita}

In this section, we provide a fundamental ingredient to construct the foliation of $\overline{\Omega}$ for a general $\Gamma$, smooth embedding of the $d-2$-dimensional sphere $\mathbb{S}^{d-2}$ in the hyperplane $\bp$: $\Omega \cap \pa E$ is a $C^{1, \alpha}$ manifold, the set of finite perimeter $E$ is constrained to contain a fixed amount of volume $\eps$ and $\Omega \cap \pa E$ meets othogonally the $\delta$-tubular neighbourhood $W$.

The main theorem we are going to prove is the following one.

\begin{theorem}\label{t:reg-C1alpha}
Let $S_+$, $S_-$, and $\widetilde \Omega$ be as in \cref{prop:lunga!!!}. Then, $\widetilde\Omega\cap\partial S_+$ and $\widetilde\Omega\cap\partial S_-$ are $C^{1,\alpha}$ graphs up to the boundary. Precisely, there exists $\e_0>0$ (depending on $\Gamma$, $\delta$ and $\mathcal C$) such that for every fixed $\varepsilon \in \left( 0, \e_0\right]$, there are two closed domains $D_{\pm} \subset \R^{d-1}$ with $C^{1,\alpha}$ boundaries and functions $u_{\pm}\in C^{1,\alpha}(D_{\pm})$ with 
$$\pm u_\pm (x')>0\qquad\text{for every}\qquad x'\in D_\pm,$$
and
\begin{equation*}
   \cl(\widetilde\Omega\cap\partial S_\pm) = \big\{(x',u_\pm(x'))\ :\ x'\in D_{\pm}\big\}.
\end{equation*}
Furthermore
\begin{equation*}
\lim_{\eps\to0}\Big(\|u_\pm\|_{L^\infty(D_{\pm})}+\|\nabla u_\pm\|_{L^\infty(D_{\pm})}\Big)=0.
\end{equation*}
\end{theorem}
\begin{proof}
By \cref{lemma:Estanellastriscia}, 
for every $\sigma>0$, we can find $\eps$ small enough such that $\widetilde\Omega\cap\partial S_+$ and $\widetilde\Omega\cap\partial S_-$ are contained in the strip $\{(x',x_d)\in \overline \Omega\ :\ -\sigma<x_d<\sigma\}$. Moreover, by \eqref{e:strip-upper} and the fact that both $\widetilde\Omega\cap\partial S_+$ and $\widetilde\Omega\cap\partial S_-$ are $\C$-spanning, we get that 
$$\H^{d-1}(D_0)\le \H^{d-1}(\Omega\cap \partial S_\pm)\le \H^{d-1}(D_0)+C\eps^2.$$
In particular, for every $\sigma>0$, we can choose $\eps_0$ such that, in every ball $B_r(x_0)$ centered in $x_0\in D_0$ of radius $r=\delta$, $\widetilde\Omega\cap\partial S_+$ and $\widetilde\Omega\cap\partial S_-$ are contained in a strip of width $\sigma\delta$; they have $\H^{d-1}$ measure at most $(1+\sigma)\H^{d-1}(D_0\cap B_\delta(x_0))$; and they have mean curvature $\lambda$ bounded by $\Pi\eps$ (see \cref{l:upper_bound_curvatura_eps}). Thus, the $C^{1,\alpha}$-regularity of $\widetilde\Omega\cap\partial S_\pm$ follows from the Allard's boundary regularity theorem \cite{Allard_bordo, GruterJost} (see \cref{th:eps_reg} below) applied to $\widetilde\Omega\cap\pa S_\pm$. 
\end{proof}

\begin{theorem}[Epsilon-regularity theorem \cite{Allard_bordo, GruterJost}]
\label{th:eps_reg}
Let $F$ and $\Omega$ be as follows.
\begin{itemize}
    \item $\Omega$ is an open set in $\R^d$ with $C^{1,\sigma}$ boundary $\partial\Omega$ for some $\sigma>0$;
    \item $F$ satisfies the following Euler-Lagrange equation
    $$\int_{\pa^*F}\lambda_F\,X\cdot\nu_F\,d\H^{d-1}=\int_{\pa^*F}\Div^K\,X\,d\H^{d-1}\,,
    $$
    for every $X\in C^1_c(\R^{d};\R^{d})$ with $X\cdot\nu_\Om=0$ on $\pa\Om$, where $\nu_F$ and $\nu_\Omega$ are the exterior normals to $F$ and $\Omega$; and the mean curvature $\lambda_F:\partial F\cap\Omega\to\R$ is uniformly bounded.
\end{itemize}
\noindent Then, for some $x_0 \in \pa \Omega\cap\pa F$, there exist a dimensional constant $\overline{\tau}>0$ and a radius $r_0$ such that the following holds. If for some $r\leq r_0$ and $x_0\in\pa\Omega$ we have: 
\begin{equation}\label{e:eps-reg-density-condition}
\H^{d-1}(\partial^\ast F\cap B_r(x_0)\cap\Omega)\le \frac{(1+\overline\tau)d\omega_dr^{d-1}}{2},
\end{equation}
 and $\pa F$ is $r \overline{\tau}$-flat in $B_r(x_0)$ along a direction $\nu\in\partial B_1$ with $\nu_\Omega(x_0) \cdot \nu = 0$, namely in $B_r(x_0) \cap \Omega$ it holds
 $$
  \left\{(x-x_0) \cdot  \nu + r a >0\right\}\subset F \subset \left\{(x-x_0) \cdot \nu + r b>0\right\},
 $$
 where $b-a \le \overline\tau$. Then, there exists $\alpha >0$ such that, in $B_{\sfrac{r}{2}}(x_0)\cap\overline\Omega$, $\partial F$ is $C^{1,\alpha}$ regular manifold with boundary.
\end{theorem}

In the rest of this section, we give an alternative proof to 
 \cref{th:eps_reg} based on the viscosity approach from \cite{DeSilva}, which can be easily extended to more general boundary conditions. 
 A fundamental ingredient for proving \cref{th:eps_reg} is the Partial Harnack inequality which can be stated as follows, the proof being an adaptation of \cite{DeSilva} (for the details we refer to \cite{DeSilva} and \cite{velichkov2023regularity}).

\begin{lemma}[Partial Harnack inequality]
    \label{lemma:PH_inequality}
    Let $F$ and $\Omega$ be as in \cref{th:eps_reg}.
Then, there exist dimensional constants $r_0>0$, $\overline{\tau}>0$, $\zeta\in\left(\frac{1}{2},1\right)$ and $c \in (0,1)$ such that for every $x_0 \in \overline{\Omega}$ the following holds. 
Suppose that for some $\nu \in \pa B_1$ and $r<r_0$ we have that  for all $x\in B_r(x_0)\cap \partial\Omega$ and $\widetilde{x}\in B_r(x_0)\cap \Omega\cap \partial F$ 
\begin{equation}
    \label{eq:hyp_PH}
    \tau \leq \overline{\tau}, \qquad \hbox{ and }\qquad \abs{\nu_\Omega(x) \cdot \nu} \leq {\tau}^{1+\zeta} \qquad \hbox{ and }\qquad \abs{\lambda_F(\widetilde{x})} \leq {\tau}^{1+ \zeta}.
\end{equation}
Suppose, moreover, that there is $\overline x\in \Omega$ such that 
\begin{equation}
    \label{eq:overline-x}
    B_{\sfrac{3r}{4}}(\overline{x})\subset B_r(x_0)\ ,\qquad B_{\sfrac{r}{20}}(\overline{x})\subset \Omega\qquad\text{and}\qquad \H^{d-1}(\partial F\cap B_{\sfrac{r}{20}}(\overline{x}))\le (1+\overline\tau)d\omega_d\left(\sfrac{r}{20}\right)^{d-1}.
    \end{equation}
Finally, suppose that $\pa F$ is $r \tau$-flat in $B_r(x_0)$ along the $\nu$-direction, namely in $B_r(x_0) \cap \Omega$ it holds
 $$
 \left\{(x-x_0) \cdot  \nu + r a >0\right\}\subset F \subset \left\{(x-x_0) \cdot \nu + r b>0\right\},
 $$
 where $b-a = \tau$. Then, there exist two real numbers $\overline{a}, \overline{b}$ with $a \leq \overline{a} < \overline{b} \leq b$ with
  $$
    |\overline{b}-\overline{a}|\leq (1-c)\abs{b-a},
    $$
     such that in $B_{\sfrac{r}{50}}(x_0) \cap \Omega $
    \begin{align}
    \label{eq:tesi_PH}
   \Big\{(x-x_0) \cdot  \nu + r \overline{a} >0\Big\}\subset F \subset \Big\{(x-x_0) \cdot \nu + r \overline{b}>0\Big\}.
    \end{align} 
\end{lemma}
\begin{proof}
First of all, by scaling, we can assume that $r=1$. In addition, we can suppose that $x_0$ is at distance at most $\sfrac{1}{10}$ from $\partial\Omega$. Indeed, if  $x_0\in {\rm int}\, \Omega$ is at larger distance, then the claim follows from the classical (interior) Allard's theorem (for a viscous version we refer to \cite{de2023short}). 
Moreover, it is not restrictive to choose $\nu = e_d$ since there exists a rotation $S:SO(\R^d)\to SO(\R^d)$ mapping the direction $\nu$ in $e_d$, and we can assume that $\abs{\nu_\Omega(x) - e_1}\leq \tau$ for all $x\in B_r(x_0)\cap \partial\Omega$. 

\medskip

Let $\overline{x}$ be as in \eqref{eq:overline-x}. We consider the following function on the hyperplane $\Sigma:= \{x \cdot \nu = 0\}$
\begin{equation}
\label{eq:w}
    w: \Sigma \to \R \qquad \hbox{ such that } \qquad \left\{
\begin{aligned}
    &w(y) = 1 &&\forall\, y \in \Sigma\cap B_{\sfrac{1}{50}}(\overline{x}),\\
    &w(y) = 0 &&\forall\, y \in \Sigma \setminus B_{\sfrac{3}{4}}(\overline{x}),\\
    &w(y) = \frac{C}{\abs{y- \overline{x}}^{d-1}}  &&\forall\, y \in \Sigma\cap B_{\sfrac{3}{4}}(\overline{x}) \setminus B_{\sfrac{1}{50}}(\overline{x}),
 \end{aligned}
\right.
\end{equation}
with $C$ a positive constant. Hence, $w$ is non a vanishing function inside the ball $B_{\sfrac{3}{4}}(\overline{x})$ and outside $\overline{B}_{\sfrac{1}{50}}(\overline{x})$ and in $B_{\sfrac{3}{4}}(\overline{x}) \setminus \overline{B}_{\sfrac{1}{50}}(\overline{x})$ it has the following properties:
\begin{itemize}
    \item the minimal surface operator has a sign, i.e. for all $y \in \Sigma\cap B_{\sfrac{3}{4}}(\overline{x}) \setminus B_{\sfrac{1}{50}}(\overline{x})$
    \begin{equation}
        \label{eq:operator_w_min_surface}
        \diver\left(\frac{\nabla w(y)}{\sqrt{1+ \abs{\nabla w(y)}^2}}\right) >0,
    \end{equation}
    since the constant $C$ is strictly positive;
    \item by construction, for all $y \in \pa \Omega \cap\Sigma\cap B_{\sfrac{3}{4}}(\overline{x}) \setminus B_{\sfrac{1}{50}}(\overline{x})$, we have 
    \begin{equation}
        \label{eq:boundary}
        \nabla w(y) \cdot e_1 \geq C >0,
    \end{equation}
    where $C$ is a positive dimensional constant.
\end{itemize}
Choosing $\overline{\tau}$ small enough, by the Allard's theorem in the ball $B_{\sfrac{1}{10}}(\overline{x})$, we have that $F$ is the subgraph of a function $u:\Sigma\cap B_{\sfrac{1}{10}}(\overline{x})\to\R$ and that $\partial F$ is its graph. Moreover, by the flatness condition on $\partial F$ and the density estimate, we have 
$$
a\le u(y)\le b\quad\text{for every}\quad y\in \Sigma\cap B_{\sfrac{1}{10}}(\overline x).
$$
Without loss of generality, we can suppose that $\displaystyle u(\overline x)\ge \frac{a+b}{2}=a+\frac\tau2$. 
Applying again the Allard's theorem, we obtain 
\begin{equation}
    \label{eq:Harnack}
    u(y) - a \geq c \tau \qquad \hbox{on  } B_{\sfrac{1}{50}}(\overline{x}),
\end{equation}
where $c$ is a dimensional constant. Taking the family of functions $\{v_t\}_{t \in [0,1)}$ given by
\begin{equation}
    \label{eq:v_t}
    v_t(y) = a + c \tau w(y)+ \left(t-1\right) c \tau,
\end{equation}
we want to show that the subgraphs of $v_t:\Sigma\cap B_{\sfrac{3}{4}}(\overline x)\to\R$ are contained in $F$  for all $t \in [0,1)$.
By contradiction, suppose that for some $t\in[0,1)$ the graph of $v_t$ touches $\partial F$ at some point $y+\nu v_t(y)$ with $y \in \Sigma$. Since $y\notin \overline B_{\sfrac{1}{50}}(\overline x)$ and $y\notin B_1\setminus\overline B_{\sfrac{3}{4}}(\overline x)$, we have that $v_t$ is smooth at $y$ and so there is a smooth function $u:\Sigma\to\R$ defined in a neighborhood of $y$ and such that $\partial F$ is precisely the graph of $u$. 

To conclude the proof, it is enough to exclude that  $y\in \Sigma\cap B_{\sfrac{3}{4}}(\overline{x}) \setminus \overline{B}_{\sfrac{1}{50}}(\overline{x})$. Thus, we analyze the following two cases: $y+ u(y)\nu \in {\rm int} (\Omega)$ and $y+ u(y)\nu\in \partial\Omega$.
In order to rule out the first case, we notice that $v_t$ is a strict subsolution in the internal part, that is
\begin{equation}
    \label{eq:Deltavt}
    \begin{aligned}
        \diver\left(\frac{\nabla v_t}{\sqrt{1+ \abs{\nabla v_t}^2}}\right) - \lambda_F&= \diver\left(\frac{\nabla \left(a + c \tau w(y)  + (t-1) c\tau\right)}{\sqrt{1+ \abs{\nabla \left(a + c \tau w(y)  + (t-1) c \tau\right)}^2}}\right) - \lambda_F\\
        &=c\tau \diver\left(\frac{\nabla w(y)}{\sqrt{1+ \abs{c \tau\nabla w(y)}^2}}\right) - \lambda_F > C,
    \end{aligned}
\end{equation}
having used \eqref{eq:hyp_PH}, \eqref{eq:operator_w_min_surface} and the fact that $C$ is a positive dimensional constant.  
In order to exclude the second case, we compute
$$
\nabla v_t(y)=\nabla \left(p + c \tau w(y)  + (t-1) c \tau\right) = \tau c \nabla w(y).
$$
Thus, the boundary condition yields
\begin{equation}
    \label{eq:boundary_cond}
    \frac{\left(- c \tau \nabla w(y), 1\right)}{\sqrt{1+ c^2 \tau^2\abs{\nabla w(y)}^2}}\cdot \nu_\Omega(y, w(y))\geq 0.
\end{equation}
Precisely \eqref{eq:boundary_cond} simplifies into 
$$
- c \tau \nabla w(y) \cdot e_1 +\left( \sum_{l= 2}^{d-1}(-c \tau \nabla w(y) \cdot e_l)e_l+e_d\right)\cdot( \nu_\Omega(y,v_t(y)) - e_1) \geq 0,
$$
which leads to a contradiction for $\overline\tau$ small enough (such that $\tau^\zeta$ 
 is smaller than the constant $C$ from \eqref{eq:boundary}).
\end{proof}

Finally, we are in position to prove the main result of this section: the improvement of flatness lemma for viscosity solutions. We are going to show the result considering points $x_0 \in \pa \Omega$ since for all internal points Allard's regularity theorem holds true.

\begin{lemma}[Improvement of flatness]
\label{lemma:IOF}
Let $F$ and $\Omega$ as in \cref{th:eps_reg}. 
Suppose that for all $x_0 \in \pa \Omega\cap\pa F$ and for all parameters $\frac{1}{2}< \zeta < 1$, there exist a dimensional constant $\overline{\tau}>0$ and a radius $r_0>0$ such that the following holds.
Suppose that for some $\nu \in \pa B_1$ and $r<r_0$ we have that for all $x\in B_r(x_0)\cap \partial\Omega$ and $\widetilde{x}\in B_r(x_0)\cap \Omega\cap \partial F$
the density bound \eqref{e:eps-reg-density-condition} holds, 
\begin{equation}
    \label{eq:hyp_IOF}
    \tau \leq \overline{\tau}, \qquad  \nu_\Omega(x_0) \cdot \nu =0, \qquad  \abs{\nu_\Omega(x) - \nu_\Omega(x_0)} \leq {\tau}^{1+\zeta} \qquad \hbox{ and }\qquad \abs{\lambda_F(\widetilde{x})} \leq {\tau}^{1+ \zeta}.
\end{equation}
Moreover, if $\pa F$ is $r \tau$-flat in $B_r(x_0)$  along the $\nu$-direction, namely in $B_r(x_0) \cap \Omega$ it holds
 $$
 \left\{(x-x_0) \cdot  \nu + r a >0\right\}\subset F \subset \left\{(x-x_0) \cdot \nu + r b>0\right\},
 $$
 where $b-a = \tau$, then there exists $\rho \in \left(\sfrac{1}{2}\,,\,1\right)$ such that in $B_{\rho r}(x_0)$, $\pa F$ is $\frac{\tau}{2}$-flat along $\nu'$-direction, namely 
 $$
 \left\{(x-x_0) \cdot  \nu' - \frac{ \rho r \tau}{2} >0\right\}\subset F \subset \left\{(x-x_0) \cdot \nu' +  \frac{ \rho r \tau}{2} >0\right\},
 $$
 and 
    $$
    \abs{\nu - \nu'} \leq C \tau, \qquad \hbox{ and } \qquad \nu'\cdot \nu_{\Omega}(x_0)=\beta(x_0).
    $$
    for some constant $C$.
\end{lemma}

\begin{proof}
First of all, as in the proof of \cref{lemma:PH_inequality},
it is not restrictive to choose $\nu = e_d$ and $\nu_\Omega(x_0)=e_{d-1}$ since there exists a rotation $S:SO(\R^d)\to SO(\R^d)$ such that $S\nu = e_d$ and $S\nu_\Omega(x_0) = e_{d-1}$. Moreover, by scaling and translation, we can assume that $r=1$ and $x_0=0$.

Then, we prove the result by compactness. Let $F_j$ be a sequence of sets $\tau_j$-flat in $B_1$ along the $e_d$-direction, with mean curvature $\lambda_{F_j}$ in $\Omega_j$ such that for all $x\in B_1\cap \partial\Omega$ and $\widetilde{x}\in B_1\cap \Omega\cap \partial F$ it holds
\begin{equation}
    \label{eq:hyp_IOF_assurdo}
    \tau_j\to0 , \qquad  \nu_{\Omega_j}(0) =e_{d-1}, \qquad  \abs{\nu_{\Omega_j}(x) - e_{d-1}} \leq {\tau_j}^{1+\zeta} \qquad \hbox{ and }\qquad \abs{\lambda_{F_j}(\widetilde{x})} \leq {\tau_j}^{1+ \zeta},
\end{equation}
and \eqref{e:eps-reg-density-condition} is valid for all $\pa F_j$ and $\Omega_j$.
By the monotonicity formula, for all $x_0 \in \pa F_j \cap \pa \Omega_j\cap B_{1/2}$, the density of $\partial F_j$ remains bounded at every scale. Precisely
\begin{equation}\label{e:F-jmonotonicityfinalversion}
\H^{d-1}(\partial F_j\cap B_{r_k}(x_0)\cap\Omega_j)\le (1+\overline\tau)\frac{d\omega_dr_k^{d-1}}{2}
\end{equation} 
for every $r_k= 50^{-k}$, where $\overline\tau$ is the constant from \cref{lemma:PH_inequality}. Moreover, if $\pa F_j$ is $\overline\tau r_k$-flat in $B_{r_k}(x_0)$ and \eqref{e:F-jmonotonicityfinalversion} holds, then  \eqref{eq:overline-x} holds with $r=r_k$. Thus, we can apply \cref{lemma:PH_inequality} for every $x_0\in \partial\Omega$ and every $r_k=50^{-k}$ with $k$ such that $\tau_j50^k\le \overline\tau$.
In a similar way, for internal points $x_0\in\Omega$, the thesis is exactly the same changing the considered monotonicity formula.
Hence, if we consider the rescalings
$$\widetilde F_j:=\Phi_{\tau_j}(F_j)\qquad\text{where}\qquad\Phi_{\tau_j}(x',x_d)=(x',\tau_j^{-1}x_d),$$
we have that $\partial \widetilde F_j$ converge in the Hausdorff distance to the graph of an H\"older continuous function
$$\widetilde{u}:B_{\sfrac{1}{2}}' \cap \{x_{d-1}\ge 0\}\to\R\ ,\qquad -1\le \widetilde{u}\le 1, \qquad \widetilde u(0)=0,
    $$ 
    that solves inside
    $$
    \Delta \widetilde{u}= 0 \quad \hbox{in}\quad B_{\sfrac{1}{2}}' \cap \{x_{d-1}> 0\},
    $$
where we remarked that we do not have any right-hand side by the chosen rescaling of the mean curvature $\lambda_F$ in \eqref{eq:hyp_IOF_assurdo}. 
    While, on the boundary, we derive the condition using the hypothesis on $\nu_{\Omega_j}$ in \eqref{eq:hyp_IOF_assurdo}.
   Hence, we get
    $$
   \pa_{d-1} \widetilde{u} = 0 \quad \hbox{on}\quad B_{\sfrac{1}{2}}' \cap \{x_{d-1}= 0\}.
    $$
    Hence, $\widetilde{u}$ solves the following elliptic PDE problem in viscosity sense
    \begin{equation}
        \label{eq:prob_linearizzato}
        \left\{
\begin{aligned}
&\Delta \widetilde{u}= 0 &&\hbox{in}\quad B_{\sfrac{1}{2}}' \cap \{x_{d-1}> 0\}\\
    & \pa_{d-1} \widetilde{u} = 0 &&\hbox{on}\quad B_{\sfrac{1}{2}}' \cap \{x_{d-1}= 0\}\\
    & \norm{\widetilde{u}}_{L^\infty(B'_{\sfrac{1}{2}})} \leq 1,
\end{aligned}
        \right.
    \end{equation}
   that is, $\widetilde{u}$ is harmonic in $B_{\sfrac{1}{2}}' \cap \{x_{d-1}> 0\}$ and if $P$ is a polynomial touching $\widetilde{u}$ from below (resp. from above) at a point $x_0 \in B_{\sfrac{1}{2}}' \cap \{x_{d-1}= 0\}$ then $\pa_{d-1} P(x_0) \leq 0$ (resp. $\pa_{d-1} P(x_0) \geq 0$). By  \cite[Lemma 2.6]{DeSilva}, we get that $\widetilde{u}$ is $C^\infty$ smooth in $B_{\sfrac{1}{2}}'\cap\{x_{d-1}\ge 0\}$ and solves the above elliptic PDE in the classical sense.  
      By classical theory of elliptic PDE \cite{gilbarg1977elliptic}, there exists a positive dimensional constant $C_d >0$ such that
    \begin{equation}
    \label{eq:Schauder}
         \norm{\nabla \widetilde{u}}_{L^\infty\left(B_{\sfrac{1}{2}}' \cap \{x_{d-1}\geq 0\}\right)}+\norm{\tens{D}^2 \widetilde{u}}_{L^\infty\left(B_{\sfrac{1}{2}}' \cap \{x_{d-1}\geq 0\}\right)} \leq C_d,
    \end{equation}
    where $\tens{D}^2$ is the hessian.
  Thus, we have the following Taylor expansion $\forall \rho<\frac12$ and $ \forall\, x \in B_{\rho}' \cap \{x_d\ge 0\}$ 
    $$
    \begin{aligned}
\abs{\widetilde{u}(x) - \nabla \widetilde{u}(0) \cdot x} &\leq  C_d \abs{x}^2 \leq C_d \rho^2,
    \end{aligned}
    $$
    which can be written in the terms of the graph of $\widetilde u$ in  $B_{\sfrac{1}{2}}'\times(-1,1)$ as
    $$
\text{graph}(\widetilde u)\subset \{ \nabla\widetilde u(0)\cdot x'-C_d\rho^2<x_d<\nabla\widetilde u(0)\cdot x'+C_d\rho^2\}.
$$
   Moreover, by the Hausdorff convergence of $\partial F_j$, we have that along the sequence $F_j$
$$
\left\{(-\tau_j\nabla \widetilde u(0),1)\cdot x - \frac{C_d\rho^2}{2}\right\}\subset F_j\subset\left\{(-\tau_j\nabla \widetilde u(0),1)\cdot x + \frac{C_d\rho^2}{2}\right\}.
$$
  Thus, the new direction along which we are flatter is given by
    \begin{equation}
    \nu'_j := \frac{1}{\sqrt{1+\tau_j^2|\nabla \widetilde u(0)|^2}}(-\tau_j \nabla \widetilde{u}(0),1),
\end{equation}
which is still orthogonal to $e_{d-1}$, since $\partial_{d-1}\widetilde u(0)=0$.
Finally, since $\abs{\nabla \widetilde u(0)}\le C_d$, we get that 
$$
|\nu_j-\nu_j'| \leq C_d {\tau_j}\ ,
$$
which concludes the proof. 
\end{proof}

Finally, in order to prove \cref{th:eps_reg} in each point $x \in \pa F \cap \pa \Omega$, we iterate the improvement of flatness (\cref{lemma:IOF})). The following remark is useful in such a direction.
\begin{remark}[Improvement of flatness in terms of the rescalings of $F$]
\label{rem:IOFinB1}
 
The statement of \cref{lemma:IOF} can be reformulated as follows. Let $F$ and $\Omega$ as in \cref{th:eps_reg}. 
Suppose that for all $x_0 \in \pa \Omega\cap\pa F$ and for all parameters $\sfrac{1}{2}< \zeta < 1$, there are dimensional constants $k \in (0,1)$, $\sigma \in (0,1)$, $C >0$, $\overline{\tau}>0$ and a radius $r_0>0$ 
such that if for some $\nu \in \pa B_1$, for all $r \leq r_0$ and for all $x \in B_{1}(x_0)$ we have
\begin{equation*}
    \tau \leq \overline{\tau}, \qquad  \nu_\Omega(x_0) \cdot \nu =0, \qquad  \abs{\nu_\Omega(x) - \nu_\Omega(x_0)} \leq {\tau}^{1+\zeta} \qquad \hbox{ and }\qquad \abs{\lambda_F(x)} \leq {\tau}^{1+ \zeta},
\end{equation*}
and
\[
\H^{d-1}\left(\left(\frac{\partial F - x_0}{r}\right)\cap B_1\cap\Omega\right)\le \frac{(1+\overline\tau)d\omega_d}{2}.
\] 
Then, if $\pa F$ satisfies the following condition along the $\nu$-direction in $B_1$
 $$
 \left\{x\cdot  \nu - \tau >0\right\}\subset F_{r,x_0} \subset \left\{x\cdot \nu +\tau>0\right\},
 $$
 where for any $r>0$ with $r \leq r_0$
 $$
 F_{r,x_0}:= \frac{F - x_0}{r},
 $$
there exists a new direction $\nu'\in \pa B_1$ such that for all $x \in B_1\cap \Omega$ it holds 
 $$
 \left\{x\cdot  \nu' - \sigma \tau >0\right\}\subset F_{kr, x_0} \subset \left\{x\cdot \nu' +  \sigma \tau  >0\right\},
 $$
with
    $$
    \abs{\nu - \nu'} \leq C \tau, \qquad \hbox{ and } \qquad \nu'\cdot \nu_{\Omega}(x_0)=0.
    $$
\end{remark}

By replacing the statement of \cref{lemma:IOF} with \cref{rem:IOFinB1}, which gives the same thesis of the improvement of flatness but at scale $1$, there are just two missed steps to conclude \cref{th:eps_reg}: uniqueness of the blow-up of $\pa F$ with its right decay and the use of classical interior Allard's regularity theorem to show that $\pa F$ touches $\pa \Omega$ in a $C^{1, \sigma}$ way. Since these two points are nothing but a variation of \cite[Chapter 8]{velichkov2023regularity}, we condense everything in the same proof.

\begin{proof}[Proof of Theorem \ref{th:eps_reg}]
We divide the proof in two steps.
\\
\\
{\em Step 1. We study the blow-up convergence.} Assume that $\pa F_{r,x_0}$ is $\tau$-flat in $B_1$ for some $x_0 \in \pa \Omega$, $\tau \leq \overline{\tau}$ and $r\leq r_0$ given by the density estimate. Then, $\pa F_{r, x_0}$ is $\sigma \tau$-flat in $B_{\sfrac{1}{2}}$ with $\sigma <1$. Using the same notation as in \cref{rem:IOFinB1}, we denote with $F_j:= F_{r_j, x_0}$ the blow-up sequence of the sets and $\displaystyle r_j:= \frac{k^j}{2}$. Thus, iterating the Improvement of flatness Lemma in $B_1$, 
 there exist a sequence of directions $\nu_j \in \pa B_1$ such that
    $$
    \abs{\nu_j -\nu_{j+1}} \leq C \overline{\tau} \sigma^j \qquad \text{for all}\qquad j \in \N\,.
    $$
    In particular, for $1 \leq j < i$, we have
    $$
    \begin{aligned}
        \abs{\nu_j- \nu_i} &\leq \sum_{k= j}^{i-1} \abs{\nu_k -\nu_{k+1}} \leq \sum_{k=j}^{i-1} C \overline{\tau} \sigma^k\leq C \overline{\tau} \sum_{k=j}^{\infty} \sigma^k = C \overline{\tau} \frac{\sigma^j}{1-\sigma}.
    \end{aligned}
    $$
    Hence, since the series converge, there exists
    $$
    \nu = \lim_{j \to \infty} \nu_j,
    $$
    with $\nu \in \pa B_1$. Finally, we end up with
    $$
    \abs{x \cdot \nu- \left(x\cdot \nu_j \pm \overline{\tau} \sigma^j\right)} \leq \left(1 + \frac{C}{1-\sigma}\right) \overline{\tau} \sigma^j,
    $$
    which holds for all $x \in B_1$. Thus, choosing properly $\overline \tau$ and the radius $r_0$, we can conclude that for all $y_0\in B_r(x_0)\cap\partial F\cap\partial \Omega$, the blow-up of $\pa F$ is unique and precisely it is the half-plane $\{x\cdot \nu(y_0)= 0\}$ with the property that $\nu(y_0)\perp\nu_{\Omega}(y_0)$. Moreover, just by interpolation, its convergence is polynomial, i.e. there are constants $\alpha_1>0$ and $C_1>0$ such that the oscillation of $\nu(y_0)$ is controlled as 
\begin{equation}
    \label{eq:oscillazione_normali}
    |\nu(y_0)-\nu(z_0)|\le C_1|y_0-z_0|^{\alpha_1}\qquad\text{for every}\qquad y_0,z_0\in \partial\Omega\cap\partial F\cap B_{\sfrac{r}{2}}(x_0),
\end{equation} 
and for every $s<\sfrac{r}{2}$ it holds
\begin{equation}
    \label{eq:osc_insieme}
    \left\{x\,:\,(x-y_0) \cdot  \nu(y_0) - C_1s^{1+{\alpha_1}}>0\right\}\subset F \subset \left\{x\,:\,(x-y_0) \cdot \nu(y_0) + C_1s^{1+{\alpha_1}}>0\right\}\quad\hbox{ in}\quad\Omega\cap B_s(y_0),
\end{equation}
which is true for every $y_0\in \partial\Omega\cap\partial F\cap B_{\sfrac{r}{2}}(x_0)$.
\\
\\
{\em Step 2. Epsilon-regularity.}
 Combining \eqref{eq:oscillazione_normali} and \eqref{eq:osc_insieme}, we get that $\partial\Omega\cap\partial F$ is a $C^{1,\alpha_1}$ graph, in the manifold $\partial\Omega$, in the direction $\nu$. To complete the proof, we need to show that $\partial F\cap\Omega$ touches the boundary $\partial\Omega\cap\partial F$ in a $C^{1,\alpha}$ way for some $\alpha >0$. Indeed, given an interior point $y_I\in B_{\sfrac{r}{2}}(x_0)\cap\Omega\cap\partial F$ and setting $\rho:=|y_B-y_I|$, where $y_B$ is the projection of $y_I$ on $\partial\Omega\cap\partial F$, we have that in $B_{2\rho}(y_B)$ the boundary $\partial F$ is $C^{1+\alpha_1}(B_{2\rho}(y_B))$-flat and by the boundary monotonicity formula it holds true the following estimate
 \begin{equation}
    \label{eq:stimadensita2}
    \H^{d-1}(\partial F\cap B_{2\rho}(y_B)\cap\Omega)\le \frac{(1+\overline\tau)d\omega_d(2\rho)^{d-1}}{2}.
 \end{equation}
Since $\partial F$ is flat in $B_{2\rho}(y_B)$, there are $\overline C>0$ and $\gamma>0$, depending on the $C^{1,\sigma}$ regularity of the fixed boundary $\partial\Omega$ and the constants $\alpha_1$ and $C_1$ from \eqref{eq:oscillazione_normali} and \eqref{eq:osc_insieme},  such that 
\begin{equation}
\label{eq:stimadensita}
    \H^{d-1}(B_{2\rho}(y_B)\setminus B_{\sfrac{\rho}{2}}(y_I)\cap\partial F\cap\Omega)\ge (1-\overline C\rho^\gamma)\H^{d-1}(B_{2\rho}(y_B)\setminus B_{\sfrac{\rho}{2}}(y_I)\cap \boldsymbol{\pi}\cap\Omega),
\end{equation}
where $\boldsymbol{\pi}$ is any plane passing through $y_B$ and $y_I$.
As a consequence of the density estimates \eqref{eq:stimadensita2} and \eqref{eq:stimadensita}, in $B_{\sfrac{\rho}{2}}(y_I)\subset\Omega$, the density of $\partial F$ is controlled by
\[
\H^{d-1}(\partial F\cap B_{\sfrac{\rho}{2}}(y_I))\le (1+\overline\tau+\overline C\rho^\gamma)d\omega_d\left(\frac{\rho}{2}\right)^{d-1}.
\] 
Thus, in $B_{\sfrac{\rho}{2}}(y_I)$, we can apply the classical Allard's theorem (\cite{Allard_interior,de2023short}). In particular, $|\nu(y_I)-\nu(y_B)|\leq C\rho^\alpha$, for some $\alpha >0$ with $\alpha = \inf\{\alpha_1, \gamma\}$ and $C = \inf\{C_1, \overline C\}$. This is exactly the thesis: $\partial F\cap\Omega$ up to the boundary $\partial\Omega$ is a  $C^{1,\alpha}$ graph over $\bp$.
 \end{proof}

\section{Comparison principle for classical solutions}
\label{sec:graph_ordered}

Once \cref{th:eps_reg} holds, the final step to conclude the proof is to show that the obtained regular graphs up to the boundary of $\Omega$ are ordered with respect to the volume and symmetric with respect to $\bp$. 
The main result of this section is the uniform curvature 
estimate up to the boundary \cref{prop:C2alpha_graph}, which will allow us to use comparison arguments, \cref{th:ordered_graphs} and \cref{t:symmetry}. 

\begin{proposition}
\label{prop:C2alpha_graph}
    Let $\pa S_+$ and $\pa S_-$ be as in \cref{t:reg-C1alpha}. Then, there are compact domains $D_\pm$ with $C^\infty$ boundary such that $\partial S_+$ and $\partial S_-$ are $C^{\infty}$ graphs up to the boundary respectively over $D_+$ and $D_-$. Precisely, there exists $\lambda>0$ such that the functions $u^\pm$ satisfy
\begin{equation}
    \label{e:minimal_surfaces_equation}
      \left\{
  \begin{aligned}
  &-{\rm div}\, \left(\frac{\nabla u^{-}}{\sqrt{1 + \abs{\nabla u^{-}}^2}}\right) = -\lambda &&\hbox{ on } D_-\,,\\
  &\left( \nabla u^- , -1 \right) \cdot \nu_W = 0  &&\hbox{ on } \partial D_-\,,
  \end{aligned}
  \right.
  \qquad \quad
  \left\{
  \begin{aligned}
  &-{\rm div}\, \left(\frac{\nabla u^{+}}{\sqrt{1 + \abs{\nabla u^{+}}^2}}\right) = \lambda &&\hbox{ on } D_+\,,\\
  &\left( -\nabla u^+ , 1 \right) \cdot \nu_W = 0 &&\hbox{ on } \partial D_+\,,
  \end{aligned}
  \right.
\end{equation}
where $\nu_W$ is the outer unit normal vector field to $\pa W$. Moreover, fixed $k\ge 2$ and $\omega>0$ there is $\overline\eps=\overline\eps(k,\omega)>0$ such that for every $\eps\le\overline\eps$ it holds that $\|u_\pm\|_{C^k(D_\pm)}\le \omega$. 
\end{proposition}
\begin{proof}
We divide the proof in some steps.
\\
\\
{\it Step 1. Change of variables.} Let $B''$ be a ball in $\R^{d-2}$ and let 
$$\gamma:B''\to\R^{d-1}$$
be a smooth parametrization of $\Gamma$. We define the map
$$\phi:= \left\{\begin{aligned}
    B''\times\R&\to\R^{d-1}\\
\phi(s,t)&\mapsto\gamma(s)+t\nu_\gamma(s);
\end{aligned}
\right. \qquad   \qquad 
\Phi: = \left\{\begin{aligned}
    B''\times\R\times\R&\to\R^d\\
\Phi(s,t,x_d)&\mapsto(\phi(s,t),x_d),
\end{aligned}
\right. 
$$
where $\nu_\gamma(s)$ is the normal to $\Gamma$ at the point $\gamma(s)$ and $s=(s_1,\dots,s_{d-2})\in B''$. We notice that locally the tubular neighborhood $I_\delta(\Gamma)$ in terms of these new coordinates is given precisely by:
$$
\begin{aligned}
    W&=I_\delta(\Gamma)\cap \Phi(B''\times (-\delta_0,\delta_0)\times (-\delta_0,\delta_0))\\
    &= \Big\{\big(\phi(s,t),x_d\big)\ :\ s\in B'',\ t\in (-\delta_0,\delta_0),\,\ x_d\in (-\delta_0,\delta_0), \,\ t^2+x_d^2\le \delta^2\Big\}.
\end{aligned}
$$
We next define the function 
$$
v:= \left\{\begin{aligned}
    \phi^{-1}(D_+)&\to \R\\\
v(s,t)&\mapsto u_+(\phi(s,t)),
\end{aligned}
\right. $$
and we compute 
\begin{align}
\nonumber
\int_{D_+}\sqrt{1+|\nabla u_+|^2}\,dx'&=\int_{\phi^{-1}(D_+)}\sqrt{1+\nabla v(s,t)\cdot \tens{A}(s,t) \nabla v(s,t)}\,\Big(\det \tens{A}(s,t)\Big)^{-1}\,ds\,dt\,,\\
\label{eq:funzionale-grafici-F}
&=\int_{\phi^{-1}(D_+)}F(s,t,\nabla v(s,t))\,ds\,dt\,,
\end{align}
where the symmetric $(d-1)\times (d-1)$ matrix $\tens{A}$ and the function $F:B''\times \R\times \R^{d-1}\to\R$ are given by
$$\tens{A}(s,t):=(\tens{D}\phi)^{-T}(\tens{D}\phi)^{-1}\qquad\text{and}\qquad F(s,t,p):=\sqrt{1+p\cdot \tens{A}(s,t) p\,}\,\Big(\det \tens{A}(s,t)\Big)^{-1}.$$
Since $\phi(s,t)$ is linear in $t$, we have that $\tens{A}(s,t)$ is analytic in $t$ and 
$$\tens{A}(s,t)=\sum_{m=0}^{+\infty}\tens{A}_m(s)t^m \,,$$
where $\tens{A}_m$ are smooth symmetric $(d-1)$-matrices such that $\|\tens{A}_m\|_{L^\infty}\le C^m$ for some fixed constant $C$ depending on $\gamma$. 
\\
\\
{\it Step 2. Hodograph transform.} We now perform an hodograph-type transform that allows to transform the Robin-type boundary condition in \eqref{e:minimal_surfaces_equation} into a homogeneous Neumann condition.  
Precisely,  we define the function $\Theta: B''\times [0,\delta_0)\to\R$ through the identity:
$$
v\Big(s,(\delta + \rho) \cos\Theta(s, \rho)\Big) =(\delta + \rho) \sin\Theta(s, \rho). 
$$
The existence of such a function $\Theta$ follows from the $C^{1,\alpha}$ regularity of $v$ and the Implicit Function Theorem. Moreover, by \cref{t:reg-C1alpha}, for every $\omega>0$, the following estimate holds whenever $\eps$ is chosen small enough:
\begin{equation}\label{e:smallness-norms-Theta}
\|\Theta\|_{C^{1,\alpha}(B''\times[0,\delta_0))}\le \omega.
\end{equation}
Next, we compute the energy in terms of $\Theta$. By differentiating in $s$ and in $\rho$, we obtain the system 
$$\begin{cases}
\nabla_sv\big(s,(\delta + \rho) \cos\Theta\big)-(\delta+\rho)\sin\Theta\,\partial_tv\big(s,(\delta + \rho) \cos\Theta\big)=(\delta+\rho)\nabla_s\Theta\cos\Theta\\
\partial_tv\big(s,(\delta + \rho) \cos\Theta\big)\,\Big(\cos\Theta-(\delta + \rho)\partial_\rho\Theta\sin\Theta\Big)=\sin\Theta+(\delta+\rho)\partial_\rho\Theta\cos\Theta,
\end{cases}$$
where $\Theta$, $\nabla_s\Theta$ and $\partial_\rho\Theta$ are computed in $(s,\rho).$ This gives
\begin{equation}
\label{e:nablav}
    \begin{aligned}
    \nabla v&(s, (\delta + \rho) \cos\Theta) \\
    &= \frac{1}{\cos \Theta -(\delta + \rho) \pa_\rho \Theta \sin \Theta} \begin{pmatrix}
        &(\delta + \rho) \nabla_s \Theta\\
        &\sin \Theta + (\delta + \rho)\pa_\rho \Theta \cos \Theta
    \end{pmatrix}\\
    &= \frac{1}{\cos \Theta -(\delta + \rho)\pa_\rho \Theta  \sin \Theta}\underbrace{\begin{pmatrix}
        &(\delta + \rho) \nabla_s \Theta\\
        &(\delta + \rho) \pa_\rho \Theta \cos \Theta
    \end{pmatrix}}_{=:q} + \frac{1}{\cos \Theta -(\delta + \rho) \sin \Theta \,\pa_\rho \Theta}\, \sin\Theta\,e_{d-1}\\\
    &= \frac{1}{\cos \Theta -(\delta + \rho)\pa_\rho \Theta  \sin \Theta}\left(q + \sin \Theta\, e_{d-1}\right).
\end{aligned}
\end{equation}
By applying the change of variables
$$(s,t)=(s, (\delta + \rho) \cos\Theta(s,t)),$$
in the energy functional \eqref{eq:funzionale-grafici-F}, we get that the energy can be written in terms of the function $\Theta=\Theta(s,\rho)$ as follows
\begin{equation}
\label{eq:cambiodivariabile_funzionale}
    \begin{aligned}
    &\int_{\B''\times [0, \delta_0)} F\big(s, (\delta + \rho) \cos\Theta, \nabla v(s, (\delta + \rho) \cos\Theta)\big)\left(\cos\Theta - (\delta + \rho) \pa_\rho \Theta \sin\Theta\right)\, dsd\rho\\
    =&\int_{\B''\times [0, \delta_0)} \sqrt{1 + \nabla v \cdot \tens{A}(s,(\rho + \delta) \cos\Theta)  \nabla v\,} \,\,({\rm det}\,\tens{A}(s,(\rho + \delta) \cos\Theta))^{-\sfrac{1}{2}}\left(\cos\Theta - (\delta + \rho) \pa_\rho \Theta \sin\Theta\right) \, dsd\rho,
\end{aligned}
\end{equation}
where $\nabla v$ is given by \eqref{e:nablav} and is computed in $(s, (\delta + \rho) \cos\Theta)$.
\\
\\
{\it Step 3. The function $\Theta$ is $C^{2,\alpha}$.} We set 
\begin{equation}
    \label{e:matrixA}
    \tens{A}(s,t):= \tens{I} + \tens{B}(s,t)\quad\text{where}\quad \tens{B}(s,t)=\tens{A}_0(s)-\tens{I}+\sum_{m=1}^{+\infty}t^m\tens{A}_m(s).
\end{equation}
Thus, using \eqref{e:nablav} and \eqref{e:matrixA},  \eqref{eq:cambiodivariabile_funzionale} simplifies into
$$
\begin{aligned}
    &\sqrt{1 + \nabla v \cdot \tens{A}(s,(\rho + \delta) \cos\Theta)  \nabla v\,}\left(\cos\Theta - (\delta + \rho) \pa_\rho \Theta \sin\Theta\right)\\
    =&\sqrt{\left(\cos\Theta - (\delta + \rho) \pa_\rho \Theta \sin\Theta\right)^2 + q\cdot \tens{A} q + 2 q\cdot  (\tens{A} e_{d-1}) \sin \Theta+  \sin^2 \Theta (e_{d-1}\cdot \tens{A}e_{d-1})}\\
    =& \sqrt{1+\left((\delta + \rho) \pa_\rho \Theta \sin\Theta \right)^2 + |q|^2 + q\cdot  \tens{B} q + 2 q\cdot (\tens{B} e_{d-1}) \sin\Theta+ \sin^2\Theta (e_{d-1}\cdot \tens{B}e_{d-1}) }\\
    =& \sqrt{1+(\delta + \rho)^2 \abs{\nabla \Theta}^2 + q\cdot \tens{B} q + 2 q\cdot (\tens{B} e_{d-1}) \sin\Theta + \sin^2\Theta (e_{d-1}\cdot \tens{B}e_{d-1})},
\end{aligned}
$$
so we get that
\begin{align}
\nonumber
    \int_{\B''\times [0, \delta_0)}& \sqrt{1 + \nabla v \cdot \tens{A}(s,(\rho + \delta) \cos\Theta)  \nabla v\,} \,\,{\rm det}(\tens{A}(s,(\rho + \delta) \cos\Theta))^{-\sfrac{1}{2}}\left(\cos\Theta - (\delta + \rho) \pa_\rho \Theta \sin\Theta\right) \, dsd\rho\\
    \nonumber
    &=\int_{B''\times [0, \delta_0)} \sqrt{1+(\delta + \rho)^2 \abs{\nabla \Theta}^2 + q\cdot \tens{B} q + 2 q \cdot(\tens{B} e_{d-1}) \sin\Theta + \sin^2\Theta (e_{d-1}\cdot \tens{B}e_{d-1})}\,\,({\rm det}\,\tens{A})^{-\sfrac{1}{2}}\, dsd\rho\\
    \label{e:funGdiGiulia}
    &=: \int_{B''\times [0, \delta_0)} \sqrt{1 + M(s, \rho, \Theta, \nabla \Theta)} N(s, \rho,\Theta)\,dsd\rho.
\end{align}
Computing the outer variation of the energy functional \eqref{e:funGdiGiulia} we get that $\Theta$ satisfies an equation of the form
\begin{equation}
\label{eq:EL_bordo_dritto}
    \left\{
\begin{aligned}
    &{\rm div}\left(\frac{\nabla_pM}{2 \sqrt{1 + M}} \, N\right) = \frac{\pa_\Theta M}{2 \sqrt{1 + M}}\, N + \sqrt{1 + M}\, \pa_\Theta N \quad&&\hbox{ in } B''\times (0,\delta_0),\\
    & e_{d-1}\cdot \nabla_pM= 0\quad &&\hbox{ on } B''\times\{0\},
\end{aligned}
\right.
\end{equation}
where $\nabla_pM(s,\rho,\Theta,p)$ is the gradient of $M$ with respect to the last variable.

We notice that, if $\tens{T}=\tens{T}(\Theta)$ is the $(d-1)$-diagonal matrix 
$$\tens{T}(\Theta)={\rm diag}(1,\dots,1,\cos\Theta),$$ 
we can write the vector $q$ as 
$$q:=(\delta+\rho)\tens{T}(\Theta)\nabla\Theta,$$
so we get 
\begin{align*}
q\cdot \tens{B} q + 2 q \cdot(\tens{B} e_{d-1}) \sin\Theta&=(\delta+\rho)^2 \nabla \Theta\cdot \left(\tens{T} \tens{B}\tens{T}\right) \nabla\Theta  + 2 \nabla \Theta\cdot(\tens{T}\tens{B} e_{d-1}) \sin\Theta
\end{align*}
where $\tens{B}=\tens{B}(s,(\delta+\rho)\cos\Theta)$ and $\tens{T}=\tens{T}(\Theta)$. In particular, this implies that 
$$\nabla_pM(s,\rho,\Theta,\nabla \Theta)=2(\delta+\rho)^2\Big(\tens{I}+\tens{T} \tens{B}\tens{T}\Big)\nabla \Theta+2 (\delta+\rho) \sin\Theta\,\tens{T}\tens{B} e_{d-1},$$
so we can write \eqref{eq:EL_bordo_dritto} in the form 
\begin{equation*}
    \left\{
\begin{aligned}
    &{\rm div}\left(\frac{N}{\sqrt{1 + M}}{\big(\tens{I}+\tens{T}\tens{B}\tens{T}\big) \nabla \Theta} \right) = \frac{G}{(\delta+\rho)^2} \quad&&\hbox{ in } B''\times (0,\delta_0),\\
    & e_{d-1}\cdot (\tens{I}+\tens{T} \tens{B}\tens{T})\nabla \Theta= 0\quad &&\hbox{ on } B''\times\{0\},
\end{aligned}
\right.
\end{equation*}
where the right-hand side $G=G(s,\rho,\Theta,\nabla\Theta)$ can be estimated as follows:
$$
|G|=\left|\frac{\pa_\Theta M}{2 \sqrt{1 + M}}\, N + \sqrt{1 + M}\, \pa_\Theta N + 2(\delta + \rho) \,{\rm div} \big(\sin \Theta\, \tens{T}\tens{B} e_{d-1}\big)\right| \leq c_1\abs{\Theta} + c_2\abs{\nabla \Theta}.
$$
Finally, using \eqref{e:smallness-norms-Theta} and the classical $C^{k,\alpha}$ Schauder estimates, we get the claim.  
\end{proof}

Thus, we are in position to construct the foliation.

\begin{theorem}
    \label{th:ordered_graphs}
    Let $\pa S_+$ be as in \cref{prop:C2alpha_graph}. Then, as functions of $\eps$, the sets $D_+^\e$, the functions $u^\e_+$ and their mean curvatures are strictly increasing.
\end{theorem}
\begin{proof}
In the following, we drop the subscript $_+$ to make the notation easier. Let $(u_1,D_1, \eps_1, \lambda_1)$ and $(u_2, D_2, \eps_2, \lambda_2)$ be the two quadruples of functions $u_{1,2}:D_{1,2}\to\R_+$ parametrizing the upper boundaries $\partial S_+^{1,2}$ with mean curvatures $\lambda_{1,2}>0$ respectively, associated to the two different volumes $\eps_1$ and $\eps_2$ with $\eps_1 < \eps_2$. Consider the family of functions 
$$-t+u_2:D_2\to\R\ :\quad  t\ge 0.$$
The volume of the set 
$$E^2_t:=\left\{(x',x_d)\in\Omega\ :\ x'\in D_2,\ 0<x_d<u_2(x')-t\right\}$$
is decreasing in $t$, is equal to $0$ when $t$ is large, and is $\eps_2$ when $t=0$. Moreover, when $t\ge 0$
$$\overline\Omega\cap\{x_d>0\}\cap \partial E^2_t=\left\{(x',x_d)\in\overline\Omega\ :\ x'\in D_2,\ x_d=u_2(x')-t\right\}=\overline\Omega\cap\{x_d>0\}\cap(-te_d+\partial S_+^2).$$
Since $\eps_1<\eps_2$ there is $\overline t>0$ such that $\overline\Omega\cap\{x_d>0\}\cap(-\overline t e_d+\partial S_+^2)$ touches $\overline\Omega\cap\{x_d>0\}\cap\partial S_+^1$ from below at some point $P$.
\\
\\
{\it Step 1. We claim that $P \notin \partial\Omega$.} Let suppose by contradiction that there exists $P \in \pa \Omega$ such that $P \in \pa S_+^1\cap(-\overline t e_d+\partial S_+^2)$. Consider the plane 
$$J:={\rm span}\{\nu_\Omega(P),e_d\}.$$
Then:
\begin{itemize}
\item $\partial\Omega\cap J$ is a circle of radius $\delta$ with center $C'\in\Gamma\subset\bp$;
\item for any $x_0\in \partial \Omega \cap J$ the normal $\nu_{\Omega}(x_0)$ lies in $J$; precisely $\nu_{\Omega}(x_0)=x_0-C'$;
\item the sections $\Omega \cap \partial S_+^1\cap J$ and $\Omega\cap \partial S_+^2\cap J$ are smooth curves meeting $\partial\Omega$ orthogonally.
\end{itemize}
\begin{figure}[htbp]
		\centering
		\begin{tikzpicture}[rotate=0, scale=1.2]
\draw[thick, blue, name path=n] (-3.1,0.45) arc [start angle=150, end angle = 85,x radius = 60mm, y radius = 60mm];
\draw[thick, red, name path=n] (-3.4,1.20) arc [start angle=130, end angle = 74,x radius = 65mm, y radius = 65mm];
\draw[thick, red, densely dashed, name path=n] (-3.4,-2.10) arc [start angle=130, end angle = 74,x radius = 65mm, y radius = 65mm];
\draw[thick, densely dashed, red, name path=n] (-3.4,0.25) arc [start angle=130, end angle = 74,x radius = 65mm, y radius = 65mm];
\draw[dashed] (-3.4,-2.5) -- (-3.4,3.10);
\draw[very thick, color=black, name path=a] (-5,0) circle [radius=2];
\draw [color = black] node at (-6,-1) {$W\cap J$};
\draw [color = black] node at (-4.8,2.2) {$\partial\Omega\cap J$};
\draw [color = black] node at (-2.67,0.275) {$P$};
\draw [color = black] node at (-3.8,1.1) {$Q$};
\draw [color = red] node at (-2.6,2.5) {$T_2(Q)$};
\draw [color = blue] node at (-1.7,1.575) {$T_1(P)$};
\draw [color = red] node at (-1.2,0.675) {$T_2(P+\overline t e_d)$};
\draw [color = black] node at (1,3.6) {$\partial S_+^1$};
\draw [color = black] node at (1,2.9) {$\partial S_+^2$};
\draw [color = black] node at (1,2) {$-\overline te_d+\partial S_+^2$};
\draw [color = black] node at (1,-0.3) {$-te_d+\partial S_+^2$};
\draw [color = red, very thick, ->] (-3.1,0.48) -- (-2.17,1.10);
\draw [color = red, very thick, ->] (-3.4,1.20) -- (-2.5,2.00);
\draw [color = blue, very thick, ->] (-3.1,0.48) -- (-2.45,1.40);
\filldraw [black] (-3.07,0.475) circle (2pt);
\filldraw [black] (-3.4,1.20) circle (2pt);
\end{tikzpicture}
	\caption{Graphical representation of the procedure used in Step 1 of \cref{th:ordered_graphs} to show that normal graphs are ordered with respect to the volume $\eps$.}
	\label{fig:normal_graphs_ordered}
\end{figure}
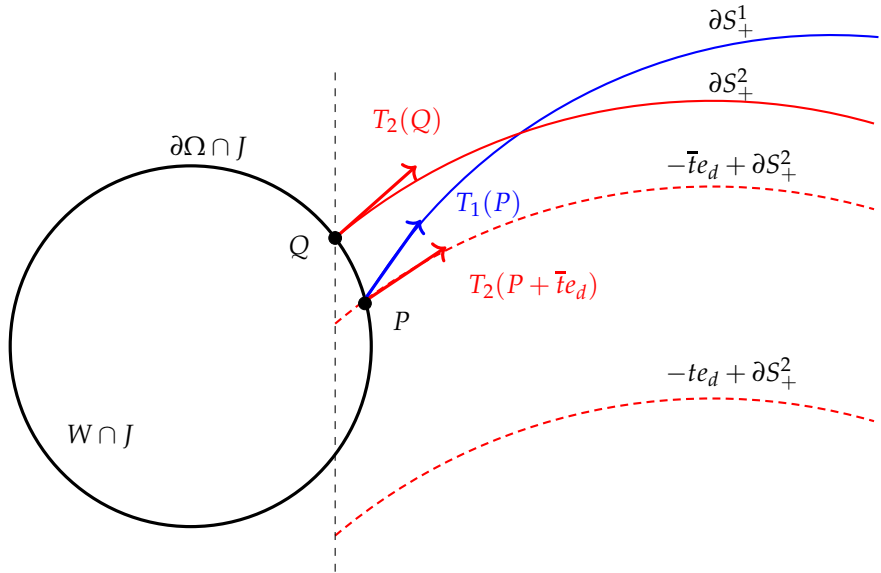
We define $Q$ as the end point of the curve $\Omega \cap \partial S_+^2\cap J$ lying on $\partial\Omega$. Moreover, let $T_2(Q)$ be the tangent vector to the curve $\Omega \cap \partial S_+^2\cap J$ at the end-point  $Q \in \pa \Omega$. Similarly, let $T_1(P)$ be the tangent vector to the curve $\Omega \cap \partial S_+^1\cap J$ at the end point $P\in\partial\Omega$.
Moreover, let $T_2(P+\overline t e_d)$ be the tangent vector to $\Omega \cap \partial S_+^2\cap J$ at the interior point $P+\overline t e_d\in\Omega$; notice that $T_2(P+\overline t e_d)$ is also the tangent vector to $\Omega \cap (-\overline te_d+\partial S_+^2)\cap J$ at the point $P\in\partial\Omega$.  
For a graphical representation we refer to \cref{fig:normal_graphs_ordered}. In particular, we have 
$$T_1(P)=-\nu_\Omega(P)\ ,\qquad T_2(Q)=-\nu_\Omega(Q)\qquad\text{and}\qquad\nu_\Omega(P) - \nu_\Omega(Q) = \frac{P-Q}{\delta} .$$
Thus, in the plane, $J$ a consequence of \cref{t:reg-C1alpha} is that there exists a positive small constant $C$ depending on the volume $\eps$ and on the flatness $\tau$ such that
\begin{equation}
    \label{e:flatness_controllo_normali}
    \abs{T_2(Q) - T_2(\overline t e_d+ P)} \leq C\abs{{\bf p}(Q) - {\bf p}(\overline t e_d+ P)}= C\abs{{\bf p}(Q) - {\bf p}(P)} \leq C \abs{P-Q},
\end{equation}
where ${\bf p}$ is the orthogonal projection on $\bp$. Since $(-\overline te_d+ \partial S_+^2)$ touches $\partial S_+^1$ from below, we have
$$
\begin{aligned}
    T_2(\overline t e_d+ P) \cdot e_d \leq T_1(P) \cdot e_d
    &= -\nu_\Omega(Q) \cdot e_d + ( T_1(P) + \nu_\Omega(Q)) \cdot e_d\\
        &= -\nu_\Omega(Q) \cdot e_d + ( -\nu_\Omega(P) + \nu_\Omega(Q)) \cdot e_d\\
        &= -\nu_\Omega(Q) \cdot e_d  + \frac{1}{\delta} (Q- P) \cdot e_d\\
        &= T_2(Q) \cdot e_d  + \frac{1}{\delta} (Q- P) \cdot e_d, 
\end{aligned}
$$
which, using \eqref{e:flatness_controllo_normali}, gives
$$
\begin{aligned}
   C (Q-P) \cdot e_d \geq \Big(-T_2(\overline t e_d+ P) + T_2(Q)\Big) \cdot e_d  &\geq\frac{1}{\delta} (Q- P) \cdot e_d,
\end{aligned}
$$
leading to a contradiction since $C$ is small and $\pa \Omega \cap J$ is close to be a vertical boundary.
\\
\\
{\it Step 2. We claim that $\lambda_1 < \lambda_2$.} As a consequence of Step 1, we have that $P\in\Omega$. Then, by the maximum principle $\lambda_2\ge \lambda_1$. Moreover, if $\lambda_2=\lambda_1$, then $\partial S_+^1$ and $-\overline t e_d+\partial S_+^2$ must coincide in $\Omega$, so there is a second contact point on $\partial\Omega$, which is impossible by Step 1. 
\\
\\
{\it Step 3. We claim that $u_1 <u_2$.} Consider the family of functions 
$$-t+u_1:D_1\to\R\ :\ t\ge 0,$$
and define the set 
$$E^1_t:=\left\{(x',x_d)\in\Omega\ :\ x'\in D_1,\ 0<x_d<u_1(x')-t\right\};$$
the volume of $E^1_t$ is decreasing in $t$, is equal to $0$ when $t$ is large, and is $\eps_1$ when $t=0$. Moreover, when $t\ge 0$
$$\overline\Omega\cap\{x_d>0\}\cap \partial E^1_t=\left\{(x',x_d)\in\overline\Omega\ :\ x'\in D_1,\ x_d=u_1(x')-t\right\}=\overline\Omega\cap\{x_d>0\}\cap(-te_d+\partial S_+^1).$$
Suppose by contradiction that there is $\widetilde t>0$ such that $\overline\Omega\cap\{x_d>0\}\cap(-\widetilde t e_d+\partial S_+^1)$ touches $\overline\Omega\cap\{x_d>0\}\cap\partial S_+^2$ from below at some point $\widetilde P$. By the same argument of Step 1, we know that $\widetilde P\notin \partial\Omega$. On the other hand, by Step 2 and the maximum principle, $\widetilde P\notin\Omega$, which concludes the proof.
\end{proof}

Finally, once we have shown that normal graphs are ordered with respect to the volume, we are in position to close the proof of \cref{th:sogno_finale} providing the symmetry with respect to the plane $\bp$.

\begin{theorem}
\label{t:symmetry}
    Let $\pa S_+$ and $\pa S_-$ be as in \cref{prop:C2alpha_graph}. Then $D_+=D_-$ and $u_+=-u_-$.
\end{theorem}

\begin{proof}
Let $\lambda>0$ be the mean curvature of $\partial S_+$. 
In order to use the same notion as the one of the proof of \cref{th:ordered_graphs}, we set: 
$$(u_1,D_1, \eps_1, \lambda_1):=(u_+,D_+,\eps_+,\lambda)\qquad\text{and}\qquad(u_2, D_2, \eps_2, \lambda_2):=(-u_-,D_-,\eps_-,\lambda),$$ 
where 
$$\eps_+:={\rm Vol}(E\cap\{x_d>0\})\qquad\text{and}\qquad \eps_-:={\rm Vol}(E\cap\{x_d<0\}).$$
If $\eps_+=\eps_-$, then by \cref{th:ordered_graphs}, $u_1=u_2$. Suppose that $\eps_1\neq \eps_2$. Without loss of generality $\eps_2>\eps_1$, so by \cref{th:ordered_graphs}, $D_1\subset D_2$ and $u_1\le u_2$ on $D_1$. Similar as in \cref{th:ordered_graphs}, we consider the family of functions  
$$u_2^t:=-t+u_2:D_2\to\R\ :\ t\ge 0,$$
and sets 
$$E^2_t:=\left\{(x',x_d)\in\Omega\ :\ x'\in D_2,\ 0<x_d<u_2(x')-t\right\}.$$
Since $\eps_2>\eps_1$, there exists $t>0$ such that the graph of $u_2^{t}$ touches the graph of $u_1$ from below at some $P\in\overline\Omega$. Applying the same argument as in Step 1 of \cref{th:ordered_graphs}, $P\notin\pa \Omega$. Then $P\in\Omega$ and by the fact that the two curvatures are both equal to $\lambda$, the maximum principle immediately implies that $u_1=-t+u_2$, which concludes the proof. 
\end{proof}

\begin{appendices}
\section{Homotopic Plateau admits a unique solution}
\label{app:Plateau_unica_sol}

In this Appendix, we show the existence and the uniqueness of a minimizer for the Homotopic Plateau problem 
\begin{equation*}
\ell(W,\C):=
\inf\Big\{\H^{d-1}(S)\ :\ S\in\Sc(W,\C)\Big\}\,,
\end{equation*}
where 
\begin{equation*}
\Sc(W,\C):=\left\{ S\subset \Omega\, \colon \, S \hbox{ is relatively closed in } \Omega \hbox{ and } S \hbox{ is } \C-\hbox{spanning } W \right\}\,,
\end{equation*}
which is exactly the same as the one set in \eqref{e:hp-intro}. Precisely, we show the result $W:=I_\delta(\Gamma)$ the $\delta$-tubular neighbourhood on $\Gamma$ ($\delta >0$ sufficiently small), which is a smooth embedding of $(d-2)$-dimensional sphere $\mathbb{S}^{d-2}$ in the hyperplane $\bp$.

\medskip

First of all, concerning the existence and the uniqueness for the Homotopic Plateau problem on the curve $\Gamma$ the result is immediate, Indeed, we notice that if $\C \subset \mathscr{C}_\Gamma$ is such that Plateau's problem $\ell(\Gamma,\C)$ is finite, then necessarily none of the loops in $\C$ is homotopically trivial. Since $\pi_1(\R^d\setminus \Gamma) = \Z$, it is immediate that there exists a unique spanning class $\C_0 \subset \mathscr{C}_\Gamma$ such that $\ell(\Gamma,\C_0) < \infty$. Notice also that, evidently, there exists a unique minimizer of $\ell(\Gamma,\C_0)$ in this case: the disc $D:=B_1(0) \cap \bp$. 

Then, we observe that the above conclusions remain valid when $\Gamma$ is replaced by a tubular neighbourhood $W = I_\delta (\Gamma)$ when $\delta$ is suitably small. We show this fact in the following lemma.  

\begin{lemma}
\label{lemma:de_0}
 There exists $\de_0>0$ such that if $\de\in(0,\de_0]$, then there exists a unique spanning class $\C$ for $W=I_\de(\Gamma)$ such that $\ell(W,\C) < \infty$. The homotopic Plateau's problem $\ell(W,\C)$ admits a unique solution given by $S=B_{1-\de}(0)\cap\bp$.
\end{lemma}
\begin{proof}
The first part follows from the fact that, for $\delta$ sufficiently small, $\pi_1(\R^d \setminus W) = \Z$. For the second part, let $S' \in \mathcal{S}(W,\C)$ be a minimizer for $\ell$, and notice that we can assume that $S' \subset B_2$. Letting ${\bf p}$ denote the orthogonal projection $\R^d \to \bp$, we claim that necessarily
\begin{equation}
\label{eq:projection}
{\bf p}(S') \supset S\,.   
\end{equation}
Indeed, for every $z \in S$ the line ${\bf p}^{-1}(\{z\})$ must intersect $S'$, for otherwise any loop $\gamma \in \C$ which agrees with the line ${\bf p}^{-1}(\{z\})$ in $B_2$ would have $\gamma \cap S' = \emptyset$, which contradicts the hypotheses on $S'$. The inclusion in \eqref{eq:projection} implies immediately that
\[
\H^{d-1}(S) \leq \H^{d-1}({\bf p}(S')) \leq \H^{d-1}(S') \leq \H^{d-1}(S)\,.
\]
In particular, the above inequalities are all equalities. Hence, the coarea formula implies that, for $\mathcal H^{d-1}$-a.e. $z \in {\bf p}(S')$, ${\bf p}^{-1}(\{z\}) \cap S'$ is a singleton, and that for $\mathcal H^{d-1}$-a.e. $p \in S'$ the approximate tangent plane ${\rm Tan}(S',p)$ is $\bp$. It follows that $S'$ contains a vertical translate of $S$; since $S'$ is $\C$-spanning and $\H^{d-1}(S') = \H^{d-1}(S)$, concluding that $S'=S$.
\end{proof}

\end{appendices}

\medskip

\section*{Acknowledgements}
We warmly thank Francesco Maggi for his interest in our work on the {\it non-collapsing conjecture}. 

\medskip

\noindent GB and BV are supported by the European Research Council (ERC), under the European Union's Horizon 2020 research and innovation program, through the project ERC VAREG - {\em Variational approach to the regularity of the free boundaries} (grant agreement No.\,853404). GB and BV acknowledge the MIUR Excellence Department Project awarded to the Department of Mathematics, University of Pisa, CUP I57G22000700001. GB and SS are supported by Gruppo Nazionale per l'Analisi Matematica, la Probabilit\`a e le loro Applicazioni (GNAMPA) of Istituto Nazionale di Alta Matematica (INdAM) though the INdAM-GNAMPA project 2024 CUP E53C23001670001. SS acknowledges the MIUR Excellence Department Project awarded to the Department of Mathematics, University of Milan, CUP G43C22004580005 and the support from the project PRIN 2022PJ9EFL "\textit{Geometric Measure Theory: Structure of Singular Measures, Regularity Theory and Applications in the Calculus of Variations}", funded by the European Union under NextGenerationEU and by the Italian Ministry of University and Research. BV acknowledges support from the projects PRA 2022 14 GeoDom (PRA 2022 - Università di Pisa) and MUR-PRIN ``NO3'' (No. 2022R537CS).  

\printbibliography

\end{document}

%% file: defs.tex
\theoremstyle{plain} 

\DeclareMathOperator{\diver}{div}

\DeclareMathOperator*{\argmin}{arg\,min}

\usepackage{mathtools}
\usepackage{pgfplots}
\pgfplotsset{/pgf/number format/use comma,compat=newest}

\renewcommand\epsilon{\varepsilon}
\newcommand{\R}{\mathbb{R}}
\newcommand{\N}{\mathbb{N}}

\newcommand{\tens}[1]{\mathsf{#1}}

\newcommand{\abs}[1]{\left\lvert#1\right\rvert}
\newcommand{\norm}[1]{\left\lVert#1\right\rVert}